\documentclass[12pt]{amsart}
\usepackage{amsmath,amssymb,amsthm,amsfonts,setspace,hyperref,dsfont}

\newcommand{\NN}{\mathbb{N}}
\newcommand{\RR}{\mathbb{R}}
\newcommand{\CC}{\mathbb{C}}
\newcommand{\QQ}{\mathbb{Q}}
\newcommand{\ZZ}{\mathbb{Z}}

\newcommand{\abs}[1]{\lvert#1\rvert}
\DeclareMathOperator{\GL}{GL} 
 
\newtheorem{theorem}{Theorem}
\newtheorem{corollary}{Corollary}[section]

\newtheorem{lemma}{Lemma}[section]
\newtheorem{proposition}{Proposition}[section]

\theoremstyle{definition}
\newtheorem{definition}{Definition}[section]
\newtheorem{remark}{Remark}[section]
\newtheorem{fact}{Fact}[section]

\numberwithin{equation}{section}

\DeclareMathOperator{\diag}{diag}
\begin{document}

\title{LOCAL RIGIDITY OF PARTIALLY HYPERBOLIC
ACTIONS}
\author{Zhenqi WANG}
\address{Department of Mathematics\\ The Pennsylvania State
University\\ University Park, PA,16802}
\email{wang\_z@math.psu.edu}

\maketitle
\begin{abstract}
We consider partially hyperbolic abelian algebraic high-rank actions
on compact homogeneous spaces obtained from simple indefinite
orthogonal and unitary groups. In the first part of the paper, we
show local differentiable rigidity for such actions. The conclusions
are based on progress towards computations of the Schur multipliers
of these non-split groups, which is the main aim of the second part.
\end{abstract}
\tableofcontents
\section{Introduction}

In this paper we extend
 results of D.Damjanovic and A.Katok  about rigidity of  certain diagonal actions on compact homogeneous spaces \cite{Damjanovic1,Damjanovic2,Damjanovic3} from split to some non-split Lie groups.

A. Katok and R.Spatzier considered the differentiable rigidity of
Weyl chamber flows on symmetric spaces: let $G$ be a connected
semisimple Lie group of real rank $\ge 2$,
$\Gamma$ a cocompact torsion-free lattice in
$G$; let $A$ be a maximal split Cartan subgroup of $G$, and let $K$
be the compact part of the centralizer of $A$ which intersects with
$A$ trivially. If $G$ is of $\RR$-rank greater than one, the action
of $A$ on the space $K\backslash G/\Gamma$ is an Anosov (normally hyperbolic)
action, and
it is locally differentiably rigid \cite{Spatzier}. The method of proof  can be called {\em a priori} regularity since it is  based on showing  smoothness of the Hirsch-Pugh-Shub orbit equivalence \cite{shub}. Another ingredient  in the Katok-Spatzier method is cocycle rigidity used to
``straighten out''  a time change; it is proved by a harmonic analysis method.

In \cite{Damjanovic1,Damjanovic2}, a special case
$G=SL(n,\RR)(n\geq
 3)$  is considered.
 In this case, rather than the full $A$ action by left
 translations on $SL(n,\RR)/\Gamma$, they considered the restrictions of the full
 diagonal action to subgroups  in the acting group $\RR^{n-1}$ that contain lattices in  2-planes in general position. Those actions are
 partially hyperbolic rather than Anosov and {\em a priori} regularity
 methods is not applicable.
 The proof of cocycle rigidity and differentiable
 rigidity
 in \cite{Damjanovic1,Damjanovic2} is ``geometric'', in contrast with earlier  proofs in \cite{Spatzier}.
 The totally new manner is based on geometry and
combinatorics of invariant foliations and using insights from
algebraic $K$-theory as an essential tool.

The approach of \cite{Damjanovic1,Damjanovic2} was further employed
in \cite{Damjanovic3}, for extending cocycle rigidity and
differentiable
 rigidity from
$SL(n,\RR)/\Gamma$ and $SL(n,\CC)/\Gamma$ to compact homogeneous
spaces obtained from some simple split Lie groups of nonsymplectic type.

The purpose of this paper is to further extend the results of the
above-mentioned authors to higher rank partially hyperbolic actions
on compact homogeneous spaces obtained from indefinite orthognal and
unitary groups. In the present work we derive information about
generators and relations in    these
non-split groups which are not readily available from the
literature. As soon as this algebraic information is obtained, the  geometric scheme  developed by Katok and Damjanovic
essentially applies to non-split cases. However, there are still at
some places significant
 technical differences which require some new arguments to
handle them.

In Section \ref{sec:10}, based on the conclusions about Schur
multipliers proved in Section
\ref{sec:12}-\ref{sec:15}, we apply the approach of \cite{Damjanovic1} to prove
trivialization of small, non-abelian, group valued cocycles over
partially hyperbolic abelian algebraic acrion described in Section
\ref{sec:9}. The key points are:  (i)
 local transitivity of Lyapunov foliations, (ii) the fact  that some
Lyapunov cycle can be approximated by a composition of conjugates to
stable cycles, and (iii) vanishing of  periodic cycle functionals on broken paths
along leaves of stable and unstable foliations generated by
multiple-dimension root spaces.
 Once the cocycle rigidity is
obtained in Section \ref{sec:10} via geometric method, proving that
cocycle rigidity is robust under $C^2$-small perturbations is
similar to the case of $G=SL(n,\RR)$, which is treated in
\cite{Damjanovic2} (See Section~\ref{sec:8}).

In Section \ref{sec:12}-\ref{sec:15}, we make sufficient progress
towards the computations of the Schur multipliers of $SO^+(m,n)$ and
$SU(m,n)$ where $m\geq n\geq 3$ to obtain information needed for the
proofs in Section \ref{sec:10}. This work is extension of the ideas
of R. Steinberg \cite{Steinberg} and V. Deodhar \cite{Deodhar}.
Steinberg considered the so-called Schur multipliers (cf. Steinberg
[\cite{Steinberg2}, Section 7]) of rational points of simply
connected Chevalley groups and obtained results of importance,
especially in the case of split semisimple algebraic groups. Deodhar
extended Steinberg's theory to a more general class  of quasi-split
groups. In the present work, our main aim is to obtain similar results
for some non-split groups that are not quasi-split. Deodhar's construction carries carry over to
the non-split groups with minor changes. However,  it does not provide sufficient information  and needs to be  supplemented by  some new method.

Let $k$ denote an arbitrary local field, and $G$ denote a connected,
simply connected algebraic group which is defined and absolutely
almost simple over $k$. In Section \ref{sec:12} we use Deodhar's
results freely in the general case: We construct explicitly, in
terms of generators and relations, a ``universal central extension"
for $G^+_k$ generated by $k$-rational unipotent elements which
belong to the radical of a parabolic subgroup defined over $k$ in
$G$.

In Section \ref{sec:16} and \ref{sec:15}, our main aim is to make
some progress towards the computation of the fundamental group
$\pi_1(G_k)$ (= Schur multiplier) of $G_k$ when
$G_\RR=SO^+(m,n)(m\geq n\geq 3)$ and $G_\RR=SU(m,n)(m\geq n\geq 3)$.
For this purpose, we introduce a way to tackle the ``rotation" and
``reflection" in root spaces with dimensions greater than 1.

For definitions and general background on
partially hyperbolic dynamical systems see \cite{pesin2}; all
necessary background on algebraic actions can be found in
\cite{Nitica}. We will also strongly rely on definitions,
constructions and results from the earlier papers on the subject
\cite{Damjanovic1,Damjanovic2,Damjanovic3}. In the present paper we
consider algebraic actions of $\ZZ^{k}\times \RR^{\ell}, k+\ell\geq
2$. We treat generic restrictions of full split Cartan actions on
$SO^+(m,n)/\Gamma$ and on $SU(m,n)/\Gamma$(where $\Gamma$ is a
cocompact lattice).

I'd like to thank Nigel Higson for pointing out to some relevant
sources and results in algebraic $K$-theory and  many stimulating
discussions. I am grateful to Grigory Margulis and Yuri Zarhin who
kindly helped me in algebraic area. My sincere thanks are also due
to my advisor Anatole Katok, who suggested the problem to me. He
also looked through several preliminary versions of this paper,
suggested some changes, made important comments and encouraged me a
lot.

\section{Setting and results}
\subsection{Generic restrictions of split Cartan actions}\label{sec:6} Let $Q$ be a
non-degenerate standard bilinear form on  $\RR^{m+n}$
 of signature
$(n,m)$. The group $SO^+(m,n)$ is then the connected Lie group of
$(m+n)\times(m+n)$ matrices that preserve $Q$ with determinant 1.
Then we can choose a base of $\RR^{m+n}$ in terms of which the
quadratic form $Q$ is given by $Q(e_i,e_{i+n})=1$, if $1\leq i\leq
n$; $Q(e_j,e_j)=1$, if $2n+1\leq j\leq m+n$ and $Q(e_i,e_j)=0$
otherwise.

Using  this base,  the  Lie algebra $so(m,n)$ of $SO^+(m,n)$   can be
represented  as $(m+n)\times (m+n)$ matrices
$$ \begin{pmatrix}A_1 & A_2 & \vline & B_1\\
A_3 & A_4 & \vline & B_2\\
 \hline C_1& C_2 &\vline &D \end{pmatrix},$$
  where $A_1, A_2, A_3, A_4$ are $n\times n$ matrices, $B_1, B_2$
  are $n\times (m-n)$ matrices, $C_1, C_2$ are $(m-n)\times n$
  matrices and $D$ is a $(m-n)\times (m-n)$ matrix satisfying
\begin{alignat*}{3}
A_1&=-A^{\tau}_4, &\qquad A^{\tau}_2&=-A_2, &\qquad
A^{\tau}_3&=-A_3, \\
D^{\tau}&=-D, &\qquad B_1&=-C^{\tau}_2, &\qquad B_2&=-C^{\tau}_1.
\end{alignat*}
Here $M^\tau$ denotes the transpose of the matrix $M$.

 Let $H$ be a non-degenerate standard Hermitian form of signature $(n,m)$. Then we can choose a base of $\CC^{m+n}$(under a
linear transformation with real coefficients) 
in terms of which the quadratic form $H$ is given by
$H(e_i,e_{i+n})=1$, $1\leq i\leq n$, $H(e_j,e_j)=1$, $2n+1\leq j\leq
m+n$ and $H(e_i,e_j)=0$, otherwise 0.
The group $SU(m,n)$ is then the connected Lie group of
$(m+n)\times(m+n)$ matrices that preserve $H$ with determinant 1.

Using  this base  the Lie algebra   $su(m,n)$ of $SU(m,n)$ can be expressed as $(m+n)\times (m+n)$ matrices
$$ \begin{pmatrix}A_1 & A_2 & \vline & B_1\\
A_3 & A_4 & \vline & B_2\\
 \hline C_1& C_2 &\vline &D \end{pmatrix},$$
  where $A_1, A_2, A_3, A_4$ are $n\times n$ matrices, $B_1, B_2$
  are $n\times (m-n)$ matrices, $C_1, C_2$ are $(m-n)\times n$
  matrices and $D$ is a $(m-n)\times (m-n)$ matrix satisfying
  \begin{alignat*}{3}
A_1&=-\overline{A}^{\tau}_4, &\qquad \overline{A}^{\tau}_2&=-A_2,
&\qquad
\overline{A}^{\tau}_3&=-A_3, \\
\overline{D}^{\tau}&=-D, &\qquad B_1&=-\overline{C}^{\tau}_2,
&\qquad B_2&=-\overline{C}^{\tau}_1.
\end{alignat*}
Here $\overline{M}^\tau$ denotes the complex conjugate transpose of
the matrix $M$.

 Let $X:= SO^+(m,n)/\Gamma,$ (corr. $SU(m,n)/\Gamma$) with $m\geq
n\geq 3$ and $\Gamma$ a cocompact lattice in $X$. Let
\begin{align*}
D_+ = \exp\mathbb{D_+}= &\{\text{diag}\bigl(\exp
t_1,\exp t_2, \dots ,\exp t_n,\exp(-t_1),\exp(-t_2),\dots,
\\&\text{exp}(-t_{n}),1,\dots,1\bigl) :(t_1, \dots, t_{n})\in\RR^{n}\}
\end{align*}
 be the group of diagonal matrices with lower $(m-n)\times(m-n)$ matrix
 identity. In fact, $D_+$ is
 the maximal split Cartan subgroup both in
 $SO^+(m,n)$ and $SU(m,n)$.

We denote the action of $D^+$  on $X$ by left translations
 by $\alpha_0$ and call it {\em the split Cartan action}.

For $1\leq i\neq j\leq n$ the hyperplanes in $\mathbb{D_+}$ defined
by $$\mathbb{H}_{i-j} = \{(t_1, \dots , t_{n}) \in \mathbb{D_+} : t_i
= t_j\},$$
$$\mathbb{H}_{i+j} = \{(t_1, \dots , t_{n}) \in
\mathbb{D_+} : t_i+t_j=0\}\quad \text{and}$$   $$\mathbb{H}_{i} = \{(t_1, \dots,
t_{n}) \in \mathbb{D_+} : t_i=0\}$$
(exist if $m-n\geq 1$) are
\emph{Lyapunov hyperplanes} for the action $\alpha_0$, i.e. kernels
of Lyapunov exponents of $\alpha_0$. Elements of
$\mathbb{D_+}\backslash \bigcup \mathbb{H}_{r}$(where $r=i\pm j,i$)
are \emph{regular} elements of the action. Connected components of
the set of regular elements are $Weyl$ $chambers$.

The smallest non-trivial intersections of stable foliations of
various elements of the action $\alpha_0$ are $Lyapunov$
$foliations$. Each regular element either exponentially expands or
exponentially contracts each of those leaves. For more details, see
Section \ref{sec:7}.

\begin{definition} A two-dimensional plane $\mathbb{P}\subset\mathbb{D_+}$ is in
$general$ $position$ if it intersects any two distinct Lyapunov
hyperplanes
along distinct lines.
\end{definition}
 Let $\mathbb{G}\subset
\mathbb{D_+}$ be a closed subgroup which contains a lattice
$\mathbb{L}$ in a plane in general position and let $G=\exp\mathbb{G}$.
One can naturally think of $G$ as the image of an
injective homomorphism
\\$i_0 : \ZZ^k\times\RR^{\ell}\rightarrow D_+$
(where $k+\ell\geq 2)$.
\begin{definition}
The  action $\alpha_{0,G}$ of $G$ by
left translations on $X$ is given by
\begin{align}\label{for:15}
\alpha_{0,G}(a,x)=i_0(a)\cdot x
\end{align}
and will be referred to as a {\em higher-rank generic restriction of split Cartan actions} or just a
{\em generic restriction} for short.
\end{definition}

\subsection{Cocycles and rigidity}\label{sec:9} Let $\alpha:A\times
M\rightarrow M$ be an action of a topological group $A$ on a compact
Riemannian manifold M by diffeomorphisms. For a topological group
$Y$ a $Y$-valued {\em cocycle} (or {\em an one-cocycle})
over $\alpha$ is a
continuous function $\beta : A\times M\rightarrow Y$ satisfying:
\begin{align}
\beta(ab, x) = \beta(a, \alpha(b, x))\beta(b, x)
\end{align}
 for any $a, b \in A$. A cocycle is
{\em cohomologous to a constant cocycle} (cocycle not depending on $x$) if
there exists a homomorphism $s : A\rightarrow Y$ and a continuous
transfer map $H : M\rightarrow Y$ such that for all $a \in A $
\begin{align}\label{for:13}
 \beta(a, x) = H(\alpha(a, x))s(a)H(x)^{-1}
\end{align}
In particular, a cocycle is a {\em coboundary} if it is cohomologous to
the trivial cocycle $\pi(a) = id_Y$, $a \in A$, i.e. if for all $a
\in A$ the following equation holds:
\begin{align}
 \beta(a, x) = H(\alpha(a, x))H(x)^{-1}.
\end{align}
For more detailed information on cocycles adapted to the present
setting see [3]. Let
$$ Y=\begin{pmatrix}A_1 & 0 \\
0 & A_2 \\
  \end{pmatrix},$$
be the subgroup of $SU(m,n)$ with \begin{align*}
A_1=&\{\diag\bigl(\exp z_1,\dots,\exp z_n,\exp
(-\overline{z_1}),\dots,\exp (-\overline{z_n})\bigl):\\
&(z_1,\dots,z_n)\in\CC^n\}
\end{align*} and $A_2$ $(m-n)\times(m-n)$ unitary
matrices. 
Let $Y_X=Y\cap
SO^+(m,n)$ if $X=SO^+(m,n)/\Gamma$(corr. $Y_X=Y$ if
$X=SU(m,n)/\Gamma$). $Y_X$ is isomorphic to $\RR^{n}\times SO(m-n)$
when $X=SO^+(m,n)/\Gamma$ and isomorphic to $\RR^{n}\times
\mathbb{T}^{n-1}\times U(m-n)$ if $X=SU(m,n)/\Gamma$.

Let $\mathbb{P}\subset\mathbb{D_+}$ be a 2-dimensional plane in
general position. We will show (Theorem \ref{th:5} in Section
\ref{sec:10}) that every small H\"older cocycle with values
in $Y_X$ over the action $\alpha_{0,G}$, where $G$ is any subgroup
of $D_+$ which contains $\exp\mathbb{P}$, is cohomologous to a
constant cocycle. Similarly to the proofs in
\cite{Damjanovic1,Damjanovic2} we use the geometric structure of
Lyapunov foliations of the action. By applying the method of
\cite{Damjanovic1,Kononenko} we show that a cocycle over a partially
hyperbolic action with locally transitive Lyapunov foliations is
cohomologous to a constant cocycle if and only if the periodic cycle
functional (PCF) vanishes on all closed broken paths whose pieces
lie in leaves of Lyapunov foliations of the action. Furthermore, the
presentations of Schur multipliers of $SO^+(m,n)$ and $SU(m,n)$ that
we will construct in Sections \ref{sec:12}-\ref{sec:15}, give explicit description
of closed broken paths along Lyapunov foliations which leads to
vanishing of the PCF on all such paths and to cocycle rigidity.
Smoothness of the transfer map for smooth cocycles is a consequence
of the fact that for a generic restriction, the Lyapunov
distributions along with their Lie brackets generate the tangent
space at every point.

\subsection{Formulation of results}
Our main results are contained in the following two
theorems.
\begin{theorem}[Differentiable rigidity of generic restrictions]\label{th:1}  Let $\alpha_{0,G}$ be a high rank generic restriction of the action of a maximal split Cartan subgroup on $SO^+(m,n)/\Gamma$ or $SU(m,n)/\Gamma$ where $m\ge n\ge 3$.

If $\tilde{\alpha}$ is $C^{\infty}$ action of $\ZZ^k\times \RR^{\ell}$
sufficiently $C^2$-close to $\alpha_{0,G}$, then there exists a
homomorphism $i : \ZZ^k\times \RR^{\ell} \rightarrow Y_X$ close to
$i_0$ and a $C^\infty$ diffeomorphism $h : X \rightarrow X$ such
that $\tilde{\alpha}(a, h(x)) = h(i(a)\cdot x)$ for all $\ZZ^k\times
\RR^{\ell}$.
\end{theorem}

The principal ingredient in the proof of Theorem \ref{th:1} is the
next theorem which is the main technical result of the present
paper. It extends the cocycle rigidity result from
\cite{Damjanovic2,Damjanovic3}.

\begin{theorem}[Cocycle rigidity for perturbations]\label{th:2} Let $\alpha_{0,G}$ be a generic restriction of the action of a maximal split Cartan subgroup on $SO^+(m,n)/\Gamma$ or $SU(m,n)/\Gamma$ where $m\ge n\ge 3$.
Let $\tilde{\alpha}$ be a sufficiently
$C^2$-small $C^1$ perturbation of $\alpha_{0,G}$.

If $\beta$ is a
$H\tilde{o}lder$ cocycle over $\tilde{\alpha}$ with values in $Y_X$
then $\beta$ is cohomologous to a constant cocycle given by a
homomorphism $s : \ZZ^k\times \RR^{\ell}\rightarrow Y_X$ via a
continuous transfer function. Furthermore, if the cocycle $\beta$ is
sufficiently small in a $H\tilde{o}lder$ norm the transfer map is
$C^0$ arbitrary small.
\end{theorem}
Let $X_1:= M\backslash SO^+(m,n)/\Gamma,$ with $m\geq n\geq 3$,
where $M=SO(m-n)$ and $\Gamma$ a cocompact lattice in $SO^+(m,n)$.
\begin{corollary}\label{cor:3}  Let $\alpha_{0,G}$ be a high rank generic restriction of the action of a maximal split Cartan subgroup
on $M\backslash SO^+(m,n)/\Gamma$ where $m\ge n\ge 3$.

If $\tilde{\alpha}$ is $C^{\infty}$ action of $\ZZ^k\times
\RR^{\ell}$ sufficiently $C^2$-close to $\alpha_{0,G}$, then there
exists a homomorphism $i : \ZZ^k\times \RR^{\ell} \rightarrow Y_X$
close to $i_0$ and a $C^\infty$ diffeomorphism $h : X \rightarrow X$
such that $\tilde{\alpha}(a, h(x)) = h(i(a)\cdot x)$ for all
$\ZZ^k\times \RR^{\ell}$.
\end{corollary}
The action by left translations of $D_+$ on $X_1$ is the Weyl
chamber flow (WCF) and we denote this action by $\alpha_0$. The
following result is a special case of Corollary \ref{cor:3}.
\begin{corollary}\label{cor:1}
Let $\alpha_{0}$ be the WCF on $M\backslash SO^+(m,n)/\Gamma$ where
$m\ge n\ge 3$. If $\tilde{\alpha}$ is $C^{\infty}$ action
sufficiently $C^2$-close to $\alpha_{0}$, then there exists a
homomorphism $i : \ZZ^k\times \RR^{\ell} \rightarrow D_+$ close to
identity and a $C^\infty$ diffeomorphism $h : X \rightarrow X$ such
that $\tilde{\alpha}(a, h(x)) = h(i(a)\cdot x)$.
\end{corollary}

\subsection{Comments on related problems}
Results of this paper belong to the general program of establishing  various flavors of local differentiable rigidity for  partially hyperbolic algebraic actions of  higher rank abelian groups.  For general comments on that program see \cite[Section 1]{Damjanovic4}. We do not attempt a comprehensive  overview of the current state of the program but  restrict ourselves to few comments on those aspects that are closely related to our results.

First let us discuss generic restrictions of Cartan actions on other higher rank
simple Lie groups.

 The condition $n\geq 3$  for $SO^+(m,n)$ and $SU(m,n)$ is necessary for the method used in  this
paper.  Algebraically, the complications in these
cases are similar  to those encountered in
the case $n=2$ for the linear Steinberg groups (\cite{Meara} Section
1.4E). Geometrically, homotopy
classes can't be reduced to each other using allowable substitution.
Thus the cases of groups $SO^+(m,2)$ and $SU(m,2)$ remain open.

Using the technique in dealing with an infinite fundamental group
similar to that we use in case of $SU(m,n)$, we can solve the
cocycle rigidity and differentiable rigidity problem of generic
restrictions for Cartan action on the split groups  $Sp(n,\RR)$ that
is left out in \cite{Damjanovic3}. This result  will appear in a
separate paper.

Extension to generic restrictions of split Cartan actions for other
classical non-split simple groups requires information about
generators and relations not  available from the literature. This is not surprising that this is already the case with the groups $SO(m,n)$ and $SU(m,n)$.
While our general approach may (and probably should) work
more techniques are needed for calculations of generators and
relations of matrix groups of $SL(n,\mathbb{H})$, $SP(m,n)$ and
$SO^*(2n)$. Those are defined over quaternions and one has to
double the sizes of matrices to represent them by   complex matrices
which makes the ``rotations'' much more complex.

We have not looked into  exceptional
groups  and are not aware of any  effective way to solve the generators and relations problem for those groups. Similar to the quaternionic cases,  one hopes that with sufficiently hard work those groups (with the possible exception of the rank two case) may be amenable to our method.

A  necessary condition for applicability of the Damjanovic--Katok  geometric method  (although not for local rigidity) is that   contracting distributions of various action elements and their brackets of all orders   generate  the tangent space to the
phase space.  Generic restrictions  for Cartan actions satisfy that condition.
Naturally one  may look at  non-generic restrictions of Cartan actions.
Some of those still satisfy it  but nothing is  known even for  the $SL(n, \RR)$ case
since one cannot use the full force of the algebraic $K$-theory machinery.
More detailed analysis of  generators and relations may   help resolve some of those cases.

The next  natural step is to consider  a similar problem on   products of simple groups factored by irreducible lattices. There are certain cases  that  look amenable to the method.  Finally, one may consider  actions  of higher-rank abelian subgroups on homogeneous spaces of compact  extensions of simple or semisimple Lie groups.
Since the compact fibers  are included into  the neutral directions one should consider cocycles with values in  more general groups that are extensions of abelian groups by compact groups.  Our methods can be extended to at least some of those cases.

\section{Cocycle rigidity for the  actions $\alpha_{0,G}$}\label{sec:10}
The purpose of this section is to describe a geometric method for
proving cocycle rigidity for this action following \cite{Damjanovic1, Damjanovic2}. \begin{theorem}\label{th:5}
Any $Y_X$-valued H\"older cocycle over the generic
restriction of the split Cartan action on $X$ is cohomologous to a
constant cocycle via a  H\"older $C^\infty$ transfer function.

Any $Y_X$-valued $C^\infty$ cocycle over the generic
restriction of the split Cartan action on $X$ is cohomologous to a
constant cocycle via a $C^\infty$ transfer function.

\end{theorem}

\subsection{Preliminaries} Let $\alpha: A \rightarrow
\text{Diff}(M)$ be an action of $A := \RR^k$, $k\in\NN$ on a compact
manifold $M$ by diffeomorphisms of $M$ preserving an ergodic
probability measure $¦Ì$. Then there are finitely many linear
functionals $\lambda$ on $A$, called $Lyapunov$ $exponents$, a set
of full measure $\Lambda$ and a measurable splitting of the tangent
bundle $T_\Lambda M =\bigoplus_\lambda E^\lambda$, such that for
$v\in E^\lambda$ and $a\in A$ the Lyapunov exponent of $v$ with
respect to $\alpha(a)$ is $\lambda(a)$.

If $\chi$ is a non-zero Lyapunov exponent then we define its
$coarse$ $Lyapunov$ $subspace$ by $$E_\chi:=
\bigoplus_{\lambda=c\chi:c>0}E^\lambda.$$ For every $a\in A$ one can
define stable, unstable and neutral subspaces for a by: $E^s_a
=\bigoplus_{\lambda(a)<0}E^\lambda$, $E^u_a
=\bigoplus_{\lambda(a)>0}E^\lambda$ and $E^0_a
=\bigoplus_{\lambda(a)=0}E^\lambda$. In particular, for any $a\in A
:= \bigcap_{\chi\neq 0}(Ker\chi)^c$ the subspace $E^0_a$  is the
same and thus can be denoted simply by $E^0$. Hence we have for any
such $a:$ $$TM = E^s_a\oplus E^0\oplus E^u_a .$$ See
[\cite{Kalinin2}, Section 5.2] for more details.
If in addition $E^0$ is a continuous distribution uniquely
integrable to a foliation $\mathcal{N}$ with smooth leaves, and if
there exists $a\in A$ such that $\alpha(a)$ is uniformly normally
hyperbolic with respect to $\mathcal{N}$ (in the sense of the
Hirsch-Pugh-Shub \cite{shub}) then $\alpha$ is a $partially$
$hyperbolic$ $action$. Elements in $A$ which are uniformly normally
hyperbolic with respect to $\mathcal{N}$ are called regular. Let
$\widetilde{A}$ be the set of regular elements. We call an action a
partially hyperbolic $A$ action if the set $\widetilde{A}$ is dense
in $A$. In particular, if $E^0$ is the tangent distribution to the
orbit foliation of a normally hyperbolic action, then the action is
called Anosov.

If the set $\widetilde{A}$ is dense in $A$, then for each non-zero
Lyapunov exponent $\chi$ and every $p\in M$ the coarse Lyapunov
distribution is: $$E_\chi(p) = \bigcap_{a\in \widetilde{A},
\chi(a)<0}E^s_a(p).$$ The right-hand side is H\"older
continuous and $E^\chi$ can be extended to a H\"older
distribution tangent to the foliation $\mathcal{T}_\chi:=
\bigcap_{a\in \widetilde{A}, \chi(a)<0} \mathcal{W}^s_a(p)$ with
$C^\infty$ leaves. This is the $coarse$ $Lyapunov$ $foliation$
corresponding to $\chi$ (See [\cite{Damjanovic1}, Section 2] and
\cite{Kalinin}).

We denote by $\chi_1,\dots,\chi_r$ a maximal collection of non-zero
Lyapunov exponents that are not positive multiples of one another
and by $\mathcal{T}_1,\dots, \mathcal{T}_r$ the corresponding coarse
Lyapunov foliations.

Given a foliation $\mathcal{T}_i$ and $x\in M$ we denote by
$\mathcal{T}_i(x)$ the leaf of $\mathcal{T}_i$ through $x$.
\subsection{Paths and cycles for a collection of foliations}. In this section we recall some notation and results from
\cite{Damjanovic2}. Let $\mathcal{T}_1, \dots, \mathcal{T}_r$ be a
collection of mutually transversal continuous foliations on $M$,
with smooth simply connected leaves.

For $N\in\NN$ and $j_k\in\{1, \dots , r\}, k\in\{1, \dots ,N-1\}$ an
ordered set of points $p(j_1,\dots , j_{N-1}) : x_1, \dots , x_N \in
M$ is called an $\mathcal{T}$-path of length $N$ if for every $k\in
\{1, \dots ,N-1\}, x_{i+1}\in \mathcal{T}_{j_k}(x_k)$. A closed
$\mathcal{T}$-path(i.e., when $x_N=x_1$) is a $\mathcal{T}$-cycle.

A $\mathcal{T}$-cycle $p(j_1, \dots , j_{N-1}) : x_1, \dots ,
x_N=x_1 \in M$ is called $stable$ for the $A$ action $\alpha$ if
there exists a regular element $a\in A$  such that the whole cycle
$p$ is contained in a leaf of the stable foliations for the map
$\alpha(a,\cdot)$, i.e., if
\begin{align*}
\bigcap_{k=1}^N\{a:\chi_{j_k}(a)<0\}\neq \phi.
\end{align*}
\begin{definition}\label{def:1}
Let $p(j_1, \dots , j_{N-1}) : x_1, \dots , x_N$ and $p_n(j_1,
\dots, j_{N-1}) : x^n_1, \dots , x^n_N$ be two $\mathcal{T}$-paths.
Then $p=\lim_{n\rightarrow\infty}p_n$ if for all $k\in
\{1,\dots,N\}$
\begin{align*}
x_k=\lim_{n\rightarrow\infty}x^n_k.
\end{align*}
Limits of $\mathcal{T}$-cycles are defined similarly.

Two $\mathcal{T}$-cycles, $p(j_1, \dots , j_{N+1}) : x_1, \dots
,x_k,y,x_k,\dots, x_{N}=x_1$ and $p(j_1, \dots , j_{N-1}) :
x_1,\dots , x_k, x_{k+1},\dots x_N$ are said to be conjugate if
$y\in \mathcal{T}_i(x_k)$ for some $i\in \{1,\dots,r\}$.
For $\mathcal{T}$-cycles, $p(j_1, \dots , j_{N-1}) : x_1, \dots
,x_{N}=x_1$ and $p'(j_1', \dots , j_{K-1}') :x_1= x_1', \dots,
x_{K}'=x_1$ define their composition or concatenation $p\ast p'$ by
$$p\ast p'(j_1,\dots, j_{N-1}, j_1', \dots , j_{K-1}') : x_1, \dots x_N,
x_1', \dots , x_K'= x_1.$$ Let
$\mathcal{A}\mathcal{S}^s_{\mathcal{T}}(\alpha)$ denote the
collection of stable $\mathcal{T}$-cycles. Let
$\mathcal{A}\mathcal{S}_{\mathcal{T}}(\alpha)$ denote the collection
of $\mathcal{T}$-cycles which contains
$\mathcal{A}\mathcal{S}^s_{\mathcal{T}}(\alpha)$ and is closed under
conjugation, concatenation of cycles, and under the limitation
procedure defined above.
$\mathcal{A}\mathcal{S}_{\mathcal{T}}^x(\alpha)$ denotes the subset
of $\mathcal{A}\mathcal{S}_{\mathcal{T}}(\alpha)$ which contain
point $x$.

A path $p : x_1,\dots ,x_k,\dots, x_N$ reduces to a path $p^{'}:
x_1, x^{'}_2,\dots ,x^{'}_k,\dots, x_N^{'}$ via an
$\alpha$-$allowable$ $\mathcal{T}$-substitution if the
$\mathcal{T}$-cycle
\begin{align*}
p\ast p^{'}: x_1, \dots ,x_k,\dots, x_{N-1},x_N,
x_{N-1}^{'},\dots,x^{'}_2,x_1
\end{align*}
obtained by concatenation of $p$ and $p^{'}$ is in the collection
$\mathcal{A}\mathcal{S}_{\mathcal{T}}(\alpha)$.

Two $\mathcal{T}$-cycle $c_1$ and $c_2$ are $\alpha$-equivalent if
$c_1$ reduces to $c_2$ via a finite sequence of $\alpha$-allowable
$\mathcal{T}$ -substitutions. A $\mathcal{T}$-cycle we call
$\alpha$-reducible if it is in
$\mathcal{A}\mathcal{S}_{\mathcal{T}}(\alpha)$.
\end{definition}

\begin{definition}
For $N\in\NN$ and $j_k\in\{1, \dots , r\}, k\in\{1, \dots ,N\}$ an
ordered set of points $p(j_1, \dots, j_{N}) : x_1, \dots , x_N,
x_{N+1}=x_1 \in M$ is called an $\mathcal{T}$-cycle of length $N$ if
for every $k\in \{1, \dots ,N\}, x_{i+1}\in \mathcal{T}_{j_k}(x_k)$.
A $\mathcal{T}$ cycle which consists of a single point is a trivial
$\mathcal{T}$-cycle.
\end{definition}
\begin{definition}
Foliations $\mathcal{T}_1, \dots, \mathcal{T}_r$ are locally
transitive if there exists $N\in\NN$ such that for any $\varepsilon
> 0$ there exists $\delta> 0$ such that for every $x\in M$ and for every
$y\in  B_X(x, \delta)$ (where $B_M(x, \delta)$ is a $\delta$ ball in
$M$) there is a $\mathcal{T}$-path $p(j_1, \dots , j_{N-1}) : x =
x_1, x_2, \dots, x_{N-1}, x_N = y$ in the ball $B_M(x, \varepsilon)$
such that $x_{k+1}\in \mathcal{T}_{j_k}(x_k)$ and
$d_{\mathcal{T}_{j_k}(x_k)}(x_{k+1}, x_k) < 2\varepsilon$.
\end{definition}
In other words, any two sufficiently close points can be connected
by a $\mathcal{T}$-path of not more than $N$ pieces of a given
bounded length. Here, for a submanifold $Y$ in $M$, $d_Y (x, y)$
denotes the infimum of lengths of smooth curves in $Y$ connecting
$x$ and $y$.

\begin{definition}
For a partially hyperbolic $A$-action $\alpha$ on a compact manifold
$M$ with coarse Lyapunov foliations $\mathcal{T}_1, \dots,
\mathcal{T}_r$ and for a cocycle $\beta : A ¡Á M \rightarrow Y$ over
$\alpha$, where $Y$ is a Lie group, we define $Y$-valued potential
of $\beta$ as
\begin{align*}\label{for:1}
&\left\{\begin{aligned} &P^j_a(y,x)=\lim_{n\rightarrow
+\infty}\beta(na,
y)^{-1}\beta(na,x),\qquad \chi_j(a)<0\\
&P^j_a(y,x)=\lim_{n\rightarrow -\infty}\beta(na,
y)^{-1}\beta(na,x),\qquad \chi_j(a)>0
 \end{aligned}
 \right.
\end{align*}
\end{definition}
Now for any $\mathcal{T}$-cycle $\mathfrak{c}: x_1, \dots, x_{N+1} =
x_1$ on $M$, we can define the corresponding periodic cycle
functional:
\begin{align}
\text{PCF}(\mathfrak{c})(\beta)=\prod_{i=1}^{N}P^{j(i)}_a(x_i,x_{i+1})(\beta).
\end{align}
It is proved in \cite{Damjanovic1} that the expression for (PCF)
does not depend on the choice of $a$. For a general Lie group H
there may be a ``competition" between the exponential speed of decay
for the distance between $nax$ and $nay$ on the one hand, and the
exponential growth of the cocycle norm on the other. More
information about guaranteeing convergence of non-abelian potentials
can be found in \cite{Damjanovic1}.

In this paper, we only consider $Y=Y_X$ which possesses a
bi-invariant metric, so the limits in the right hand part of
\ref{for:1} always exist.

Two essential properties of the PCF which are crucial for our
purpose are that PCF is continuous and that it is invariant under
the operation of moving cycles around by elements of the action
$\alpha$. We end this section with an important proposition which is
the base of our further proof.
\begin{proposition}\emph{(Proposition 4. \cite{Damjanovic1})} \label{le:8}
Let $\alpha$ be an $\RR^k$ action by diffeomorphisms on a compact
Riemannian manifold $M$ such that a dense set of elements of $\RR^k$
acts normally hyperbolically with respect to an invariant foliation.
If the foliations $\mathcal{F}_1,\dots,\mathcal{F}_r$ are locally
transitive and if $\beta$ is a H\"older cocycle over the
action $\alpha$ such that $F(\mathcal{C})(\beta) = 0$ for any cycle
$\mathcal{C}$ then: $\beta$ is cohomologous to a constant cocycle
via a continuous map $h:M\rightarrow Y$.
\end{proposition}
\subsection{Split Cartan actions on $SO^+(m,n)/\Gamma$ and
$SU(m,n)/\Gamma$}\label{sec:7}
We use  notations from
Section
\ref{sec:6}. Let $d(\cdot , \cdot)$ denote a right invariant metric
on $SO^+(m,n)$ and the induced metric on $SO^+(m,n)/\Gamma$. We use
$e_{k,\ell}$ to denote the matrix with the
 $(k,\ell)$ element  equal to 1, and all other elements equal to  0.  Let $1\leq i,j\leq
n, i\neq j$ be two distinct indices, $\ell \leq
 m-n$, and let $\exp$ be the
exponentiation map for matrices.

Let $\Phi$ be the root system of $SO^+(m,n)$ with respect to $D_+$.
The roots are $\pm L_i \pm L_j(i<j\leq n) $, whose dimensions are
one and $\pm L_i(1\leq i\leq n)$ are also roots if $m\geq n+1$ with
dimensions
 $m-n$.
 The corresponding root spaces are
\begin{align*}
 &\mathfrak{g}_{L_i+L_j}=\RR(e_{i,j+n}-e_{j,i+n})_{i<j}, \qquad
 \mathfrak{g}_{L_i-L_j}=\RR(e_{i,j}-e_{j+n,i+n})_{i\neq j},\\
 &\mathfrak{g}_{-L_i-L_j}=\RR(e_{j+n,i}-e_{i+n,j})_{i<j},\\
 & \mathfrak{g}_{L_i}=\bigoplus_{\ell\leq m-n}\RR f_{L_i}^\ell,\text{ where }f_{L_i}^\ell=e_{i,2n+\ell}-e_{2n+\ell,i+n},\\
 &\mathfrak{g}_{-L_i}=\bigoplus_{\ell\leq m-n}\RR f_{-L_i}^\ell,\text{ where }f_{-L_i}^\ell=e_{i+n,2n+\ell}-e_{2n+\ell,i}.
\end{align*}
Let
\begin{align*}
&f_{L_i+L_j}=(e_{i,j+n}-e_{j,i+n})_{i<j},\qquad
f_{L_i-L_j}=(e_{i,j}-e_{j+n,i+n})_{i\neq j},\\
&f_{-L_i-L_j}=(e_{j+n,i}-e_{i+n,j})_{i<j}.
\end{align*}
 With these notations
for $t\in \RR$, $a=(a_1,...,a_{m-n})\in \RR^{m-n}$ we define
foliations $F_{r}$ for $r=\pm L_i\pm L_j,i\neq j$ and $F_{\rho}$ for
$\rho=\pm L_i$ for which the leaf through $x$
\begin{align}\label{for:2}
&F_{r}(x)=\exp (tf_{r})x,\qquad F_{\rho}(x)=\bigl(\prod_j\exp
(a_jf^{j}_{\rho})\bigl)x
\end{align}
consists of all left multiples of $x$   by matrices of the form
$F_{r}(t)$ or $F_{\rho}(a)$.

The foliations $F_r$ and $F_{\rho}$ are invariant under
$\alpha_0$. In fact, let
$\mathfrak{t}=(t_1,t_2,\dots,t_n)\in\mathbb{D_+},$ for $\forall
t\in\RR$ we have Lie bracket relations
$$[\mathfrak{t}, tf_r]=r(\mathfrak{t})tf_r,\qquad [\mathfrak{t},tf_\rho^\ell]=\alpha(\mathfrak{t})tf_\rho^\ell$$
where $r(\mathfrak{t})=\pm t_i\pm t_j$ if $r=\pm L_i\pm L_j$;
$\rho(\mathfrak{t})=\pm t_i$ if $\rho=\pm L_i$.

Using the basic identity for any square matrices $X,Y$:
$$\exp X\exp
Y=\exp(e^sY)\exp X, \text{ if } [X,Y]=sY,$$ it follows
\begin{align}
&\alpha_0(\mathfrak{t})\bigl(\exp (tf_{r})\bigl)x=\exp (te^{r(\mathfrak{t})}f_{r})\alpha_0(\mathfrak{t})x,\\
&\alpha_0(\mathfrak{t})\bigl(\prod_j\exp
(a_jf^{j}_{\rho})\bigl)x=\bigl(\prod_j\exp
(a_je^{\rho(\mathfrak{t})}f^{j}_{\rho})\bigl)\alpha_0(\mathfrak{t})x
\end{align}
where $r(\mathfrak{t})=\pm t_i\pm t_j$ if $r=\pm L_i\pm L_j$;
$\rho(\mathfrak{t})=\pm t_i$ if $\rho=\pm L_i$. Hence the leaf
$F_r(x)$ is mapped into $F_r(\alpha_0(\mathfrak{t})x)$ and
$F_\rho(x)$ is mapped into $F_\rho(\alpha_0(\mathfrak{t})x)$.
Consequently the foliation $F_r$ and $F_\rho$ are contracted
(corr. expanded or neutral) under $\mathfrak{t}$ if
$r(\mathfrak{t})<0$ (corr. $r(\mathfrak{t})>0$ or
$r(\mathfrak{t})=0$). If the foliation $F_r$ and $F_\rho$ are
neutral under $\alpha_0(\mathfrak{t})$, it is in fact isometric
under $\alpha_0(\mathfrak{t})$. The leaves of the orbit foliation is
$\mathcal{O}(x) = \{\alpha_0(\mathfrak{t})x : \mathfrak{t}\in
\mathbb{D_+}\}$.

The tangent vectors to the leaves in \eqref{for:2} for various $r$ and
$\rho$ together with their length one Lie brackets form a basis of
the tangent space at every $x\in X$.

 Let $\Phi'$ be the root system of $SU(m,n)$ with respect to
$D_+$. The roots are $\pm L_i \pm L_j(i<j\leq n)$, whose dimensions
are 2 and $\pm 2 L_i(i\leq n)$ whose dimension is 1. Also the $\pm
L_i(i\leq n)$ are roots if $m\neq n$ with dimensions
 $2(m-n)$. The corresponding root spaces are
\begin{align*}
 &\mathfrak{g}_{L_i+L_j}=\RR(e_{i,j+n}-e_{j,i+n})_{i<j}\oplus\RR \textrm{i}(e_{i,j+n}+e_{j,i+n})_{i<j},
 \\
 &\mathfrak{g}_{L_i-L_j}=\RR(e_{i,j}-e_{j+n,i+n})_{i\neq j}\oplus\RR\textrm{i}(e_{i,j}+e_{j+n,i+n})_{i\neq j},\\
 &\mathfrak{g}_{-L_i-L_j}=\RR(e_{j+n,i}-e_{i+n,j})_{i<j}\oplus\RR\textrm{i}(e_{j+n,i}+e_{i+n,j})_{i<j},\\
 & \mathfrak{g}_{L_i}=\bigoplus_{\ell\leq m-n}\bigl(\RR(e_{i,2n+\ell}-e_{2n+\ell,i+n})\oplus\RR \textrm{i}(e_{i,2n+\ell}+e_{2n+\ell,i+n})\bigl),\\
 &\mathfrak{g}_{-L_i}=\bigoplus_{\ell\leq m-n}\bigl(\RR(e_{i+n,2n+\ell}-e_{2n+\ell,i})\oplus\RR
 \textrm{i}(e_{i+n,2n+\ell}+e_{2n+\ell,i})\bigl),\\
 &\mathfrak{g}_{-2L_i}=\RR\textrm{i} e_{i+n,i},\qquad
 \mathfrak{g}_{2L_i}=\RR\textrm{i} e_{i,i+n}.
\end{align*}

 For $z\in \CC$ and $t\in \RR$, let
\begin{alignat*}{3}
 f_{L_i+L_j}(z)&=(ze_{i,j+n}-\overline{z}e_{j,i+n})_{i<j}, \qquad
 &f_{L_i-L_j}(z)&=(ze_{i,j}-\overline{z}e_{j+n,i+n})_{i\neq j},\\
 f_{-L_i-L_j}(z)&=(ze_{j+n,i}-\overline{z}e_{i+n,j})_{i<j},\qquad & f_{L_i}^\ell(z)&=ze_{i,2n+\ell}-\overline{z}e_{2n+\ell,i+n},\\
 f_{-L_i}^\ell(z)&=ze_{i+n,2n+\ell}-\overline{z}e_{2n+\ell,i},\qquad &
 f_{2L_i}(t)&=t \textrm{i} e_{i,i+n},\\
 f_{-2L_i}(t)&=t \textrm{i} e_{i+n,i}.
\end{alignat*}

 With these notations, for $z\in \CC$, $t\in \RR$,
$a=(a_1,\dots,a_{m-n})\in \CC^{m-n}$ individual expanding and
contracting foliations are similarly given by $F_{r}$ for $r=\pm
L_i\pm L_j,i\neq j$ and $F_{\rho}$ for $\rho=\pm L_i$ for which
the leaf through $x$
\begin{align}\label{for:25}
F_{r}(z)=\exp \bigl(f_{r}(z)\bigl)x,\qquad
F_{\rho}(t,a)=\exp\bigl(f_{2\rho}(t)\bigl)\exp \bigl(\sum_j
f^{j}_{\rho}(a_j)\bigl)x
\end{align}
consists of all left multiples of $x$ by matrices of the form
$F_{r}(t)$ or $F_{\rho}(a)$.

The foliations $F_r$ and $F_{\rho}$ are invariant under
$\alpha_0$. In fact, let
$\mathfrak{t}=(t_1,t_2,\dots,t_n)\in\mathbb{D_+},$ for $\forall
z\in\CC$ and $t\in\RR$ we have Lie bracket relations
\begin{align*}
[\mathfrak{t}, f_r(z)]&=r(\mathfrak{t})f_r(z),
\qquad [\mathfrak{t},f_\rho^\ell(z)]=\alpha(\mathfrak{t})f_\rho^\ell(z),\\
[\mathfrak{t},f_{2\rho}(t)]&=2\rho(\mathfrak{t})f_{2\rho}(t)
\end{align*}
where $r(\mathfrak{t})=\pm t_i\pm t_j$ if $\rho=\pm L_i\pm L_j$;
$\rho(\mathfrak{t})=\pm t_i$ if $\rho=\pm L_i$.

Using the basic identity for any square matrices $X,Y$:
$$\exp X\exp
Y=\exp(e^sY)\exp X, \text{ if } [X,Y]=sY,$$ it follows
\begin{align}
&\alpha_0(\mathfrak{t})\exp \bigl(f_{r}(z)\bigl)x=\exp \bigl(f_{r}(e^{r(\mathfrak{t})}z)\bigl)\alpha_0(\mathfrak{t})x,\\
&\alpha_0(\mathfrak{t})\exp\bigl(f_{2\rho}(t)\bigl)\exp
\bigl(\sum_jf^{j}_{\rho}(a_j)\bigl)x\notag\\
&=\exp\bigl(f_{2\rho}(e^{2\rho(\mathfrak{t})}t)\bigl)\exp
\bigl(\sum_j
f^{j}_{\rho}(e^{\rho(\mathfrak{t})}a_j)\bigl)\alpha_0(\mathfrak{t})x
\end{align}
where $r(\mathfrak{t})=\pm t_i\pm t_j$ if $r=\pm L_i\pm L_j$;
$\rho(\mathfrak{t})=\pm t_i$ if $\rho=\pm L_i$. Hence the leaf
$F_r(x)$ is mapped into $F_r(\alpha_0(\mathfrak{t})x)$ and
$F_\rho(x)$ is mapped into $F_\rho(\alpha_0(\mathfrak{t})x)$.
Consequently the foliation $F_r$ and $F_\rho$ are contracted (corr. expanded or
neutral) under $\mathfrak{t}$ if $r(\mathfrak{t})<0$ (corr.
$r(\mathfrak{t})>0$ or $r(\mathfrak{t})=0$). If the foliation $F_r$
and $F_\rho$ are neutral under $\alpha_0(\mathfrak{t})$, they are
in fact isometric under $\alpha_0(\mathfrak{t})$. The leaves of the
orbit foliation are  $\mathcal{O}(x) =
\{\alpha_0(\mathfrak{t})x : \mathfrak{t}\in \mathbb{D_+}\}$.

The tangent vectors to the leaves in \eqref{for:25} for various $r$
and $\rho$ together with their length one Lie   brackets form a
basis of the tangent space at every $x\in X$.

If $\mathbb{P}$ is a $2$-plane in general position then the
foliations $F_r$ and $F_\rho$ are also Lyapunov foliations for
$\alpha_{0,\mathbb{P}}$. The leaves  of  $F_r$ and $F_\rho$ are
intersections of the leaves of stable manifolds of the action by
different elements of $\mathbb{P}$. The same holds for the action by
any regular lattice in $\mathbb{P}$ and thus for any generic
restriction $\alpha_{0,G}$. The neutral foliation for a generic
restriction $\alpha_{0,G}$ will be denoted by $\mathcal{N}_0$.
\begin{remark}
If $m\geq n+1$, neutral foliation of the full split Cartan action,
as well as any generic restriction contains not only the orbit
foliation, but also compact part of the centralizer of the maximal
split Cartan subgroup. In fact, the neutral foliation of the
Cartan action is given by
$$\mathcal{N}_0(x)=\{Y_X\cdot x: x\in X\}.$$
\end{remark}

The above discussion can be summarized as follows.
\begin{proposition}
(1) Non-zero Lyapunov exponents for the full Cartan action on
$SO^+(m,n)/\Gamma$ are $\pm t_i\pm t_j$(each has multiplicity 1)and
$\pm t_i$(each has multiplicity $m-n$) where $i\neq j$ and $1\leq
i,j\leq n$. Zero Lyapunov exponent comes from the  neutral foliation and has
multiplicity $n+\frac{(m-n-1)(m-n)}{2}$. Consequently any matrix
$d\in D_+$ whose first $n$ diagonal entries are pairwise different
acts normally hyperbolically on $SO^+(m,n)/\Gamma$ with respect to
the neutral foliation and hence is only partially hyperbolic.

(2) If  $m\geq n+1$, non-zero Lyapunov exponents for the full Cartan
action on $SU(m,n)/\Gamma$ are $\pm t_i\pm t_j$(each has
multiplicity 2) and $\pm t_i$(each has multiplicity $2m$) where
$i\neq j$ and $1\leq i,j\leq n$; if $m=n$, non-zero Lyapunov
exponents are $\pm t_i\pm t_j$, $i\neq j$(each has multiplicity 2)
and $\pm 2t_i$(each has multiplicity 1) where $i\neq j$ and $1\leq
i,j\leq n$. Zero Lyapunov exponent comes from the neutral foliation
and has multiplicity $2n+(m-n)^2-1$. Any matrix $d\in D_+$ whose
first $n$ diagonal entries are pairwise different acts normally
hyperbolically on $SU(m,n)/\Gamma$ with respect to the neutral
foliation.
\end{proposition}
\subsection{ Generating relations and Steinberg symbols}\label{relations-Steinberg}  In this section
we state two theorems which play a crucial role in proofs of Theorem
\ref{th:5}. The proof of those theorems are given in Section
\ref{sec:12}--7 which comprise the algebraic part of the paper.

We use  notation set in Section
\ref{sec:7}. Since $\RR$ is embedded in $\RR^{m-n}$ in a obvious
way, there is no confusion if we write $F_{r}(t,0,\dots,0)=F_{r}(t)$
for $r=\pm L_i\pm L_j$. On the other hand, if we write $F_{r}(a)$
where $a\in\RR^{m-n}$, then $a=(a_1,0,\dots,0)$ for some
$a_1\in\RR$.
\begin{theorem}\label{th:3}
$SO^+(m,n)$, $3\leq n\leq m$ is generated by  $F_{r}(a)$, where
$r=\pm L_i\pm L_j, \pm L_i$, $0\leq i\neq j\leq n$ and $a\in
\RR^{m-n}$ subject to the relations:
\begin{align}
&F_r(a)F_r(b)=F_r (a+b),\label{for:26}\\
&[F_{r} (a), F_p (b)]=\prod_{ir+jp\in \Phi, i,j>0}
F_{ir+jp}(g_{ijpr}(a,b)), r+p\neq 0\label{for:27}\\
&[F_{r} (a), F_p (b)]=\emph{id}, \qquad 0\neq r+p\notin
\Phi,\label{for:5}
\end{align}
here $a,b\in\RR^{m-n}$ and $g_{ijpr}$ are functions of $a,b$
depending only on the structure of $SO^+(m,n)$;

\begin{align}
&h_{L_1-L_2}(t)h_{L_1-L_2}(s)=h_{L_1-L_2}(ts),\label{for:6}
\end{align}
where
$h_{L_1-L_2}(t)=F_{L_1-L_2}(t)F_{L_2-L_1}(-t^{-1})F_{L_1-L_2}(t)F_{L_1-L_2}(-1)F_{L_2-L_1}(1)F_{L_1-L_2}(-1)$
for each $t\in \RR^{*}$;
\begin{align}
&h_{L_1-L_2}(-1)h_{L_1+L_2}(-1)=\emph{id},\label{for:24}
\end{align}
where
$h_{L_1+L_2}(t)=F_{L_1+L_2}(t)F_{-L_1-L_2}(-t^{-1})F_{L_1+L_2}(t)F_{L_1+L_2}(-1)F_{-L_1-L_2}(1)F_{L_1+L_2}(-1)$
for each $t\in \RR^{*}$;
\begin{align}\label{for:7}
&h^1_{L_n}(\sqrt{2}a,\sqrt{2}b)h^{1}_{L_n}(\sqrt{2}c,\sqrt{2}d)=h^{1}_{L_n}\bigl(\sqrt{2}(ac-bd),\sqrt{2}(ad+bc)\bigl),
\end{align}
where
\begin{align*}
&h^1_{L_n}(\sqrt{2}a,\sqrt{2}b)=F^{1}_{L_n}(\sqrt{2}a)F^{2}_{L_n}(\sqrt{2}b)F^{1}_{
-L_n}(\sqrt{2}a)F^{2}_{
-L_n}(\sqrt{2}b)F^{1}_{L_n}(\sqrt{2}a)F^{2}_{L_n}(\sqrt{2}b)\notag\\
&\cdot
F^{1}_{L_n}(-\sqrt{2})F^{1}_{-L_n}(-\sqrt{2})F^{1}_{L_n}(-\sqrt{2})
\end{align*}
for each $(a,b)\in S^1$.

If $n\leq m\leq n+1$, there is no relation \ref{for:7}.
\end{theorem}
Now we consider the  group $SU(m,n)$.  We write
$F_{r}(0,z,0,\dots,0)=F_{r}(z)$ for $r=\pm L_i\pm L_j$ where
$z\in\CC$. On the other hand, if we write $F_{r}(a)$ where
$a\in\RR\times\CC^{m-n}$, then $a=(0,a_1,0,\dots,0)$ for some
$a_1\in\CC$.
 \begin{theorem}\label{th:4}
Let $\Phi$ be the root system of $SU(m,n)$ $3\leq n\leq m$. Then
$SU(m,n)$ is generated by $F_{r}(a)$ where $r=\pm L_i\pm L_j,\pm
L_i$, $i\neq j$ and $a\in\RR\times\CC^{m-n}$ subject to the
relations:
\begin{align}
&F_r (a)F_r(b)=F_r (a+b)\label{for:3}, \\
&[F_{r} (a), F_p (b)]=\prod_{\substack{ i,j>0\\ir+jp\in
\Phi}}F_{ir+jp}(N_{r,p,i,j}(a,b)),r+p\neq 0\label{for:4}\\
&[F_{r} (a), F_p (b)]=e, \qquad 0\neq r+p\notin \Phi,\label{for:21}
\end{align}
here $a,b\in\CC^{m-n}$, $N_{r,p,i,j}$ are funtions depend only on
the structure of $SU(m,n)$;

 and the following relations:
\begin{align}\label{for:10}
h_{L_1-L_2}(z_1)h_{L_1-L_2}(z_2)=h_{L_1-L_2}(z_1z_2)
\end{align}
where
\begin{align*}
h_{L_1-L_2}(z)&=F_{L_1-L_2}(z)F_{L_2-L_1}(-z^{-1})F_{L_1-L_2}(z)\\
&\cdot F_{L_1-L_2}(-1)F_{L_2-L_1}(1)F_{L_1-L_2}(-1)
\end{align*}
for each $z\in \CC^*$;
\begin{align}\label{for:22}
h_{2L_{n}}(-1)h_{2L_{n}}(-1)=\emph{id},
\end{align}
where
\begin{align*}
h_{2L_{n}}(-1)&=\bigl(F_{L_{n}}(-1,0,\dots,0)F_{-L_{n}}(-1,0,\dots,0)F_{L_{n}}(-1,0,\dots,0)\bigl)^2;
\end{align*}
and
\begin{align}
h^1_{L_n}(\sqrt{2}a,\sqrt{2}b)h^1_{L_n}(\sqrt{2}c,\sqrt{2}d)&=h^1_{L_n}\bigl(\sqrt{2}(ac-bd),\sqrt{2}(ad+bc)\bigl)\label{for:17}\\
h^1_{L_n}(\sqrt{2}a,\sqrt{2}b
\emph{i})h^1_{L_n}\bigl(\sqrt{2}c,\sqrt{2}d
\emph{i})&=h^1_{L_n}(\sqrt{2}(ac-bd),\sqrt{2}(ad+bc)
\emph{i}\bigl)\label{for:18}
\end{align}
for each $(a,b),(c,d)\in S^1$ where
\begin{align*}
h^1_{L_n}(\sqrt{2}z,\sqrt{2}w)&=F_{L_{n}}(0,\sqrt{2}z,\sqrt{2}w,0,\dots,0)F_{-L_{n}}(0,\sqrt{2}z,\sqrt{2}w,0,\dots,0)\\
&\cdot F_{L_{n}}(0,\sqrt{2}z,\sqrt{2}w,0,\dots,0)F_{L_{n}}(0,-\sqrt{2},0,\dots,0)\\
&\cdot F_{-L_{n}}(0,-\sqrt{2},0,\dots,0)\cdot
F_{L_{n}}(0,-\sqrt{2},0,\dots,0)
\end{align*}
for $(z,w)\in\CC^2$ with $\abs{z}^2+\abs{w}^2=1$.

If $n\leq m\leq n+1$ there are no relations \ref{for:17}  and
\ref{for:18}.
\end{theorem}

Relations \ref{for:26}--\ref{for:24} in Theorem \ref{th:3} and
\ref{for:3}--\ref{for:22} in Theorem \ref{th:4} are similar to those
in split groups \cite{Steinberg}; while relation \ref{for:17} and
\ref{for:18} are dealing with ``rotations'' inside the compact part
of the centralizer of the maximal split Cartan subgroups.

Any bi-multiplicative map $c:\mathbb{K}^*\times
\mathbb{K}^*\rightarrow B$ into an abelian group $B$ satisfying
$c(t,1-t) = 1_B$ is called a {\em Steinberg symbol} on the field
$\mathbb{K}$. We will use the following result about continuous
Steinberg symbols for the field $\RR$ and $\CC$ \cite{Milnor}:
\begin{theorem}[Milnor]\label{th:6}
a)Every continuous Steinberg symbol on the field $\CC$ of complex
numbers is trivial.\\
b) If $c(t,s)$ is a continuous Steinberg symbol on the field $\RR$,
then $c(t,s) = 1$ if $s$ or $t$ are positive, and $c(t, s) =
c(-1,-1)$ has order at most 2 if $s$ and $t$ are both negative.
\end{theorem}
The following Lemma treats a case which occurs in the proof of
Theorem \ref{sec:10} when one can study reducible classes within
homotopy classes of Lyapunov-cycles.
\begin{lemma}\label{le:38}
Let $L$ be an irreducible lattice in $\widetilde{SO}^+(m,n)$(corr.
$\widetilde{SU}(m,n)$) where $\widetilde{SO}^+(m,n)$ is the
universal cover of $SO^+(m,n)$(corr. $SU(m,n)$. Then for any
homomorphism from $L$ to $Y_X$, the order of image $h(L)$ is bounded
by a number only dependent on $m-n$.
\end{lemma}
\begin{proof}  We first consider the case of $\widetilde{SO}^+(m,n)$. Let $h$
be a homomorphism from  $L$ to $Y_X$. We decompose $h=(h_1,h_2)$
where $h_1$ maps $L$ to $\RR^{n}$ and $h_2$ maps $L$ to $SO(m-n)$.
By the Margulis Normal Subgroup Theorem \cite[4' Theorem]{Margulis},
$h_1$ is trivial. Thus $h$ can be considered as a homomorphism from
$L$ to $SO(m-n)$. By \cite[(3)Theorem]{Margulis}, Zariski closure of
$\overline{h(L)}$ is a semisimple $\QQ$-algebraic group. We first
show that $h(L)$ is finite. Suppose it is not finite. Since
$\overline{h(L)}/\overline{h(L)}^0$ is finite, we can assume
$\overline{h(L)}$ is connected. Then $\overline{h(L)}$ decomposes
uniquely into (up to permutation of the factors) a direct product of
$\QQ$-almost simple algebraic linear groups. We can assume
$\overline{h(L)}$ is almost simple. Compose $h$ with a Galois
automorphism $\sigma$ of $\CC$ over $\QQ$ to matrix coefficients of
elements from $h(L)$, then $\sigma h(L)$ is a non-compact subgroup
of $\overline{h(L)}=\sigma\overline{h(L)}$. By finiteness of
$Z(\overline{h(L)})$, we can assume $(\sigma h)'$ is from $L$ to
$\overline{h(L)}/Z(\overline{h(L)})$. By Margulis lattice
superrigidity Theorem \cite[2'Theorem]{Margulis}, $(\sigma h)'$ can
be extended to a continuous homomorphism $\widetilde{h}$ from
$\widetilde{SO}^+(m,n)$ to $\overline{h(L)}/Z(\overline{h(L)})$. We
can assume $\widetilde{h}$ is from $\widetilde{SO}^+(m,n)$ to
$\overline{h(L)}$. By simplicity of $\widetilde{SO}^+(m,n)$(in the
meaning of Lie groups), $\ker(\tilde{h})\subseteq
Z\bigl(\widetilde{SO}^+(m,n)\bigl)$ or
$\ker(\tilde{h})=\widetilde{SO}^+(m,n)$. While
$\ker(\tilde{h})=\widetilde{SO}^+(m,n)$ contradicts the infiniteness
of $h(L)$. Notice $\overline{h(L)}\subseteq SO(m-n,\CC)$. Thus we
get a continues isomorphism from
$\widetilde{SO}^+(m,n)/\ker(\tilde{h})$ to its image, which
contradicts the fact the the maximal connected compact subgroup in
$\widetilde{SO}^+(m,n)$ is the universal cover of $SO(n)\times
SO(m)$ while in $SO(m-n,\CC)$ is $SO(m-n)$. Hence we proved that
$h(L)$ is finite.

Next, we show this number is bounded independence of the
homomorphism. By Jordan's theorem which claims that any finite group
$G\subseteq GL(\ell,\CC)$ contains a normal abelian subgroup whose
index is at most $j(\ell)$, we let the biggest normal abelian
subgroup in $h(L)$ be $A$. Consider the restriction of $h$ from
$h^{-1}(A)$ to $A$. The index of $[L:h^{-1}(A)]$ is bounded by
$j(m-n)$. There are only finitely many sublattices $L'$ in $L$ with
the index smaller than $j(m-n)$, the arguments go as follows: first,
every subgroup of finite index in $L$ contains a normal subgroup of
$L$ with index dividing $j(m-n)!$. So, it suffices to check that the
number of normal subgroups of index smaller than $j(m-n)$ is finite.
Such normal subgroups are exactly the kernels of (surjective)
homomorphisms of $L$ into a finite group of order smaller than
$j(m-n)$. Recall that the set of finite groups of order small than
$j(m-n)$ is finite (up to an isomorphism). On the other hand, since
$L$ is finitely generated (as a group)\cite{Margulis}, the set of
homomorphisms of $L$ into a given finite group is finite.

Since $[h^{-1}(A):[h^{-1}(A),h^{-1}(A)]]$ is finite\cite[4'
Theorem]{Margulis}, it is bounded by a number $i(m-n)$ which depends
only on $m-n$ by above analysis. The order of $A$ is bounded by
$i(m-n)$ by its abelian property. Hence the order of $h(L)$ is
bounded by $i(m-n)j(m-n)$.

For $\widetilde{SU}(m,n)$,  Notice here $h=(h_1,h_2,h_3)$ where
$h_1$ maps $L$ to $\RR^{n}$, $h_2$ maps $L$ to $\mathbb{T}^{n-1}$
and $h_3$ maps $L$ to $U(m-n)$. $h_1$ is trivial by Margulis Normal
Subgroup Theorem; order of $h_2(L)$ is bounded by $[L,[L,L]]$. To
prove finiteness of $h_3(L)$, we can assume it is from $L$ to
$SU(m-n)$. Similar arguments hold to get a continuous isomorphism
from $\widetilde{SU}(m-n)/D$ to a subgroup inside $SL(m-n,\CC)$,
since the real locus of $SL(m-n,\CC)$ is $SU(m-n)$, where $D$ is
inside the center of $\widetilde{SU}(m-n)$. Thus
$[Z(\widetilde{SU}(m-n)):D]$ is finite by the fact that every simple
matrix group has finite center, which contradicts the fact that the
maximal connected compact subgroups in $\widetilde{SU}(m,n)/D$ is a
finite lift of $S\bigl(U(m)\times U(n)\bigl)$, while in
$SL(m-n,\CC)$ is $SU(m-n)$. The next step to prove the uniform bound
of order of $h_3(L)$ is exactly the same as previous case.
\end{proof}

\subsection{Proof of Theorem \ref{th:5}} Notice $\alpha_{0,G}$ can be
lifted to a $G$-action on $\widetilde{SO}^+(m,n)$(corr.
$\widetilde{SU}(m,n)$) where $\widetilde{SO}^+(m,n)$ is the
universal cover of $SO^+(m,n)$(corr. $\widetilde{SU}(m,n)$ is the
universal cover of $SU(m,n)$). We denote the new action by
$\widetilde{\alpha}_{0,G}$ and the projection from
$\widetilde{SO}^+(m,n)$(corr. $\widetilde{SU}(m,n)$) to
$SO^+(m,n)$(corr. $SU(m,n)$) by $p$. We proceed in exactly the same
manner as in \cite{Damjanovic1}.

At first we show the cocycle rigidity for H\"older cocycles.   The
invariant foliations that we considered in section \ref{sec:7} are
$F_{r}$ and $F_\rho$ where $r=\pm L_i\pm L_j, i\neq j$ and $\rho=\pm
L_i$. Notice that those foliations are smooth and their Lie brackets
at length one generate the whole tangent space. This implies that
this system of foliations is locally $1/2-H\ddot{o}lder$ transitive
([\cite{Kononenko}, Section 4, Proposition 1]). Hence the lifted
foliations which we still denote by $F_{r}$ and $F_\rho$ without
confusion are locally transitive on the universal covering spaces.
Every such cycle represents a relation in the group. The word
represented by this cycle can be written as a product of conjugates
of basic relations in Theorem \ref{th:3} and \ref{th:4} that can be
lifted to closed cycles in the universal covering spaces.

 At first we consider $\widetilde{SO}^+(m,n)$ which is a 4-fold covering space of
$SO^+(m,n)$ if $m\geq n\geq 3$. Since $F_{r}$ and $F_\rho$ are
Lyapunov foliations for the full Cartan action and therefore for any
generic restriction $\alpha_{0,\mathbb{P}}$(see Propostion 7 in
\cite{Damjanovic1}), which implies these relations of the type
\eqref{for:26}, \eqref{for:27} and \eqref{for:5} are contained in a
leaf of the stable manifold for some element of the action.

For relation \eqref{for:24}, in proof of Theorem \eqref{th:3}, we
showed that  if doubled, it is lifted to a closed cycle in the universal
cover and afteran  allowable substitution, it is  reducible.

For relation \eqref{for:6} follow exactly the same way as in
Milnor¡¯s proof in [\cite{Milnor}, Theorem A1] or in
\cite{Damjanovic2}, we can show that if doubled, they are
contractble and in $\mathcal{A}\mathcal{S}_{F}(\alpha)$(defined in
Definition \ref{def:1}), thus these doubled relation is lifted to
closed cycles in $\widetilde{SO}^+(m,n)$ and after an allowable
substitution, it is reducible.

 For relations \eqref{for:7}, since it is a symbol defined on $S^1$,
follow exactly the same way as in Milnor¡¯s proof in [\cite{Milnor},
Theorem A1], we can show they are contractble and in
$\mathcal{A}\mathcal{S}_{F}(\alpha)$, thus these relations are
lifted to closed cycles in $\widetilde{SO}^+(m,n)$ and after
allowable substitution, they are reducible.

Now we consider $\widetilde{SU}(m,n)$. Follow exactly the same way
as in the proof of the previous case,  we are though all the
relations except \eqref{for:22}. Notice
$h_{2L_{n}}(-1)=$diag($1,\dots,\underset{n}-1,1,\dots,\underset{2n}-1,\dots,1$),
thus homotopy classes of
$\bigl(h_{2L_{n}}(-1)h_{2L_{n}}(-1)\bigl)^{k}(k\in\ZZ)$ generate the
fundermental group of $SU(m,n)$ which is isomorphic to $\ZZ$. Hence
we don't need to consider this relation in $\widetilde{SU}(m,n)$.

Finally, to cancel conjugations one notices that canceling
$F_r(t)F_r(t)^{-1}=\textrm{id}$ or
$F_\rho(a)F_\rho(a)^{-1}=\textrm{id}$ are also an allowed
substitution and each conjugation can be canceled inductively using
that.

Thus, the value of the periodic cycle functional for any H\"older
cocycle $\beta$ depends only on the element of $p^{-1}(\Gamma)$ this
cycle represents.  Notice $p^{-1}(\Gamma)$ is an irreducible lattice
also. Furthermore, these values provide a homomorphism from
$p^{-1}(\Gamma)$ to $Y_X$. By Lemma \ref{le:38}, orders of images of
any homomorphism are bounded by a number depending only on $m-n$, 
which means no non-trivial homomorphism or it contradicts the
smallness of the cocyle.

 Hence all periodic cycle
functionals vanish on $\beta$. Now Proposition \ref{le:8} implies
that $\beta$ is cohomologous to a constant cocycle via a
H$\ddot{o}$lder transfer function.

Now consider the case of $C^{\infty}$ cocycles.
Notice that  the transfer function $H$ constructed using periodic cycle
functionals is $C^\infty$ along the  stable foliations  of  various elements of the action. Now a general result
stating that in case the smooth distributions along with their Lie
brackets generate the tangent space at any point of a manifold a
function smooth along corresponding foliations is necessarily smooth
(see \cite{Spatzier2} for a detailed discussion and references to
proofs), implies that the transfer map $H$ is $C^\infty$.

\section{Proof of Theorems \ref{th:1} and \ref{th:2}}\label{sec:8}
 The neutral foliation for a
generic restriction $\alpha_{0,G}$ is a smooth foliation, we may use
the Hirsch-Pugh-Shub structural stability theorem [\cite{shub},
Chapter 6]. Namely if $\widetilde{\alpha}_G$ is a sufficiently
$C^1$-small perturbation of $\alpha_{0,G}$ then for all elements
$a\in A$ which are regular for $\alpha_{0,G}$ and sufficiently away
from non-regular ones (denote this set by $\overline{A}$) are also
regular for $\widetilde{\alpha}_G$. The central distribution is the
same for any $a\in \overline{A}$ and is uniquely integrable to an
$\widetilde{\alpha}_G(a,\cdot)$-invariant foliation which we denote
by $\mathcal{N}$. Moreover, there is a H\"older homeomorphism
$\widetilde{h}$ of $X$, $C^0$ close to the $id_X$, which maps leaves
of $\mathcal{N}_0$ to leaves of $\mathcal{N}$:
$\widetilde{h}\mathcal{N}_0 = \mathcal{N}$. This homeomorphism is
uniquely defined in the transverse direction, i.e. up to a
homeomorphism preserving $\mathcal{N}$. Furthermore, $\widetilde{h}$
can be chosen smooth and $C^1$ close to the identity along the
leaves of $\mathcal{N}_0$ although we will not use the latter fact.
Clearly the leaves of the foliation $\mathcal{N}_0$ are preserved by
every $a\in \overline{A}$. The action $\alpha_G$ is H\"older
but it is smooth and $C^1$-close to $\alpha_{0,G}$ along the leaves
of the neutral foliation $\mathcal{N}_0$.

 Let us
define an action $\alpha_G$ of $G$ on $X$ as the conjugate of
$\widetilde{\alpha}_G$ by the map $\widetilde{h}$ obtained from the
Hirsch-Pugh- Shub stability theorem: $$\alpha_G :=
\widetilde{h}^{-1}\circ \widetilde{\alpha}_G\circ \widetilde{h}$$
Since the action  $\alpha_G$ is a $C^0$ small perturbation of
$\alpha_{0,G}$ along the leaves of the neutral foliation of
$\alpha_{0,G}$ whose leaves are $\{Y_X\cdot x : x\in X\}$, we have
that $\alpha_G$ is given by a map $\beta: (\ZZ^k\times
\RR^\ell)\times X\rightarrow Y $ by
\begin{align}\label{for:12}
\alpha_G(a, x) = \beta(a, x) \cdot \alpha_{0,G}(a, x)
\end{align}
 for $a\in \ZZ^k\times \RR^\ell$ and $x\in X$.
Notice that since $\alpha_G$ is a small perturbation of the action
by left translations $\alpha_{0,G}$, it can be lifted to a
$G$-action $\overline{\alpha}_G$ on
$\widetilde{X}=\widetilde{SO}^+(m,n)$(corr. $\widetilde{SU}(m,n)$)
commuting with the right $p^{-1}(\Gamma)$ action on
$\widetilde{SO}^+(m,n)$ (corr. $\widetilde{SU}(m,n)$), and $\beta$
is lifted to a cocycle $\overline{\beta}$ over $\overline{\alpha}_G$
(for more details see [\cite{ Margulis2}, example 2.3]). In
particular we have:
$$\overline{\beta}(ab,x)=\overline{\beta}(a,\overline{\alpha}_G(b,x))\overline{\beta}(b,x).$$
Let $U:U_1,\dots,U_r$ denote the invariant unipotent foliations for
the lifted action $\overline{\alpha}_{0,G}$ of $\alpha_{0,G}$ on
$\widetilde{X}$ which projects to invariant Lyapunov foliations
 for $\alpha_{0,G}$;
and let $T:T_1,\dots,T_r$ denote invariant Lyapunov foliations for
lifted  $\overline{\alpha}_G$ which projects to invariant Lyapunov
foliations 
for $\alpha_G$. Notice that the latter foliations  have only H\"older leaves
but we are justified in calling them Lyapunov foliations since they are images of  Lyapunov foliations  for a smooth perturbed action under a H\"older conjugacy.
 Denote  the neutral foliation $\mathcal{N}_0$ on the
covering space by $N_0$. An immediate corollary of the result of
Brin and Pesin \cite{pesin} on persistence of local transitivity of
stable and unstable foliations of a partially hyperbolic
diffeomorphisms and the fact that the collection of homogeneous
Lyapunov foliations $U:U_1,\dots,U_r$ is locally transitive and
$T:T_1,\dots,T_r$ is transitive and they are leafwise $C^0$ close.
Following the proof line closely with only trivial modifications
from those of [Section 6.2, 6.2 and 6.4 \cite{Damjanovic2}], and
[Section 5.3,5.4, \cite{Damjanovic3}], we can show $U$-cycles and
$T$-cycles project to each other along the neutral foliations
(precise definitions are in [Section 6.2,\cite{Damjanovic2}]), which
implies:
\begin{proposition}\label{po:2}
The lifted cocycle for the perturbed action $\overline{\alpha}_G$ is
cohomologous to a constant cocycle.
\end{proposition}
By Proposition \ref{po:2}, the value of the periodic cycle
functional for $H\tilde{o}lder$ cocycle $\beta$ over
$\widetilde{\alpha}_G$ or its $H\tilde{o}lder$ conjugate $\alpha_G$
 depends only on the element of $p^{-1}(\Gamma)$ this cycle
represents. Using the same trick as in proof of Theorem \ref{th:5},
we can show every homomorphism from $p^{-1}(\Gamma)$ to $Y_X$ is
trivial. Thus we proved Theorem \ref{th:2}.

Thus by Theorem \ref{th:2}, $\beta$ is cohomologous to a small
constant cocycle $s:\ZZ^k\times \RR^\ell\rightarrow Y_X$ via a
continuous transfer map $H:X\rightarrow Y_X$ which can be chosen
close to identity in $C^0$ topology if the perturbation
$\widetilde{\alpha}_G$ small in $C^2$ topology.

Let us consider the map $h'(x):= H^{-1}(x)\cdot x$. We have from the
cocycle equation \ref{for:12} and the cohomology equation
\ref{for:13}
$$h'({\alpha}_G(a, x)) = \alpha_{0,\widetilde{G}}(a, h'(x))$$ where $\alpha_{0,\widetilde{G}}(a, x) := i(a)\cdot x$,
where $i(a) := s(a)i_0(a), a\in A$ and $i_0$ is as in \ref{for:15}.
Since the map $h'$ is $C^0$ close to the identity it is surjective
and thus the action $\alpha_G$ is semi-conjugate to the standard
perturbation $\alpha_{0,\widetilde{G}}$ of $\alpha_{0,G}$, i.e.
$\alpha_{0,\widetilde{G}}$ is a factor of $\alpha_G$. It is enough
to prove that $h'$ is injective. By simple transitivity of
$U$-holonomy group and the fact that there is no non-trivial element
in $Y_X$ such that all its powers are small [Section 7.1
\cite{Damjanovic2}] we have:
\begin{proposition}\emph{(Section 6.1 \cite{Damjanovic2})}
The map $h'$ is a homeomorphism and hence provides a topological
conjugacy between $\alpha_G$ and $\alpha_{0,\widetilde{G}}$.
\end{proposition}
Now by letting $h:=h'\widetilde{h}^{-1}$ we have
$$h\circ \widetilde{\alpha}_Gh^{-1}=\alpha_{0,\widetilde{G}}$$
thus there is a topological conjugacy between $\widetilde{\alpha}_G$
and $\alpha_{0,\widetilde{G}}$. The smoothness of this homeomorphism
follows as in \cite{Spatzier}, \cite{Damjanovic2} or
\cite{Margulis2}, by the general Katok-Spatzier theory of
non-stationary normal forms for partially hyperbolic abelian
actions.

\subsection{Proof of Corollary \ref{cor:3}}
Proofs of Theorems \ref{th:1} and \ref{th:2} apply to this case with
minor changes when proving $T$-cycles are projected to $U$-cycles
when $SO^+(m,n)$ is not split. Similar to proof of Lemma
6.5\cite{Damjanovic2}, we need to show that a $T$-cycle at $x$
projected to a $U$-path starting at $x$ gives a $U$-cycle which is
either contractible or its fourth power, after adding a $U$-path of
bounded length which connects the 2 endpoints and closed up the
$U$-path. It is due to Theorem \ref{th:5} and the fact that
$h^j_{L_n}(\sqrt{2}a,\sqrt{2}b)(a,b)\in S^1$, $j\leq m-n-1$ generate
$M$.

\section{Schur multipliers of non-split groups}\label{sec:12}

\subsection{Preliminaries and notations from K-theory}\label{sec:2} In this part, we
follow nations and quote conclusions without proof fairly close to
\cite{Deodhar}. Let $k$ be any arbitrary field. Let $\Omega$ denote
its algebraic closure in a "universal domain." Let $G\hookrightarrow
\GL(n,\Omega)$ be a connected simply connected algebraic group which
is of $k$-rank $\geq 2$. We also assume that $G$ is absolutely
simple over $\Omega$. Let $G_{k}=G\bigcap GL(n,k)$ be the group of
$k$ rational points of $G$. For a subgroup $H$ of $G$, let $H_{k}$
denote set $H\bigcap G_{k}$. Let $\mathfrak{g}$ be the Lie algebra
of $G$, $S\subset G$ be the $k$-split torus. Let $\Phi$ be the
$k$-root system of $G$ with respect to $S$, and
$\mathfrak{g}_{\alpha}$ the corresponding root space. Let $\Phi^{+}$
be the set of positive roots and $\Delta$ the system  of
 simple roots with respect to $\Phi^{+}$. Define
\begin{align*}
\Phi_1=\{\alpha\in\Phi|\alpha/2\notin \Phi\}\qquad\text{ and
 }\qquad\Phi_2=\{\alpha\in\Phi|2\alpha\notin \Phi\}.
\end{align*}
For $\alpha\in \Phi$, let
$u^{\alpha}=\sum_{k>0}\mathfrak{g}_{k\alpha}$, and $U^{\alpha}$ the
corresponding algebraic subgroup of $G$. Let $U^+$  be the algebraic
subgroup of $G$ whose Lie algebra is
$\sum_{\alpha\in\Phi^{+}}\mathfrak{g}_\alpha$. $U^{-}$ is defined as
the subgroup corresponding to
$\sum_{-\alpha\in\Phi^{+}}\mathfrak{g}_\alpha$. For $\alpha\in
\Phi_1$, let $G^{\alpha}$ be the connected algebraic subgroup
generated by $U^{\alpha}$ and $U^{-\alpha}$. Let $Z(S)$ be the
centralizer of $S$ in $G$ and $N(S)$ the normalizer of $S$ in $G$
and $W_0=N(S)/Z(S)$ be the Weyl group. Let $W\subset N(S)_k$ be a
complete representatives. We also assume that $w_{\alpha}$ is so
chosen that for any $\alpha\in \Phi$, $w_{\alpha}\in N(S)\bigcap
G^{\alpha}_{k}$ and has order 2. Next, let $G^+_{k}$ be the group
generated by $k$-rational unipotent elements which belong to the
radical of a parabolic subgroup defined over $k$ in $G$. It is known
that for a wide class of $G$, $G^+_{k}=G_{k}$. Moreover, the only
proper normal subgroups of $G^+_{k}$ are central (and finite). Also,
$G^+_{k}=[G^+_{k},G^+_{k}]$. We start with a technical lemma whose
role will be clear from the subsequent development.
\begin{lemma}(\textbf{\emph{\textrm{The Chain Lemma}}})\label{le:2}
For $\alpha\in \Phi_1$, let $(e\neq) x\in U^\alpha_k$ be any
element. Then there exists elements $x_i\in U^\alpha_k$, $y_i\in
U^{-\alpha}_k(i\in \ZZ)$ such that:
\begin{itemize}
\item[1]. $\qquad x_0=x$.

\item[2]. $x_iy_ix_{i+1}=y_jx_{j+1}y_{j+1}$ $\forall i,j\in\ZZ$; we denote
this element by $w$.

\item[3]. The element $w$ belongs to $N(S)_k$ and ``acts" on  $S$ as the
reflection with respect to $\alpha$, i.e., $ww^{-1}_\alpha\in Z(S)_k
(w_\alpha\in W)$.

\item[4]. Given any $x_i$ or $y_i$ the remaining elements of the chain
$\{x_n,y_m\}$ are uniquely determined.
\end{itemize}
\end{lemma}
\begin{definition}\label{de:2}
We define $w_\alpha(x)$, for $(e\neq) x\in U^\alpha_k$ to be the
element $w$ in the chain lemma. Thus, $w_\alpha(x)\in N(S)_k$ and
``acts as the reflection with respect to $\alpha$". Further, we have
$w_\alpha(x)=w_\alpha(x_i)=w_{-\alpha}(y_j)$  and
$w_\alpha(x^{-1})=w_\alpha(x)^{-1}$ $\forall i,j\in \ZZ$. For
$\alpha\in \Phi_1$, $x, x_1$(both$\neq e$)$\in U^\alpha_k$, consider
the element $h_\alpha(x,x_1)= w_\alpha(x)w_\alpha(x_1)^{-1}$. It is
clear that $h_\alpha(x,x_1)\in Z(S)_k\bigcap G_k^+$. Let $H_{k}$ be
the subgroup generated by these elements.
\end{definition}

\begin{remark}\label{re:1}
If $\text{char} k =0,$  then we have the exponential map
$\exp:\mathfrak{g}\rightarrow G$. Let $X\in \mathfrak{g}_\alpha$ be
a $k$-rational(nilpotent) element. Then $x=\exp X\in U^\alpha_k$.
The chain associated with $x$ can also be obtained in the following
way: We have the Jacobson-Morosov theorem which asserts the
existence of an element $Y\in (\mathfrak{g}_{-\alpha})_k$ such that
$\{X,Y,[X,Y]\}$ span a three-dimension split Lie algebra over $k$.
It is then easy to prove that $x_i=x=\exp X, \forall i\in\ZZ$ and
$y_j=\exp(-Y), \forall j\in\ZZ$. If we let $y=$exp$(-Y)$, then we
denote $xyx$ by $w_\alpha(x)$.
\end{remark}
\subsection{Construction of universal central extension}\label{sec:1} Let $N$ be an abstract group.
A central extension of $N$ is a pair $(\pi,N')$ where $N'$ is a
group, $\pi$ is a homomorphism of $N'$ onto $N$ and
$\ker\pi\subseteq$(center of $N'$). A central extension  $(\pi,N')$
of $N$ is said to be universal if for any central extension $(\eta,
E')$ of $N$, there exists a unique homomorphism $\phi:N'\rightarrow
E'$ such that $\eta\circ\phi=\pi$. A necessary and sufficient
condition for $N$ to have a universal central extension is that $N=
[N,N]$. (For the proof of this and other elementary properties of a
universal central extension, one may refer to [\cite{Steinberg2},
Section 7].)

We now construct the universal central extension $(\textrm{u.c.e.})$
of the group $G^{+}_{k}$. (Such an extension exists since
$G^{+}_{k}=[G^{+}_{k}, G^{+}_{k}]$.) For $\alpha, \beta\in \Phi_1$
such that $\alpha\neq -\beta$, it is known that
$$[U^\alpha, U^\beta]\subset\prod_{\gamma=i\alpha+j\beta, i,j\geq 1}U^\gamma.$$
This clearly gives rise to relations $R_{\alpha,\beta}$ between
commutators of above form and elements which belongs to
$$\prod_{\gamma=i\alpha+j\beta, i,j\geq 1}U^\gamma.$$

Let $\widetilde{G}'=U^{+}_k*U^{-}_k$, the free product of groups.
Now for $\alpha, \beta\in \Phi, \alpha\neq -\beta$,
$R_{\alpha,\beta}$ has a natural meaning in $\widetilde{G}'$. We now
quotient $\widetilde{G}'$ by the relations $\{R_{\alpha,\beta}\}$ to
get a group we denote by $\widetilde{G}$. It is clear that there
exists a well-defined homomorphism $\pi_1:\widetilde{G}\rightarrow
G^{+}_{k}$. Further, it is clear that $\pi_1$ is surjective. We
write for $u\in U^{\alpha}_{k}$, the corresponding element in
$\widetilde{G}$ by $\tilde{x}_\alpha(u)$. Then
$\tilde{w}_\alpha(u)$, $\tilde{h}_\alpha(u,u_1)$ for $u,u_1\in
U^{\alpha}_{k}$ are obviously defined elements of $\widetilde{G}$.
\begin{lemma}
$(\widetilde{G},\pi_1)$ is a $(\emph{u.c.e.})$ of $G^{+}_{k}$
\end{lemma}

 Now we list the following lemma which is very important for the
sequel.
\begin{lemma}\label{le:5}
If $\alpha, \beta\in \Phi_1, (e\neq)u\in U^{\alpha}_{\RR}, (e\neq)v,
v_1\in U^{\beta}_{\RR}$, then
\begin{itemize}
\item[1]$\tilde{w}_{\alpha}(u)\tilde{x}_{\beta}(v)\tilde{w}_{\alpha}(u)^{-1}=\tilde{x}_{w_{\alpha}(\beta)}\bigl(w_{\alpha}(u)vw_{\alpha}(u)^{-1}\bigl)$
\item[2]$\tilde{w}_{\alpha}(u)\tilde{w}_{\beta}(v)\tilde{w}_{\alpha}(u)^{-1}=\tilde{w}_{w_{\alpha}(\beta)}\bigl(w_{\alpha}(u)x_{\beta}(v)w_{\alpha}(u)^{-1}\bigl)$
\item[3]$\tilde{w}_{\alpha}(u)\tilde{h}_{\beta}(v,v_1)\tilde{w}_{\alpha}(u)^{-1}$\\
$=\tilde{h}_{w_{\alpha}(\beta)}\bigl(w_{\alpha}(u)x_{\beta}(v)w_{\alpha}(u)^{-1},w_{\alpha}(u)x_{\beta}(v_1)w_{\alpha}(u)^{-1}\bigl)$
\item[4]$\tilde{h}_{\beta}(v,v_1)\tilde{x}_{\alpha}(u)\tilde{h}_{\beta}(v,v_1)^{-1}$\\
$=\tilde{x}_{\alpha}\bigl(h_{\beta}(v,v_1)x_{\alpha}(u)h_{\beta}(v,v_1)^{-1}\bigl)$.
\end{itemize}
\end{lemma}

\begin{lemma}\label{le:1} Let $\widetilde{N}$ be the subgroup of $G$ generated by $\{\tilde{w}_\alpha(u),
\alpha\in\Phi_1, (e\neq)u\in U_k^\alpha\}$. For $\alpha\in\Phi^+$,
denote by $\tilde{H}_{\alpha}$ the subgroup generated by
$\tilde{h}_{\alpha}(v,v_1)$, $(e\neq)v, v_1\in U^{\alpha}_k$. let
$\tilde{H}$ be he subgroup generated by $\{\tilde{H}_{\alpha},
\alpha\in\Phi_1\}$. Then
\begin{itemize}
\item[1]$\tilde{H}_{\alpha},\alpha\in\Phi_1$, is normal in
$\widetilde{H}$, and $\widetilde{H}$ is normal in $\widetilde{N}$.

\item[2]$\widetilde{H}$ normalizes each $\tilde{U}_k^\alpha$, and
hence $\tilde{U}_k^+$.

\item[3]$\tilde{H}=\prod_{\alpha\in \Delta}\tilde{H}_{\alpha}$

\item[4]$\ker(\pi_1)\subset \tilde{H}$.
\end{itemize}
\end{lemma}
\begin{remark}\label{re:4}
We now consider the condition under which $\tilde{h} =
\prod_{\alpha\in \Delta}\tilde{h}_{\alpha}$ is in the kernel of
$\pi_1$. Using the simple connectedness of $G$ over $\Omega^{-}$, it
is easy to see that $\tilde{h}\in\ker\pi_1$ iff $\tilde{h}_\alpha\in
\ker \pi_1 \forall \alpha\in \Delta$. In other words, the Schur
multiplier $\pi_1$ of $G_k$ is generated by $\{\pi_1\bigcap
\tilde{H}_{\alpha},\alpha\text{ simple}\}$. If $G$ is not simply
connected, we only have $\ker(\pi_1)\subset \tilde{H}$.
\end{remark}

\section{Generating relations of $SO^+(m,n)$}\label{sec:16}
\subsection{Basic settings for $SO^+(m,n)$}
In this part we study the generators of $SO^+(m,n)(m\geq n\geq 3)$.
We use notations as in Section \ref{sec:7} and Section \ref{sec:12}.
Explicitly, this is the case where $G=SO(m+n,\CC)$ defined by a
non-degenerate standard bilinear form of signature $(m,n)$.

We denote by $S$ the set of $(m+n)\times (m+n)$ diagonal matrices in
$G_\RR$ with lower-right $(m-n)\times(m-n)$ block identity. Let
$\Phi$ be the root system with respect to $S$. The roots are $\pm
L_i \pm L_j(i<j\leq n) $, whose dimensions are one and $\pm
L_i(1\leq i\leq n)$ are also roots if $m\geq n+1$ with dimensions
 $m-n$. We easily see $\Phi=\Phi_1$. If $m\geq n+1$, the set of positive roots $\Phi^{+}$ and the corresponding set of
 simple roots $\Delta$ are
 \begin{align*}
&\Phi^{+}=\{L_i-L_j\}_{i<j}\cup\{L_i+L_j\}_{i<j}\cup\{L_i\}_{i},\\
&\Delta=\{L_i-L_{i+1}\}_i\cup\{L_n\};
 \end{align*}
 if $m=n$, the set of positive roots $\Phi^{+}$ and the corresponding set of
 simple roots $\Delta$ are
 \begin{align*}
&\Phi^{+}=\{L_i-L_j\}_{i<j}\cup\{L_i+L_j\}_{i<j},\\
&\Delta=\{L_i-L_{i+1}\}_i\cup\{L_{n-1}+L_{n}\}.
 \end{align*}

With notations in Section \ref{sec:7} we have
\begin{align*}
&U^{r}_\RR=\{\exp (tf_{r})\mid t\in \RR\} \text{ for }r=\pm L_i\pm
L_j(i<j),\\
&U^{\alpha}_\RR=\{\prod_j\exp \left(a_jf^{j}_{\alpha}\right)\mid
a=(a_1,\dots,a_{m-n})\in \RR^{m-n}\}\text{ for }\alpha=\pm L_i.
\end{align*}
Correspondingly, for $t\in \RR$ and $a=(a_1,\dots,a_{m-n})\in
\RR^{m-n}$ we write
\begin{align*}
&x_r(t)=\exp (tf_{r})\in U^{r}_\RR\qquad\text{ for }r=\pm L_i\pm
L_j(i<j),\\
&x_\alpha(a)=\prod_j\exp (a_jf^{j}_{\alpha})\in
U^{\alpha}_\RR\qquad\text{ for }\alpha=\pm L_i.
\end{align*}
\subsection{``Chains'' in $SO^+(m,n)$} Our next step is to determine explicitly the ``chain" (cf. Lemma
\ref{le:2}) corresponding to the element $x_\alpha(a)(\neq e)\in
U^{\alpha}_\RR(\alpha=\pm L_i)$. For this, define
$f:\RR^{m-n}\backslash 0\rightarrow \RR^{m-n}\backslash 0$ by
$f(a)=\bigl(\frac{2a_1}{\sum a_i^{2}},\dots,\frac{2a_{m-n}}{\sum
a_i^{2}}\bigl)$ for $a=\left(a_1,\dots,a_{m-n}\right)\in
\RR^{m-n}\backslash 0$. With this notation, we have:
\begin{lemma}\label{le:11}
For $x_\alpha(a)(\neq e)\in U^{\alpha}_\RR(\alpha=\pm L_i)$, the
``chain" corresponding to it is given by
\begin{align*}
x_i=x_\alpha(a), i\in \ZZ;\qquad y_i=x_{-\alpha}(f(a)), i\in \ZZ.
\end{align*}
Denoting the element $w_\alpha(x_\alpha(a))$ by $w_\alpha(a)$, we
have
\begin{align*}
w_\alpha(a)=x_\alpha(a)x_{-\alpha}\left(f(a)\right)x_\alpha(a).
\end{align*}
\end{lemma}

\begin{proof} It is easy to check
$$\{\sum_j a_jf^{j}_\alpha, \sum_j
\big(-\frac{2a_j}{\sum a_j^{2}}f^{j}_{-\alpha}\big), [\sum_j
a_jf^{j}_\alpha, \sum_j \big(-\frac{2a_j}{\sum
a_j^{2}}f^{j}_{-\alpha}\big)]\}$$ span a three-dimensional Lie
algebra isomorphic to $SL_2(\RR)$.  By Remark \ref{re:1} we get the
conclusion.
\end{proof}

\begin{remark}\label{re:2}
Similar computations can be made for the other roots, such as $\pm
L_i\pm L_j(i<j)$. We record the results here:
\begin{align*}
w_{r}(t)=x_{r}(t)x_{-r}(-t^{-1})x_{r}(t),\qquad t\in \RR^*
\end{align*}
where
\begin{align*}
x_i=x_{r}(t)\forall i,\qquad y_i=x_{-r}(-t^{-1})\forall i.
\end{align*}
Correspondingly, we define
\begin{align*}
&h_r(t)=w_{r}(t)w_{r}(1)^{-1},\qquad t\in\RR^*,r=\pm L_i\pm L_j(i<j),\\
&h_\alpha(a,b)=w_{\alpha}(a)w_{\alpha}(b)^{-1},\qquad a,b\in
\RR^{m-n}\backslash 0, \alpha=\pm L_i.
\end{align*}
Let us write $p(\pi)$ the permutation matrix corresponding to the
permutation $\pi$, that is, the $i,j$ entry of $p(\pi)$ is $1$ if
$i=\pi(j)$ and zeros otherwise. Let $A,B,C,\cdots$ be square
matrices(not necessary the same size), we use
diag$\left(A_{j_1},B_{j_2},C_{j_3},\cdots\right)$ to denote the
$(m+n)\times(m+n)$ matrix that's constructed in the following way.
First, the matrix $A$ is placed as a block in
diag$\left(A_{j_1},B_{j_2},C_{j_3},\cdots\right)$ with its upper
left corner positioned at the $(j_1,j_1)$ entry. Matrices
$B,C,\cdots$ are placed similarly. Then we fill the rest of diagonal
blocks of diag$\left(A_{j_1},B_{j_2},C_{j_3},\cdots\right)$ with
identity matrices of suitable sizes and the off-diagonal blocks with
zero matrices. For example, let $A=\begin{pmatrix} 2 & 3 \\
-1 & 4 \\
  \end{pmatrix}$ and $B=(3)$. Then $\diag\{A_2, B_5\}$ is the
  following matrix
  \[
\begin{pmatrix}1 & 0 & 0 & 0 & 0 \\
0 & 2 & 3 & 0 & 0 \\
0  & -1 & 4 & 0& 0\\
0 & 0 & 0 & 1 & 0\\
0 & 0 & 0 & 0 & 3
\end{pmatrix}.
\]

With these notations we have:
\begin{align*}
w_{L_i-L_j}(t)=p(\pi)\text{diag}\left((-t^{-1})_i,t_j,(-t)_{i+n},(t^{-1})_{j+n}\right),
\qquad \text{ for }t\in \RR^*
\end{align*}
where $\pi$ only permutes $(i,j)$ and $(i+n,j+n)$ while fixes other
numbers.
\begin{align*}
w_{L_i+L_j}(t)=p(\pi)\text{diag}\left((-t^{-1})_i,(t^{-1})_j,(-t)_{i+n},t_{j+n}\right),\qquad
\text{ for }t\in \RR^*
\end{align*}
where $\pi$ only permutes $(i,j+n)$ and $(j,i+n)$ while fixes other
numbers.
\begin{align*}
w_{L_i}(a)=p(\pi)\text{diag}\bigl((-2\abs{a}^{-2})_i,(-\frac{1}{2}\abs{a}^2)_{i+n},B_{2n+1}\bigl),\qquad
\text{ for }a\in \RR^{m-n}\backslash 0,
\end{align*}
where $B\in O(m-n)$ and $\pi$ only permutes $(i,i+n)$ while fixes
other numbers.
\end{remark}
\subsection{Basic relations}\label{sec:4}
We can now define elements $\tilde{x}_{r}(t)$,
$\tilde{x}_{\alpha}(a)$, $\tilde{w}_{r}(t)$,
$\tilde{w}_{\alpha}(a)$, $\tilde{h}_r(t)$, $\tilde{h}_\alpha(a,b)$
etc. as was done in Section \ref{sec:1}. We denote by $\tilde{W}_r$
the subgroup of $\widetilde{G}$ generated by $\tilde{w}_{r}(t)$,
$\tilde{W}_\alpha$ the subgroup generated by
$\tilde{w}_{\alpha}(a)$, $\tilde{H}_r$ the subgroup generated by
$\tilde{h}_{r}(t)$, $\tilde{H}_\alpha$ the subgroup generated by
$\tilde{h}_\alpha(a,b)$. Also, from Lemma \ref{le:5}, it is clear
that certain relations hold both in $\widetilde{G}$ and $G_\RR$. We
record these results in two separate lemmas(Lemma \ref{le:22} and
Lemma \ref{le:14}), since they will serve as ready references later.
\begin{lemma}\label{le:22}
If $a\in\RR^{m-n}$, $t\in\RR\backslash 0$,  the following hold
$\widetilde{G}$(and hence in $G_\RR$ too).

\begin{itemize}
\item[1]$\tilde{w}_{L_n}(a)\tilde{w}_{L_{n-1}-L_n}(t)\tilde{w}_{L_n}(a)^{-1}=\tilde{w}_{L_{n-1}+L_n}\left(-\frac{1}{2}\abs{a}^{2}t\right)$,\\

\item[2]$\tilde{w}_{L_n}(a)\tilde{w}_{L_{n-1}+L_n}(t)\tilde{w}_{L_n}(a)^{-1}=\tilde{w}_{L_{n-1}-L_n}\left(-2\abs{a}^{-2}t\right)$,\\

\item[3]$\tilde{w}_{L_{n-1}-L_n}(t)\tilde{w}_{L_n}(a)\tilde{w}_{L_{n-1}-L_n}(t)^{-1}=\tilde{w}_{L_{n-1}}(at)$,\\

\item[4]$\tilde{w}_{L_{n-1}-L_n}(t)\tilde{w}_{L_{n-1}}(a)\tilde{w}_{L_{n-1}-L_n}(t)^{-1}=\tilde{w}_{L_{n}}\left(-at^{-1}\right)$.
\end{itemize}
Hence,
\begin{itemize}
\item[5]$\tilde{h}_{L_{n-1}-L_n}(t)\tilde{w}_{L_n}(a)\tilde{h}_{L_{n-1}-L_n}(t)^{-1}=\tilde{w}_{L_{n}}(at^{-1})$,\\

\item[6]$\tilde{w}_{L_n}(a)\tilde{h}_{L_{n-1}-L_n}(t)\tilde{w}_{L_n}(a)^{-1}\\
=\tilde{h}_{L_{n-1}+L_n}\left(-\frac{1}{2}\abs{a}^{2}t\right)\tilde{h}_{L_{n-1}+L_n}\left(-\frac{1}{2}\abs{a}^{2}\right)^{-1}$.\\
\end{itemize}
\end{lemma}

We denote by $S^{i}$ the sphere in $\RR^{i+1}$, denote by
$\tilde{W}_s$ the subgroup generated by
$\tilde{w}_{L_n}\left(\sqrt{2}a\right)$, $a\in S^{m-n-1}$. If
$a=\left(a_1,...,a_n\right)$, then
$$\pi_1(\tilde{w}_{L_n}(\sqrt{2}a))=p(\pi)\text{diag}\bigl((-1)_n,(-1)_{2n},B_{2n+1}\bigl),$$
where $\pi$ permutes n and 2n while fixes other numbers and $B\in
O(m-n)$ with entries $B_{i,j}=-2a_ia_j, i\neq j$ and
$B_{i,i}=1-2a_i^2$. Then $B$ is a reflection in the hyperplane
orthogonal to $a$. Thus for any $w\in \tilde{W}_s$,
$$\pi_1(w)=p(\pi)\text{diag}\bigl((-1)^\delta_n,(-1)^\delta_{2n},B_{2n+1}\bigl),$$
where $\delta=2$ if $p(\pi)=I_{m+n}$ and $B\in SO(m-n)$; $\delta=1$
if $p(\pi)$ permutes $n$ and $2n$ and and $B\in O(m-n)$. Without
confusion, we identify $\pi_1(w_{L_n}(\sqrt{2}a))$ with $B$. The
following holds:
\begin{lemma}\label{le:14}
If $w\in \tilde{W}_s$,
$\pi_1(w)=p(\pi)\diag\bigl((-1)^\delta_n,(-1)^\delta_{2n},B_{2n+1}\bigl)$,
$\delta=1$ or $2$, $B\in O(m-n)$, $a\in S^{m-n-1}$, we have
\begin{align*}
&w\tilde{w}_{L_n}\bigl(\sqrt{2}a\bigl)w^{-1}=\left\{\begin{aligned}
 &\tilde{w}_{L_n}\bigl(\sqrt{2}B\cdot a\bigl), & \qquad &\text{ if } p(\pi)=I_{m+n},\notag\\
 &\tilde{w}_{L_{-n}}\bigl(-\sqrt{2}B\cdot a\bigl), & \qquad &\text{ if } p(\pi)\neq I_{m+n},\\
 \end{aligned}
 \right.\\
\end{align*}
where $\cdot$ means linear operation on vectors. Using above formula
and Definition \ref{de:2} one further has
\begin{align*}
&w\tilde{w}_{L_n}\bigl(\sqrt{2}a\bigl)w^{-1}=\left\{\begin{aligned}
 &\tilde{w}_{L_n}\bigl(\sqrt{2}B\cdot a\bigl), & \qquad &\text{ if } p(\pi)=I_{m+n},\\
 &\tilde{w}_{L_{n}}\bigl(-\sqrt{2}B\cdot a\bigl). & \qquad &\text{ if } p(\pi)\neq I_{m+n}\\
 \end{aligned}
 \right.
\end{align*}
\end{lemma}
\subsection{Structure of $\tilde{H}$} We recall the notation set in Section \ref{sec:4}. We
prove the following:
\begin{lemma}\label{le:6}
If $m=n$, $\tilde{H}$ is generated by $\prod_{r\in
\Delta}\tilde{H}_{r}$, where $\Delta=\{L_i-L_{i+1},L_{n-1}+L_n\}$;
if $m\geq n+1$, $\tilde{H}$ is generated by $\bigl(\prod_{r\in
\Delta}\tilde{H}_{r}\bigl)\cdot\tilde{H}_{s}$, where
$\Delta=\{L_i-L_{i+1},L_{n-1}+L_n\}$ and $\tilde{H}_{s}$ is
generated by $\tilde{h}_{L_n}(\sqrt{2}a, \sqrt{2}b), a,b\in
S^{m-n-1}$.
\end{lemma}
\begin{proof}
By Lemma \ref{le:1}, the case for $m=n$ is obvious. If $m\geq n+1$,
$\tilde{H}$ is generated by $\prod_{r\in \Delta}\tilde{H}_{r}$,
where $\Delta=\{L_i-L_{i+1},L_n\}$.

For any $a,b\in \RR^{m-n}\backslash 0$, denote $\sum a_i^2=a_0$,
$\sum b_i^2=b_0$. Using Lemma \ref{le:22}, we have:
\begin{align*}
\tilde{h}_{L_n}(a,b)&=\tilde{w}_{L_n}(a)\tilde{w}_{L_n}(-b)\\
&=\tilde{h}_{L_{n-1}-L_n}\big(\frac{\sqrt{2}}{\sqrt{a_0}}\big)\tilde{w}_{L_n}\big(\frac{\sqrt{2}a}{\sqrt{a_0}}\big)\tilde{h}_{L_{n-1}-L_n}
\big(\frac{\sqrt{2}}{\sqrt{a_0}}\big)^{-1}\\
&\cdot
\tilde{h}_{L_{n-1}-L_n}\big(\frac{\sqrt{2}}{\sqrt{b_0}}\big)\tilde{w}_{L_n}\big(-\frac{\sqrt{2}b}{\sqrt{b_0}}\big)\tilde{h}_{L_{n-1}-L_n}\big(\frac{\sqrt{2}}{\sqrt{b_0}}
\big)^{-1}\\
&=\tilde{h}_{L_{n-1}-L_n}\big(\frac{\sqrt{2}}{\sqrt{a_0}}\big)\tilde{h}_{L_{n-1}+L_n}(-1)\tilde{h}_{L_{n-1}+L_n}\big(-\frac{\sqrt{2}}{\sqrt{a_0}}\big)^{-1}\\
&\cdot
\tilde{h}_{L_{n-1}+L_n}\big(-\frac{\sqrt{2}}{\sqrt{b_0}}\big)\tilde{h}_{L_{n-1}+L_n}(-1)^{-1}\tilde{h}_{L_n}\big(\frac{\sqrt{2}a}{\sqrt{a_0}},\frac{\sqrt{2}b}{\sqrt{b_0}}
\big)\\
&\cdot
\tilde{h}_{L_{n-1}-L_n}\big(\frac{\sqrt{2}}{\sqrt{b_0}}\big)^{-1}
\end{align*}
Thus $\tilde{H}_{L_n}$ is generated by $\tilde{H}_{s}$,
$\tilde{H}_{L_{n-1}-L_n}$ and $\tilde{H}_{L_{n-1}+L_n}$. Hence we
proved the lemma.
\end{proof}

If $m\geq n+2$, for any $(a,b)\in S^1,$ $1\leq j\leq m-n-1$, let
\begin{align*}
\tilde{h}^{j}_{L_n}\bigl(\sqrt{2}a,\sqrt{2}b\bigl)&=\tilde{w}_{L_n}\bigl(0,\dots,0,\underset{j}{\sqrt{2}a},\underset{j+1}{\sqrt{2}b},0,\dots,0\bigl)\\
&\cdot
\tilde{w}_{L_n}\bigl(0,\dots,0,\underset{j}{-\sqrt{2}},0,\dots,0\bigl).
\end{align*}

\begin{corollary}\label{le:9} If $m\geq n+2$, $\tilde{H}$
is generated by $\bigl(\prod_{r\in
\Delta}\tilde{H}_{r}\bigl)\tilde{H}_{s_0}$, where
$\Delta=\{L_i-L_{i+1},L_{n-1}+L_n\}$ and $\tilde{H}_{s_0}$ is
generated by $\tilde{h}^j_{L_n}(\sqrt{2}a, \sqrt{2}b), (a,b\in
S^{1})$.
\end{corollary}
\begin{proof}
By Lemma \ref{le:6}, we just need to show
$\tilde{H}_{s}=\tilde{H}_{s_0}$. Observe that
$\pi_1\bigl(\tilde{h}^{j}_{L_n}(\sqrt{2}a,\sqrt{2}b)\bigl)$=diag$(R_{2n+j})$,
where
$$R=\left(
\begin{array}{ccc}
   a^2-b^2& -2ab  \\
   2ab & a^2-b^2\\
\end{array}
\right).$$ Then it is clear that
$\pi_1(\tilde{h}^{j}_{L_n}(a,b))(j\leq m-n-1)$ generate a subgroup
isomorphic to $SO(m-n)$.

Hence for any $a\in S^{m-n-1}$ we can find $(e_i,f_i)\in S^1$ such
that
$$\bigl(\prod_i\pi_1(\tilde{h}^{i}_{L_n}(\sqrt{2}e_i,\sqrt{2}f_i))\bigl)\cdot (\sqrt{2},0,...,0)=\sqrt{2}a.$$
Denote $\prod_i\tilde{h}^{i}_{L_n}(\sqrt{2}e_i,\sqrt{2}f_i)$ by $A$.
By Lemma \ref{le:14} we have
\begin{align*}
A\tilde{w}_{L_n}(\sqrt{2},0,...,0)A^{-1}=\tilde{w}_{L_n}(\sqrt{2}a).
\end{align*}
Similarly, for any $b\in S^{m-n-1}$, we can find $B\in
\tilde{H}_{s_0}$ such that
\begin{align*}
B\tilde{w}_{L_n}(-\sqrt{2},0,...,0)B^{-1}=\tilde{w}_{L_n}(\sqrt{2}b).
\end{align*}
It follows
\begin{align*}
&\tilde{w}_{L_n}(\sqrt{2}a)\tilde{w}_{L_n}(\sqrt{2}b)\\
&=A\bigl(\tilde{w}_{L_n}(\sqrt{2},0,...,0)A^{-1}B\tilde{w}_{L_n}(-\sqrt{2},0,...,0)\bigl)B^{-1}.
\end{align*}
By Lemma \ref{le:14}, it is easy to check
$$\tilde{w}_{L_n}(\sqrt{2},0,...,0)\tilde{H}_{s_0}\tilde{w}_{L_n}(-\sqrt{2},0,...,0)^{-1}\subseteq\tilde{H}_{s_0}.$$
Thus we get the conclusion.
\end{proof}
\begin{lemma}\label{le:23}
\begin{align*}
&(1) \ker(\pi_1)\cap \tilde{H}_{L_1-L_2}=\{\prod_i
\tilde{h}_{L_1-L_2}(t_i)\mid \text{ \emph{with} }\prod_i t_i=1\}.\\
&(2)  \ker(\pi_1)\cap \tilde{H}_{r}=\ker(\pi_1)\cap
\tilde{H}_{L_1-L_2}, \qquad \text{ \emph{for} }r=\pm L_i\pm
L_j(i\neq j).
\end{align*}
\end{lemma}
\begin{proof} (1). Notice
$\pi_1(\tilde{h}_{L_1-L_2}(t))=$diag$\left(t_1,(t^{-1})_2,(t^{-1})_{1+n},t_{2+n}\right)$.
Thus (1) is clear.

It follows from (1) that $\ker(\pi_1)\cap \tilde{H}_{L_1-L_2}$ is
generated by elements
$$\tilde{h}_{L_1-L_2}(t_1)\tilde{h}_{L_1-L_2}(t_2)\tilde{h}_{L_1-L_2}(t_1t_2)^{-1},\text{ where }t_1,t_2\in\RR^*.$$

(2) We can prove similarly that $\ker(\pi_1)\cap
\tilde{H}_{r}$($r=\pm L_i\pm L_j$) is generated by elements
$\tilde{h}_{r}(t_1)\tilde{h}_{r}(t_2)\tilde{h}_{r}(t_1t_2)^{-1}$.
Since these simple roots belong to the same orbit under the Weyl
group, an argument similar to one in [\cite{Moore}, Lemma 8.2] shows
that $\ker(\pi_1)\cap \tilde{H}_{r}\subseteq \ker(\pi_1)\cap
\tilde{H}_{L_1-L_2}$ for all roots $r=\pm L_i\pm L_j$. This proves
(2).
\end{proof}

For $t_1,t_2\in\RR^{*}$, we define:
\begin{align*}
&\{t_1,t_2\}=\tilde{h}_{L_1-L_2}(t_1)\tilde{h}_{L_1-L_2}(t_2)\tilde{h}_{L_1-L_2}(t_1t_2)^{-1}.
\end{align*}
Now in exactly the same manner as the proof in the appendix of
\cite{Moore}, we prove that these $\{t_1,t_2\}$'s satisfy the
conditions
\begin{lemma}\label{le:24}
\begin{align*}
\{t_1, t_2\}&=\{t_2,t_1\}^{-1}\qquad \forall t_1,t_2\in \RR^*,\\
\{t_1, t_2\cdot t_3\}&=\{t_1,t_2\}\cdot\{t_1,t_3\}\qquad \forall t_1, t_2,t_3\in \RR^*,\\
\{t_1\cdot t_2,
t_3\}&=\{t_1,t_3\}\cdot\{t_2,t_3\}\qquad \forall t_1, t_2,t_3\in \RR^*,\\
\{t,1-t\}&=1\qquad \forall t\in \RR^*, t\neq 1,\\
\{t,-t\}&=1\qquad \forall t\in\RR^*.
\end{align*}
\end{lemma}
Thus we define a symbol on $\RR$.
\subsection{Construction of a $S^1$-symbol} We construct a new symbol on
 $S^1$ to get prepared for further study of $\ker{\pi_1}\cap
 \tilde{H}_{s_0}$. Up to Section \ref{sec:3}, we will doing
 calculation inside $\tilde{W}_s$(defined after Lemma \ref{le:22}). For the sake of
 simplicity, henceforth we denote $\tilde{w}_{L_n}(\sqrt{2}a)(a\in S^{m-n-1})$ by
$\tilde{w}_{L_n}(a)$, $\tilde{h}_{L_n}(\sqrt{2}a,\sqrt{2}b)(a,b\in
S^{m-n-1})$ by $\tilde{h}_{L_n}(a,b)$
 and
$\tilde{h}^i_{L_n}(\sqrt{2}b,\sqrt{2}c)((b,c)\in S^1)$ by
$\tilde{h}^i_{L_n}(b,c)$ until the end of Section \ref{sec:3}.

\begin{lemma} \label{le:3}For $\forall(a,b), (c,d)\in S^1$ we have:
\begin{align*}
&[\tilde{h}^i_{L_n}(a,b),\tilde{h}^i_{L_n}(c,d)]\notag \\
&=\tilde{h}^i_{L_n}\bigl((a^2-b^2,2ab)\cdot(c,d)\bigl)\tilde{h}^i_{L_n}(a^2-b^2,2ab)^{-1}\tilde{h}^i_{L_n}(c,d)^{-1}\notag \\
&=\tilde{h}^i_{L_n}(a,b)\tilde{h}^i_{L_n}(c^2-d^2,2cd)\tilde{h}^i_{L_n}\bigl((a,b)\cdot(c^2-d^2,2cd)\bigl)^{-1}
\end{align*}
where the $\cdot$ is multiplication among complex numbers.
\end{lemma}
\begin{proof}
Using Lemma \ref{le:14}, it follows
\begin{align*}
&\tilde{h}^i_{L_n}(a,b)\tilde{w}_{L_n}(0,\dots,c_i,d_{i+1},0\dots,0)\tilde{h}^i_{L_n}(a,b)^{-1}\\
&=\tilde{w}_{L_n}\bigl(0,\dots,(ca^2-cb^2-2dba)_i,(da^2-db^2+2cba)_{i+1},\dots,0\bigl).
\end{align*}
Thus we have:
\begin{align*}
&\tilde{h}^i_{L_n}(a,b)\tilde{h}^i_{L_n}(c,d)\tilde{h}^i_{L_n}(a,b)^{-1}\\
&=\tilde{h}^i_{L_n}\bigl((a^2-b^2,2ab)\cdot(c,d)\bigl)\tilde{h}^i_{L_n}(a^2-b^2,2ab)^{-1},
\end{align*}
and
\begin{align*}
&\tilde{h}^i_{L_n}(a,b)^{-1}\tilde{h}^i_{L_n}(c,d)\tilde{h}^i_{L_n}(a,b)\\
&=\tilde{h}^i_{L_n}(c^2-d^2,2cd)\tilde{h}^i_{L_n}\bigl((a,b)\cdot(c^2-d^2,2cd)\bigl)^{-1}.
\end{align*}
We have thus proved the lemma.
\end{proof}
\begin{lemma}\label{le:35}
For $\forall a\in S^{m-n-1}$,
$$\tilde{w}_{L_n}(a)=\tilde{w}_{L_n}(-a)\tilde{h}^1_{L_n}(-1,0)$$ and
$$\tilde{h}^1_{L_n}(-1,0)\tilde{h}^1_{L_n}(-1,0)=e.$$
\end{lemma}
\begin{proof}
Notice $\pi_1(\tilde{h}^1_{L_n}(-1,0))=I_{m+n}$. By Lemma
\ref{le:14}, for any $a=(a_1,a_2,\dots,a_{m-n})\in S^{m-n-1}$ we
have
\begin{align*}
&\tilde{h}^1_{L_n}(-1,0)=\tilde{w}_{L_n}(a)\tilde{h}^1_{L_n}(-1,0)\tilde{w}_{L_n}(a)^{-1}\\
&=\tilde{w}_{L_n}\left(1-2a_1^2,-2a_1a_2,-2a_1a_3,\dots,-2a_1a_{m-n}\right)\\
&\cdot
\tilde{w}_{L_n}\left(1-2a_1^2,-2a_1a_2,-2a_1a_3,\dots,-2a_1a_{m-n}\right).
\end{align*}
Let $f:S^{m-n-1}\rightarrow S^{m-n-1}$ be
$$f(a)=\left(1-2a_1^2,-2a_1a_2,-2a_1a_3,\dots,-2a_1a_{m-n}\right)$$
for $\forall a=\left(a_1,a_2,\dots,a_{m-n}\right)\in S^{m-n-1}$.

It is easy to check $f$ is surjective. Hence we proved the first par
of the lemma. For the second part, notice
$\pi_1(\tilde{h}^1_{L_n}(-1,0))=I_{m+n}$ and use Lemma \ref{le:22}
we have
\begin{align*}
\tilde{h}^1_{L_n}(-1,0)=&\tilde{h}_{L_{n-1}-L_n}(-1)\tilde{h}^1_{L_n}(-1,0)\tilde{h}_{L_{n-1}-L_n}(-1)^{-1}\\
=&\tilde{w}_{L_n}(1,\dots,0)\tilde{w}_{L_n}(1,\dots,0)\\
=&\tilde{h}^1_{L_n}(-1,0)^{-1}.
\end{align*}
Hence we proved the lemma.
\end{proof}
Let $H^i$ be the subgroup generated by
$\tilde{h}^i_{L_n}(a,b)((a,b)\in S^1)$. We prove the following:
\begin{lemma}\label{le:12}
\begin{align*}
&(1) \ker(\pi_1)\cap H^1=\{\prod_j
\tilde{h}^1_{L_n}(a_j,b_j)\mid \text{ \emph{with} }\prod_j (a_j^2-b_j^2+2a_jb_j \emph{i})=1\},\\
&(2)  \ker(\pi_1)\cap H^j=\ker(\pi_1)\cap H^1, \qquad \text{
\emph{for} }j\leq m-n-1.
\end{align*}
\end{lemma}
\begin{proof}
(1) Notice $\pi_1(\tilde{h}^{1}_{L_n}(a,b))$=diag$(R_{2n+1})$, where
$$R=\left(
\begin{array}{ccc}
   a^2-b^2& -2ab  \\
   2ab & a^2-b^2\\
\end{array}
\right).$$ We identify the above matrix with a complex number
$a_j^2-b_j^2+2a_jb_j \textrm{i}$, then (1) is clear.

It follows from (1) that $\ker(\pi_1)\cap H^1$ is generated by
elements
$$\tilde{h}^1_{L_n}(a,b)\tilde{h}^1_{L_n}(c,d)\tilde{h}^1_{L_n}\bigl((a,b)\cdot (c,d)\bigl)^{-1}\text{ and }\tilde{h}^1_{L_n}(-1,0).$$
where $(a,b),(c,d)\in S^1$ and $\cdot$ means multiplication among
complex numbers.

(2) We can prove similarly that $\ker(\pi_1)\cap H^i$ is generated
by elements
$$\tilde{h}^i_{L_n}(a,b)\tilde{h}^i_{L_n}(c,d)\tilde{h}^i_{L_n}\bigl((a,b)\cdot
(c,d)\bigl)^{-1}\text{ and }\tilde{h}^i_{L_n}(-1,0).$$ Let $i, j$ be
distinct, let
$$w=\tilde{w}_{L_n}\bigl(0,-\frac{\sqrt{2}}{2},\dots,(\frac{\sqrt{2}}{2})_{i+1},\dots,0\bigl)\cdot
\tilde{w}_{L_n}\bigl(-\frac{\sqrt{2}}{2},\dots,(\frac{\sqrt{2}}{2})_i,\dots,0\bigl).$$
By Lemma \ref{le:14} we have
\begin{align*}
&w\tilde{h}^i_{L_n}(a,b)w^{-1}=\tilde{h}^1_{L_n}(a,b).
\end{align*}
Since $\tilde{h}^i_{L_n}(-1,0)\in Z(\widetilde{G}),$ we have
\begin{align*}
&\tilde{h}^i_{L_n}(a,b)\tilde{h}^i_{L_n}(c,d)\tilde{h}^i_{L_n}\bigl((a,b)\cdot
(c,d)\bigl)^{-1}\\
&=w\tilde{h}^i_{L_n}(a,b)\tilde{h}^i_{L_n}(c,d)\tilde{h}^i_{L_n}\bigl((a,b)\cdot
(c,d)\bigl)^{-1}w^{-1}\\
&=\tilde{h}^1_{L_n}(a,b)\tilde{h}^1_{L_n}(c,d)\tilde{h}^1_{L_n}\bigl((a,b)\cdot
(c,d)\bigl)^{-1},
\end{align*}
and
\begin{align*}
\tilde{h}^i_{L_n}(-1,0)=w\tilde{h}^i_{L_n}(-1,0)w^{-1}=\tilde{h}^1_{L_n}(-1,0).
\end{align*}

 Thus we have proved (2).
\end{proof}
For $(a,b),(c,d)\in S^1$, we define:
\begin{align*}
&\{(a,b),(c,d)\}=\tilde{h}^1_{L_n}\bigl((a,b)\cdot(c,d)\bigl)\tilde{h}^1_{L_n}(a,b)^{-1}\tilde{h}^1_{L_n}(c,d)^{-1},
\end{align*}
By Lemma \ref{le:3}, in exactly the same manner as the proof in the
appendix of \cite{Moore}, these $\{(a,b),(c,d)\}$'s satisfy the
conditions
\begin{lemma}\label{le:17}
For all $(a,b),(c,d),(a_1,b_1),(c_1,d_1)\in S^1$, we have:
\begin{align*}
\{(a,b), (c,d)\}&=\{(c,d), (a,b)\}^{-1} ,\\
\{(a,b), (c,d)\cdot(c_1,d_1)\}&=\{(a,b),(c,d)\}\cdot\{(a,b),(c_1,d_1)\},\\
\{(a,b)\cdot(a_1,b_1),
(c,d)\}&=\{(a,b),(c,d)\}\cdot\{(a_1,b_1),(c,d)\},\\
\{(c,d),(-c,-d)\}&=1.
\end{align*}
\end{lemma}
Thus we define a symbol on $S^1$.

\subsection{Structure of $\ker(\pi_1)\cap \tilde{H}_{s_0}$}
 We recall the notation set in
Corollary \ref{le:9}. $\tilde{H}_{s_0}$ is the subgroup generated by
all $\tilde{h}^j_{L_n}(a,b)$($(a,b)\in S^1$). We focus on studying
$\ker(\pi_1)\cap \tilde{H}_{s_0}$. The crucial step in proving the
main Theorem \ref{th:3} is:
\begin{theorem}\label{th:8}
$\ker(\pi_1)\cap \tilde{H}_{s_0}=\ker(\pi_1)\cap H^1$.
\end{theorem}
The ensuing discussion (up to Lemma \ref{le:15}) proves this
theorem. Henceforth we consider the quotient group
$\tilde{W}_{s}/(\ker(\pi_1)\cap H^1)$ until the end of Section
\ref{sec:3}.(Note that $\ker(\pi_1)\cap H^1$ being central, this is
well defined.) We continue to write $\tilde{h}^j_{L_n}(a,b)$ for
$(a,b)\in S^{1}$ and $\tilde{w}_{L_n}(a)$ for $a\in S^{m-n-1}$ for
their images in $\tilde{W}_{s}/(\ker(\pi_1)\cap H^1)$. However,
there is no confusion in doing so.

As a first step towards the proof of Theorem \ref{th:8}, we prove:
\begin{proposition}\label{po:1}
\begin{align*}
H^{i}\cdot H^{i+1}\cdot H^{i}=H^{i+1}\cdot H^{i}\cdot H^{i+1}.
\end{align*}
\end{proposition}
Let $(a,b,c)\in S^2$, define
$$\tilde{w}_{L_n}^{i}(a,b,c)=\tilde{w}_{L_n}(0,\ldots,a_{i},b_{i+1},c_{i+2},0,\ldots,0).$$
Lemma \ref{le:4}--\ref{fact:1} are preparations for proof of
Proposition \ref{po:1}.
\begin{lemma}\label{le:4}
For $\forall$ $(a,b,c)\in S^2$, $\forall x\in\RR$, $\forall j$,
there exist $(d,g,f)\in S^2$,  $(d_1,g_1,f_1)\in S^2$, $y\in\RR$ and
$y_1\in\RR$ such that
\begin{align*}
&(1) \tilde{w}^{j}_{L_n}(a,b,c)\tilde{w}^j_{L_n}(\cos x,\sin
x,0)=\tilde{w}^{j}_{L_n}(d,g,f)\tilde{w}^{j}_{L_n}(0,\cos y,\sin y),\\
&(2) \tilde{w}^{j}_{L_n}(a,b,c)\tilde{w}^j_{L_n}(0,\cos x,\sin
x)=\tilde{w}^{j}_{L_n}(d_1,g_1,f_1)\tilde{w}^{j}_{L_n}(\cos y_1,\sin
y_1,0).
\end{align*}
\end{lemma}
\begin{proof} (1) If $(a,b,c)$ and $(\cos x,\sin
x,0)$ are collinear, then $(a,b,c)=\pm (\cos x,\sin x,0)$. We just
let $-(d,g,f)=(0,\cos y,\sin y)=(0,-1,0)$. By Lemma \ref{le:35}, we
get the conclusion.

Suppose $(a,b,c)$ and $(\cos x,\sin x,0)$ are not collinear. Then we
choose $y\in\RR$ and $(d,g,f)\in S^2$ such that $(0,\cos y,\sin y)$
and $(d,g,f)$ are both on the plane generated by $(a,b,c)$ and
$(\cos x,\sin x,0)$ and satisfying
$$\angle\bigl((a,b,c),(\cos x,\sin
x,0)\bigl)=\angle\bigl((d,g,f),(0,\cos y,\sin y)\bigl),$$ where
$\angle$ means the angle between 2 vectors in $\RR^3$.  Choose an
$h$ in the subgroup generated by $H^j$ and $H^{j+1}$ such that
$\pi_1(h)$ maps the 4 vectors to $xy$-plane(with last coordinate 0
in $\RR^3$). Denote
\begin{alignat*}{3}
\pi_1(h)\cdot(a,b,c)&=(a_1,b_1,0), & \qquad\pi_1(h)\cdot(\cos x,\sin
x,0)&=(\cos x_1,\sin x_1,0),\\
\pi_1(h)\cdot(d,g,f)&=(d_1,g_1,0), & \qquad\pi_1(h)\cdot(0,\cos
y,\sin y)&=(\cos y_1,\sin y_1,0).
\end{alignat*}
We have
$$\angle\left((a_1,b_1,0),(\cos x_1,\sin x_1,0)\right)=\angle\left((d_1,g_1,0),(\cos y_1,\sin y_1,0)\right),$$
since $\pi_1(h)\in SO(3)$.

From the above equation and the fact that
$\pi_1\bigl(\tilde{w}^{j}_{L_n}(a_1,b_1,0)\tilde{w}^j_{L_n}(\cos
x_1,\sin x_1,0)\bigl)$ is a rotation in $xy$-plane with angle 2
times the one between $(\cos x_1,\sin x_1,0)$ and $(a_1,b_1,0)$, it
follows
\begin{align*}
&\pi_1\left(\tilde{w}^{j}_{L_n}(a_1,b_1,0)\tilde{w}^j_{L_n}(\cos x_1,\sin x_1,0)\right)\\
&=\pi_1\left(\tilde{w}^{j}_{L_n}(d_1,g_1,0)\tilde{w}^j_{L_n}(\cos
y_1,\sin y_1,0)\right).
\end{align*}

Hence we get
\begin{align*}
&h\tilde{w}^{j}_{L_n}(a,b,c)\tilde{w}^j_{L_n}(\cos x,\sin
x,0)h^{-1}\\
&=\tilde{w}^{j}_{L_n}(a_1,b_1,0)\tilde{w}^j_{L_n}(\cos x_1,\sin
x_1,0)\\
&=\tilde{w}^{j}_{L_n}(d_1,g_1,0)\tilde{w}^j_{L_n}(\cos y_1,\sin y_1,0)\qquad(\text{ by Lemma }\ref{le:12})\\
&=h\tilde{w}^{j}_{L_n}(d,g,f)\tilde{w}^{j}_{L_n}(0,\cos y,\sin
y)h^{-1}.
\end{align*}
Thus we proved $$\tilde{w}^{j}_{L_n}(a,b,c)\tilde{w}^j_{L_n}(\cos
x,\sin x,0)=\tilde{w}^{j}_{L_n}(d,g,f)\tilde{w}^{j}_{L_n}(0,\cos
y,\sin y).$$

(2) A similar argument holds for (2).

We have thus proved the lemma completely.
\end{proof}
\begin{lemma}\label{fact:1}
For $\forall \theta_1,\theta_2,\theta_3\in\RR$, there exist
$\beta_1, \beta_2, \beta_3,$ $\beta_4, \beta_5, \beta_6\in\RR$, such
that
\begin{align*}
(1) &\tilde{h}^j_{L_n}(\cos\theta_1,\sin\theta_1)\tilde{h}^{j+1}_{L_n}(\cos\theta_2,\sin\theta_2)\tilde{h}^j_{L_n}(\cos\theta_3,\sin\theta_3)\\
&=
\tilde{h}^{j+1}_{L_n}(\cos\beta_1,\sin\beta_1)\tilde{h}^{j}_{L_n}(\cos\beta_2,\sin\beta_2)\tilde{h}^{j+1}_{L_n}(\cos\beta_3,\sin\beta_3);\\
(2) &\tilde{h}^{j+1}_{L_n}(\cos\theta_1,\sin\theta_1)\tilde{h}^{j}_{L_n}(\cos\theta_2,\sin\theta_2)\tilde{h}^{j+1}_{L_n}(\cos\theta_3,\sin\theta_3)\\
&=
\tilde{h}^{j}_{L_n}(\cos\beta_4,\sin\beta_4)\tilde{h}^{j+1}_{L_n}(\cos\beta_5,\sin\beta_5)\tilde{h}^{j}_{L_n}(\cos\beta_6,\sin\beta_6).
\end{align*}
\end{lemma}
\begin{proof}
 (1) Using Lemma \ref{le:14}, it follows
\begin{align*}
&\tilde{h}^j_{L_n}(\cos\theta_1,\sin\theta_1)\tilde{h}^{j+1}_{L_n}(\cos\theta_2,\sin\theta_2)\tilde{h}^j_{L_n}(\cos\theta_3,\sin\theta_3)\\
&=\tilde{h}^j_{L_n}(\cos\theta_1,\sin\theta_1)\tilde{w}^{j}_{L_n}(0,\cos\theta_2,\sin\theta_2)\tilde{w}^{j}_{L_n}(0,-1,0)\tilde{h}^j_{L_n}(\cos\theta_3,\sin\theta_3)\\
&=\tilde{h}^j_{L_n}(\cos\theta_1,\sin\theta_1)\tilde{w}^{j}_{L_n}(0,\cos\theta_2,\sin\theta_2)\tilde{h}^j_{L_n}(\cos\theta_1,\sin\theta_1)^{-1}\\
&\cdot
\tilde{h}^j_{L_n}(\cos\theta_1,\sin\theta_1)\tilde{w}^{j}_{L_n}(0,-1,0)\tilde{h}^j_{L_n}(\cos\theta_3,\sin\theta_3)\\
&=\tilde{w}^{j}_{L_n}\left(-\sin2\theta_1\cos\theta_2,\cos2\theta_1\cos\theta_2,\sin\theta_2\right)\tilde{h}^j_{L_n}(\cos\theta_1,\sin\theta_1)\\
&\cdot
\tilde{w}^{j}_{L_n}(0,-1,0)\tilde{h}^j_{L_n}\left(\cos\theta_3,\sin\theta_3\right)\tilde{w}^{j}_{L_n}(-1,0,0)\tilde{w}^{j}_{L_n}(1,0,0).
\end{align*}
Since
$$\tilde{h}^j_{L_n}(\cos\theta_1,\sin\theta_1)\tilde{w}^{j}_{L_n}(0,-1,0)\tilde{h}^j_{L_n}(\cos\theta_3,\sin\theta_3)\tilde{w}^{j}_{L_n}(-1,0,0)\in
H^j,$$ and
\begin{align*}
&\pi_1\left(\tilde{h}^j_{L_n}(\cos\theta_1,\sin\theta_1)\tilde{w}^{j}_{L_n}(0,-1,0)\tilde{h}^j_{L_n}(\cos\theta_3,\sin\theta_3)\tilde{w}^{j}_{L_n}(-1,0,0)\right)\\
&=\pi_1\left(\tilde{h}^j_{L_n}\left(-\sin(\theta_1-\theta_3),\cos(\theta_1-\theta_3)\right)\right),
\end{align*}
by Lemma \ref{le:12}, we have
\begin{align*}
&\tilde{h}^j_{L_n}(\cos\theta_1,\sin\theta_1)\tilde{w}^{j}_{L_n}(0,-1,0)\tilde{h}^j_{L_n}(\cos\theta_3,\sin\theta_3)\tilde{w}^{j}_{L_n}(-1,0,0)\\
&=\tilde{h}^j_{L_n}(-\sin(\theta_1-\theta_3),\cos(\theta_1-\theta_3)).
\end{align*}
It follows
\begin{align*}
&\tilde{h}^j_{L_n}(\cos\theta_1,\sin\theta_1)\tilde{h}^{j+1}_{L_n}(\cos\theta_2,\sin\theta_2)\tilde{h}^j_{L_n}(\cos\theta_3,\sin\theta_3)\\
&=\tilde{w}^{j}_{L_n}\left(-\sin2\theta_1\cos\theta_2,\cos2\theta_1\cos\theta_2,\sin\theta_2\right)\\
&\cdot
\tilde{h}^j_{L_n}(-\sin(\theta_1-\theta_3),\cos(\theta_1-\theta_3))\tilde{w}^{j}_{L_n}(1,0,0)\\
&=\tilde{w}^{j}_{L_n}(-\sin2\theta_1\cos\theta_2,\cos2\theta_1\cos\theta_2,\sin\theta_2)\\
&\cdot
\tilde{w}^j_{L_n}(-\sin(\theta_1-\theta_3),\cos(\theta_1-\theta_3),0).
\end{align*}
By Lemma \ref{le:4}, there exist $(a,b,c)\in S^2$ and $\alpha\in\RR$
satisfying
\begin{align}\label{for:20}
&\tilde{w}^{j}_{L_n}(-\sin2\theta_1\cos\theta_2,\cos2\theta_1\cos\theta_2,\sin\theta_2)\notag\\
&\cdot
\tilde{w}^j_{L_n}(-\sin(\theta_1-\theta_3),\cos(\theta_1-\theta_3),0)\notag\\
&=\tilde{w}^{j}_{L_n}(a,b,c)\tilde{w}^j_{L_n}(0,\cos\alpha,\sin\alpha).
\end{align}

Using a similar argument , for $\forall\gamma_1, \gamma_2,
\gamma_3\in\RR$ we have
\begin{align}\label{for:11}
&\tilde{h}^{j+1}_{L_n}(\cos\gamma_1,\sin\gamma_1)\tilde{h}^{j}_{L_n}(\cos\gamma_2,\sin\gamma_2)\tilde{h}^{j+1}_{L_n}(\cos\gamma_3,\sin\gamma_3)\notag\\
&=\tilde{w}^{j}_{L_n}\left(-\sin\gamma_2,\cos2\gamma_1\cos\gamma_2,\sin2\gamma_1\cos\gamma_2\right)\notag\\
&\cdot
\tilde{w}^j_{L_n}\left(0,\cos(\gamma_1-\gamma_3),\sin(\gamma_1-\gamma_3)\right).
\end{align}

It is easy to see that there exist $\beta_1,\beta_2,\beta_3\in\RR$
satisfying
\begin{align*}
(-\sin\beta_2,\cos2\beta_1\cos\beta_2,\sin2\beta_1\cos\beta_2)&=(a,b,c)
\end{align*}
and $$\beta_1-\beta_3=\alpha.$$ Let $\gamma_1=\beta_1,$
$\gamma_2=\beta_2$ and $\gamma_3=\beta_3,$ combine (\ref{for:20})
and (\ref{for:11}),
 we thus proved (1).

 A similar argument holds for
(2). We have thus proved the lemma completely.
\end{proof}

\subsection{Proof of Proposition \ref{po:1}}\label{sec:5}

By Lemma \ref{le:12},  every element of $H^i$ can be expressed as
$\tilde{h}^i_{L_n}(a,b)$ where $(a,b)\in S^1$. Thus by Lemma
\ref{fact:1}, it follows
\begin{align*}
H^{i}\cdot H^{i+1}\cdot H^{i}\subseteq H^{i+1}\cdot H^{i}\cdot
H^{i+1}
\end{align*}
and
\begin{align*}
H^{i+1}\cdot H^{i}\cdot H^{i+1}\subseteq H^{i}\cdot H^{i+1}\cdot
H^{i},
\end{align*}
which completes the proof.

We denote $H^1\cdot H^2,\dotsc,\cdot H^j$ by $\prod_{j=1}^i H^j$.
With this notation we have:
\begin{proposition}\label{le:15}
\begin{align*}
\tilde{H}_{s_0}=\bigl(\prod_{i=1}^{m-n-1}H^i\bigl)\cdot\bigl(\prod_{i=1}^{m-n-2}H^i\bigl),\dotsc,\cdot\bigl(H^1\bigl).
\end{align*}
\end{proposition}
\begin{proof} We write
$\tilde{h}^i_{L_n}$ for the set of all $\tilde{h}^i_{L_n}(a,b)$,
$(a,b)\in S^1$. It is sufficient to prove:
\begin{align*}
\tilde{H}_{s_0}=\bigl(\prod_{i=1}^{m-n-1}\tilde{h}^i_{L_n}\bigl)\cdot\bigl(\prod_{i=1}^{m-n-2}\tilde{h}^i_{L_n}\bigl),\dotsc,\cdot\bigl(\tilde{h}^1_{L_n}\bigl).
\end{align*}
We use induction on $m-n=2+k$. For $k=0$ we are though by Lemma
\ref{le:12}.

If $k=1$, any element in $\tilde{H}_{s_0}$ can be written as
$\tilde{h}^1_{L_n}\tilde{h}^2_{L_n},\dotsc,\tilde{h}^1_{L_n}\tilde{h}^2_{L_n}$.
Keep using Proposition \ref{po:1} and Lemma \ref{le:12}, then we are
through.

Now assume the lemma is correct for $m-n=2+k$. We will show it is
correct for $m-n=2+k+1$. Denote $Y_k$ the subgroup generated by
$H^j$ where $j\leq k+1$. Then $\tilde{H}_{s_0}$ is generated by
$\{Y_k,\tilde{h}^{k+2}_{L_n}\}$.  By induction, any $y_k\in Y_k$ can
be written as
$$y_k=\bigl(\prod_{i=1}^{k+1}\tilde{h}^i_{L_n}\bigl)\cdot\bigl(\prod_{i=1}^{k}\tilde{h}^i_{L_n}\bigl)\cdot,\dotsc,\cdot\bigl(\prod_{i=1}\tilde{h}^i_{L_n}\bigl).$$

Next we will see what changes are brought to this express after
adding $\tilde{h}^{i}_{L_n}(i\leq 2+k)$ from both sides of $y_k$.
Obviously, adding $\tilde{h}^{i}_{L_n}(i\leq 1+k)$ from either side
makes no changes. One considers adding $\tilde{h}^{k+2}_{L_n}$ from
both sides. Notice if $2\leq |i-j|$,
\begin{align}\label{for:9}
\tilde{h}^{i}_{L_n}(a,b)\tilde{h}^{j}_{L_n}(c,d)\tilde{h}^{i}_{L_n}(a,b)^{-1}=\tilde{h}^{j}_{L_n}(c,d),
\end{align}
thus for any $y_k$, we have
\begin{align*}
&\tilde{h}^{k+2}_{L_n}y_k\tilde{h}^{k+2}_{L_n}\\
&=\tilde{h}^{k+2}_{L_n}\bigl(\prod_{i=1}^{k+1}\tilde{h}^i_{L_n}\bigl)\bigl(\prod_{i=1}^{k}\tilde{h}^i_{L_n}\bigl),\dotsc,\bigl(\prod_{i=1}\tilde{h}^i_{L_n}\bigl)\tilde{h}^{k+2}_{L_n}\\
&=\bigl(\prod_{i=1}^{k}\tilde{h}^i_{L_n}\bigl)\tilde{h}^{k+2}_{L_n}\tilde{h}^{k+1}_{L_n}\tilde{h}^{k+2}_{L_n}\bigl(\prod_{i=1}^{k}\tilde{h}^i_{L_n}\bigl),\dotsc,\bigl(\prod_{i=1}\tilde{h}^i_{L_n}\bigl)\qquad
(\text{ by }(\ref{for:9}))\\
&=\bigl(\prod_{i=1}^{k}\tilde{h}^i_{L_n}\bigl)\tilde{h}^{k+1}_{L_n}\tilde{h}^{k+2}_{L_n}\tilde{h}^{k+1}_{L_n}\bigl(\prod_{i=1}^{k}\tilde{h}^i_{L_n}\bigl),\dotsc,\bigl(\prod_{i=1}\tilde{h}^i_{L_n}\bigl)\qquad(\text{ by Proposition }\ref{po:1})\\
&=\bigl(\prod_{i=1}^{k+2}\tilde{h}^i_{L_n}\bigl)\tilde{h}^{k+1}_{L_n}\bigl(\prod_{i=1}^{k}\tilde{h}^i_{L_n}\bigl),\dotsc,\bigl(\prod_{i=1}\tilde{h}^i_{L_n}\bigl).
\end{align*}
Since
$\tilde{h}^{k+1}_{L_n}\bigl(\prod_{i=1}^{k}\tilde{h}^i_{L_n}\bigl),\dotsc,\bigl(\prod_{i=1}\tilde{h}^i_{L_n}\bigl)\in
Y_k$, by induction, we have
\begin{align*}
&\tilde{h}^{k+2}_{L_n}y_k\tilde{h}^{k+2}_{L_n}\\
&=\bigl(\prod_{i=1}^{k+2}\tilde{h}^i_{L_n}\bigl)\bigl(\prod_{i=1}^{k+1}\tilde{h}^i_{L_n}\bigl),\dotsc,\bigl(\prod_{i=1}\tilde{h}^i_{L_n}\bigl).
\end{align*}
Next we consider the changes after adding $\tilde{h}^{i}_{L_n}(i\leq
2+k)$ from left side of the above expression. It is clear adding
$\tilde{h}^{1}_{L_n}$ makes no changes. Neither for
$\tilde{h}^{k+2}_{L_n}$, since a same argument holds as in previous
proof. For $\tilde{h}^{j}_{L_n}(1<j<k+2)$ we have:
\begin{align*}
&\tilde{h}^{j}_{L_n}\bigl(\prod_{i=1}^{k+2}\tilde{h}^i_{L_n}\bigl)\bigl(\prod_{i=1}^{k+1}\tilde{h}^i_{L_n}\bigl),\dotsc,\bigl(\prod_{i=1}\tilde{h}^i_{L_n}\bigl)\\
&=\bigl(\prod_{i=1}^{j-2}\tilde{h}^{i}_{L_n}\bigl)\tilde{h}^j_{L_n}\tilde{h}^{j-1}_{L_n}\tilde{h}^j_{L_n}\bigl(\prod_{i=j+1}^{k+2}\tilde{h}^j_{L_n}\bigl)\bigl(\prod_{i=1}^{k+1}\tilde{h}^i_{L_n}\bigl),\dotsc,\bigl(\prod_{i=1}\tilde{h}^i_{L_n}\bigl)(\text{ by }(\ref{for:9}))\\
&=\bigl(\prod_{i=1}^{j-2}\tilde{h}^{i}_{L_n}\bigl)\tilde{h}^{j-1}_{L_n}\tilde{h}^{j}_{L_n}\tilde{h}^{j-1}_{L_n}\bigl(\prod_{i=j+1}^{k+2}\tilde{h}^j_{L_n}\bigl)\bigl(\prod_{i=1}^{k+1}\tilde{h}^i_{L_n}\bigl),\dots,\bigl(\prod_{i=1}\tilde{h}^i_{L_n}\bigl)(\text{ by Proposition }\ref{po:1})\\
&=\bigl(\prod_{i=1}^{k+2}\tilde{h}^{i}_{L_n}\bigl)\tilde{h}^{j-1}_{L_n}\bigl(\prod_{i=1}^{k+1}\tilde{h}^i_{L_n}\bigl)\bigl(\prod_{i=1}^{k}\tilde{h}^i_{L_n}\bigl),\dots,\bigl(\prod_{i=1}\tilde{h}^i_{L_n}\bigl).\qquad
(\text{ by }(\ref{for:9}))
\end{align*}
Since
$\tilde{h}^{j-1}_{L_n}\bigl(\prod_{i=1}^{k+1}\tilde{h}^i_{L_n}\bigl)\bigl(\prod_{i=1}^{k}\tilde{h}^i_{L_n}\bigl),\dotsc,\bigl(\prod_{i=1}\tilde{h}^i_{L_n}\bigl)\in
Y_k$, by induction we have
\begin{align*}
&\tilde{h}^{j}_{L_n}\bigl(\prod_{i=1}^{k+2}\tilde{h}^i_{L_n}\bigl)\bigl(\prod_{i=1}^{k+1}\tilde{h}^i_{L_n}\bigl),\dotsc,\bigl(\prod_{i=1}\tilde{h}^i_{L_n}\bigl)\\
&=\bigl(\prod_{i=1}^{k+2}\tilde{h}^{i}_{L_n}\bigl)\bigl(\prod_{i=1}^{k+1}\tilde{h}^i_{L_n}\bigl)\bigl(\prod_{i=1}^{k}\tilde{h}^i_{L_n}\bigl),\dotsc,\bigl(\prod_{i=1}\tilde{h}^i_{L_n}\bigl).
\end{align*}
Hence we are through left side case. Now consider adding elements
from right side. In this case, we are left with adding
$\tilde{h}^{k+2}_{L_n}$. Then
\begin{align*}
&\bigl(\prod_{i=1}^{k+2}\tilde{h}^i_{L_n}\bigl)\bigl(\prod_{i=1}^{k+1}\tilde{h}^i_{L_n}\bigl),\dotsc,\bigl(\prod_{i=1}\tilde{h}^i_{L_n}\bigl)\tilde{h}^{k+2}_{L_n}\\
&=\bigl(\prod_{i=1}^{k+2}\tilde{h}^i_{L_n}\bigl)\bigl(\prod_{i=1}^{k+1}\tilde{h}^i_{L_n}\bigl)\tilde{h}^{k+2}_{L_n}\bigl(\prod_{i=1}^{k}\tilde{h}^i_{L_n}\bigl),\dotsc,\bigl(\prod_{i=1}\tilde{h}^i_{L_n}\bigl)\qquad(\text{ by }(\ref{for:9}))\\
&=\bigl(\prod_{i=1}^{k+1}\tilde{h}^i_{L_n}\bigl)\bigl(\prod_{i=1}^{k}\tilde{h}^i_{L_n}\bigl)\tilde{h}^{k+2}_{L_n}\tilde{h}^{k+1}_{L_n}\tilde{h}^{k+2}_{L_n}\bigl(\prod_{i=1}^{k}\tilde{h}^i_{L_n}\bigl),\dotsc,\bigl(\prod_{i=1}\tilde{h}^i_{L_n}\bigl)\\
&=\bigl(\prod_{i=1}^{k+1}\tilde{h}^i_{L_n}\bigl)\bigl(\prod_{i=1}^{k}\tilde{h}^i_{L_n}\bigl)\tilde{h}^{k+1}_{L_n}\tilde{h}^{k+2}_{L_n}\tilde{h}^{k+1}_{L_n}\bigl(\prod_{i=1}^{k}\tilde{h}^i_{L_n}\bigl),\dotsc,\bigl(\prod_{i=1}\tilde{h}^i_{L_n}\bigl).\qquad(\text{
by Proposition }\ref{po:1}).
\end{align*}
Since
$\bigl(\prod_{i=1}^{k+1}\tilde{h}^i_{L_n}\bigl)\bigl(\prod_{i=1}^{k}\tilde{h}^i_{L_n}\bigl)\tilde{h}^{k+1}_{L_n}\in
Y_k$, by induction we have
\begin{align*}
&\bigl(\prod_{i=1}^{k+1}\tilde{h}^i_{L_n}\bigl)\bigl(\prod_{i=1}^{k}\tilde{h}^i_{L_n}\bigl)\tilde{h}^{k+1}_{L_n}\tilde{h}^{k+2}_{L_n}\tilde{h}^{k+1}_{L_n}\bigl(\prod_{i=1}^{k}\tilde{h}^i_{L_n}\bigl),\dotsc,\bigl(\prod_{i=1}\tilde{h}^i_{L_n}\bigl)\\
&=\bigl(\prod_{i=1}^{k+1}\tilde{h}^i_{L_n}\bigl)\bigl(\prod_{i=1}^{k}\tilde{h}^i_{L_n}\bigl),\dotsc,\bigl(\prod_{i=1}\tilde{h}^i_{L_n}\bigl)\tilde{h}^{k+2}_{L_n}\tilde{h}^{k+1}_{L_n}\bigl(\prod_{i=1}^{k}\tilde{h}^i_{L_n}\bigl),\dotsc,\bigl(\prod_{i=1}\tilde{h}^i_{L_n}\bigl)\\
&=\bigl(\prod_{i=1}^{k+2}\tilde{h}^i_{L_n}\bigl)\bigl(\prod_{i=1}^{k}\tilde{h}^i_{L_n}\bigl),\dotsc,\bigl(\prod_{i=1}\tilde{h}^i_{L_n}\bigl)\tilde{h}^{k+1}_{L_n}\bigl(\prod_{i=1}^{k}\tilde{h}^i_{L_n}\bigl),\dotsc,\bigl(\prod_{i=1}\tilde{h}^i_{L_n}\bigl).\qquad(\text{
by }(\ref{for:9}))
\end{align*}
Notice
\begin{align*}
\bigl(\prod_{i=1}^{k}\tilde{h}^i_{L_n}\bigl),\dotsc,\bigl(\prod_{i=1}\tilde{h}^i_{L_n}\bigl)\tilde{h}^{k+1}_{L_n}\bigl(\prod_{i=1}^{k}\tilde{h}^i_{L_n}\bigl),\dotsc,\bigl(\prod_{i=1}\tilde{h}^i_{L_n}\bigl)\in
Y_k,
\end{align*}
by induction we have
\begin{align*}
&\bigl(\prod_{i=1}^{k+2}\tilde{h}^i_{L_n}\bigl)\bigl(\prod_{i=1}^{k}\tilde{h}^i_{L_n}\bigl),\dotsc,\bigl(\prod_{i=1}\tilde{h}^i_{L_n}\bigl)\tilde{h}^{k+1}_{L_n}\bigl(\prod_{i=1}^{k}\tilde{h}^i_{L_n}\bigl),\dotsc,\bigl(\prod_{i=1}\tilde{h}^i_{L_n}\bigl),\\
=&\bigl(\prod_{i=1}^{k+2}\tilde{h}^i_{L_n}\bigl)\bigl(\prod_{i=1}^{k+1}\tilde{h}^i_{L_n}\bigl),\dotsc,\bigl(\prod_{i=1}\tilde{h}^i_{L_n}\bigl).
\end{align*}
Thus we have proved adding elements from right leaves the expression
unchanged. Then we finally proved this proposition.
\end{proof}
\subsection{Proof of Theorem \ref{th:8}}\label{sec:3} By Proposition
\ref{le:15}, $\forall h\in \tilde{H}_{s_0}$ can be written as
\begin{align*}
&h=\bigl(\prod_{i=1}^{m-n-1}\tilde{h}^i_{L_n}(a^{m-n-1}_i,
b^{m-n-1}_i)\bigl)\bigl(\prod_{i=1}^{m-n-2}\tilde{h}^i_{L_n}(a^{m-n-2}_i,
b^{m-n-2}_i)\bigl)\\
&,\dots,\tilde{h}^1_{L_n}\left(a^{1}_1, b^{1}_1\right)h_0
\end{align*}
where $(a^{j}_i, b^{j}_i)\in S^1$ for $i,j\leq m-n-1$ and
$h_0\in\ker(\pi_1)\cap H^1$. Notice the element in the lower right
corner of matrix $\pi_1(h)$ is
$\left(a^{m-n-1}_{m-n-1}\right)^2-\left(b^{m-n-1}_{m-n-1}\right)^2$.

If $\pi_1(h)=I_{m+n}$, we have $a^{m-n-1}_{m-n-1}=\pm 1$,
$b^{m-n-1}_{m-n-1}=0$. Thus it follows
$\tilde{h}^{m-n-1}_{L_n}(a^{m-n-1}_{m-n-1},b^{m-n-1}_{m-n-1})\in
\ker(\pi_1)\cap H^1$. By induction, we have $h\in\ker(\pi_1)\cap
H^1$. We thus proved $\ker(\pi_1)\cap
\tilde{H}_{s_0}\subseteq\ker(\pi_1)\cap H^1$. The other side
inclusion is obvious. Hence we finished the proof completely.

\subsection{Proof of Theorem \ref{th:3}}\label{sec:13} We first consider $m=n$. By the fact coming from Lemma \ref{le:6}
that $\ker(\pi_1)\subseteq \prod_{r\in \Delta}\tilde{H}_{r}$, where
$\Delta=\{L_i-L_{i+1},L_{n-1}+L_n\}$, we only need to consider the
elements in $\prod_{r\in \Delta}\tilde{H}_{r}$. By Lemma
\ref{le:23}, $\forall h\in\prod_{r\in \Delta}\tilde{H}_{r}$ can be
written as
$$h=\tilde{h}_{L_1-L_2}(a_1)\tilde{h}_{L_2-L_3}(a_2),\dots,\tilde{h}_{L_{n-1}-L_n}(a_{n-1})\tilde{h}_{L_{n-1}+L_n}(a_{n})h_0$$
where $h_0\in \ker(\pi_1)\cap \tilde{H}_{L_1-L_2}$.

If $\pi_1(h)=I_{m+n}$, we have $a_1=a_2=\dots=a_{n-2}=1$ and
$a_{n-1}=a_n=\pm 1$. Thus we have
\begin{align*}
h=h_0\qquad\text{ or
 }\qquad h=\tilde{h}_{L_{n-1}-L_n}(-1)\tilde{h}_{L_{n-1}+L_n}(-1).
\end{align*}
Notice
\begin{align*}
&\tilde{h}_{L_{n-1}-L_n}(-1)\tilde{h}_{L_{n-1}+L_n}(-1)\\
&=h_{L_1-L_n}^\sim(-1)\bigl(\tilde{h}_{L_{n-1}-L_n}(-1)\tilde{h}_{L_{n-1}+L_n}(-1)\bigl)h_{L_1-L_n}^\sim(-1)^{-1}\\
&=\tilde{w}_{L_{n-1}-L_n}(1)\tilde{w}_{L_{n-1}-L_n}(1)\tilde{w}_{L_{n-1}+L_n}(1)\tilde{w}_{L_{n-1}-L_n}(1)\tilde{w}_{L_{n-1}+L_n}(1)\\
&=\bigl(\tilde{h}_{L_{n-1}-L_n}(-1)\tilde{h}_{L_{n-1}+L_n}(-1)\bigl)^{-1}
\end{align*}
Thus we have
$\bigl(\tilde{h}_{L_{n-1}-L_n}(-1)\tilde{h}_{L_{n-1}+L_n}(-1)\bigl)^2=e$.
 Thus we proved the case for $m=n$.

If $m=n+1$ by Lemma \ref{le:6}, $\ker(\pi_1)\subseteq
\bigl(\prod_{r\in \Delta}\tilde{H}_{r}\bigl)\cdot\tilde{H}_{s}$,
where $\Delta=\{L_i-L_{i+1},L_{n-1}+L_n\}$. Further, by Lemma
\ref{le:23}, $\forall h\in\bigl(\prod_{r\in
\Delta}\tilde{H}_{r}\bigl)\cdot\tilde{H}_{s}$ can be written as
$$h=\tilde{h}_{L_1-L_2}(a_1)\tilde{h}_{L_2-L_3}(a_2),\dots,\tilde{h}_{L_{n-1}-L_n}(a_{n-1})\tilde{h}_{L_{n-1}+L_n}(a_n)h_1h_0$$
where $h_0\in \ker(\pi_1)\cap \tilde{H}_{L_1-L_2}$ and $h_1\in
\tilde{H}_{s}$.

If $\pi_1(h)=I_{m+n}$, we have $a_1=a_2=\dots=a_{n-2}=1$,
$a_{n-1}=a_n=\pm 1$ and $\pi_1(h_1)=I_{m+n}$.

If $a_{n-1}=a_{n}=1$, we get $h\in \left(\ker(\pi_1)\cap
\tilde{H}_{s}\right)\cdot\left(\ker(\pi_1)\cap
\tilde{H}_{L_1-L_2}\right)$.

If $a_{n-1}=a_{n}=-1$, we have
\begin{align}\label{for:27}
&\tilde{h}^1_{L_n}(-1,0)\notag\\
&=\tilde{h}_{L_{n-1}-L_{n}}(-1)\tilde{w}_{L_n}(\sqrt{2},0,\ldots,0)\notag\\
&\cdot\tilde{h}_{L_{n-1}-L_{n}}(-1)^{-1}\tilde{w}_{L_n}(-\sqrt{2},0,\ldots,0)\notag\\
&=\tilde{h}_{L_{n-1}-L_{n}}(-1)\tilde{h}_{L_{n-1}+L_{n}}(-1).
\end{align}
Thus we still get  $h\in \left(\ker(\pi_1)\cap
\tilde{H}_{s}\right)\cdot\left(\ker(\pi_1)\cap
\tilde{H}_{L_1-L_2}\right)$.

Notice $\ker(\pi_1)\cap \tilde{H}_{s}$ is the 2-cyclic group
generated by $h^1_{L_n}(-1)$ by the fact
$\bigl(h^1_{L_n}(-1)\bigl)^2=e$ from lemma \ref{le:35}. Thus we
proved the case for $m=n+1$.

If $m\geq n+2$, by Corollary \ref{le:9}, $\ker(\pi_1)\subseteq
\left(\prod_{r\in \Delta}\tilde{H}_{r}\right)\tilde{H}_{s_0}$, where
$\Delta=\{L_i-L_{i+1},L_{n-1}+L_n\}$. Further, by Lemma \ref{le:23},
$\forall h\in\left(\prod_{r\in
\Delta}\tilde{H}_{r}\right)\tilde{H}_{s_0}$ can be written as
$$h=\tilde{h}_{L_1-L_2}(a_1)\tilde{h}_{L_2-L_3}(a_2),\dots,\tilde{h}_{L_{n-1}-L_n}(a_{n-1})\tilde{h}_{L_{n-1}+L_n}(a_n)h_1h_0$$
where $h_0\in \ker(\pi_1)\cap \tilde{H}_{L_1-L_2}$ and $h_1\in
\tilde{H}_{s_0}$. If $\pi_1(h)=I_{m+n}$, we have
$a_1=a_2=\dots=a_{n-2}=1$, $a_{n-1}=a_n=\pm 1$ and
$\pi_1(h_1)=I_{m+n}$.

If $a_{n-1}=a_{n}=1$, we get $h\in \left(\ker(\pi_1)\cap
\tilde{H}_{s_0}\right)\cdot\left(\ker(\pi_1)\cap
\tilde{H}_{L_1-L_2}\right)$.

If $a_{n-1}=a_{n}=-1$, by (\ref{for:27}) we still get  $$h\in
\left(\ker(\pi_1)\cap
\tilde{H}_{s_0}\right)\cdot\left(\ker(\pi_1)\cap
\tilde{H}_{L_1-L_2}\right).$$

By Theorem \ref{th:8}, $\ker(\pi_1)\cap
\tilde{H}_{s_0}=\ker(\pi_1)\cap H^1$. Thus we have proved
$$\ker(\pi_1)=\left(\ker(\pi_1)\cap H^1\right)\cdot\left(\ker(\pi_1)\cap
\tilde{H}_{L_1-L_2}\right).$$ Hence we proved Theorem \ref{th:3}
completely.

\section{Generating relations of $SU(m,n)$}\label{sec:15}
\subsection{Basic settings for $SU(m,n)$}
In this part, we study the generators of $SU(m,n)$ where $m\geq
n\geq 3$. we use $\overline{a}$ to denote complex conjugate of
complex numbers, vectors or matrices. We use notations as in Section
\ref{sec:7} and Section \ref{sec:12}. Explicitly, this is the case
where $G=SU(m,n)(\CC,H)$ with $H$ a non-degenerate standard
hermitian form of signature $(m,n)$.

  In the
sequel we freely use the notation of previous part without
confusion. We denote by set $S$ the the $(m+n)\times (m+n)$ real
diagonal matrices in $G_\RR$ with lower-right $(m-n)\times(m-n)$
block identity. Let $\Phi$ be the root system of $G$ with respect to
$S$. The roots are $\pm L_i \pm L_j(i<j\leq n)$, whose dimensions
are 2 and $\pm 2 L_i(i\leq n)$ whose dimension is 1. Also the $\pm
L_i(i\leq n)$ are roots if $m\neq n$ with dimensions
 $2(m-n)$. If $m-n\geq 1$, the set of positive roots $\Phi^{+}$ and the corresponding set of
 simple roots $\Delta$ are
\begin{align*}
&\Phi^{+}=\{L_i-L_j\}_{i<j}\cup\{L_i+L_j\}_{i<j}\cup\{L_i\}_i\cup\{2L_i\}_i,\\
&\Delta=\{L_i-L_{i+1}\}_{i<j}\cup\{L_n\};
\end{align*}
if $m=n$, the set of positive roots $\Phi^{+}$ and the corresponding
set of simple roots $\Delta$ are
\begin{align*}
&\Phi^{+}=\{L_i-L_j\}_{i<j}\cup\{L_i+L_j\}_{i<j}\cup\{2L_i\}_i,\\
&\Delta=\{L_i-L_{i+1}\}_i\cup\{L_{n-1}-L_{n}\}_i.
\end{align*}
Correspondingly, if $m-n\geq 1$, the set of $\Phi_1$ are $\pm L_i\pm
L_j(i\neq j), \pm
 L_i$;  if
 $m=n$, the set of $\Phi_1$ are $\pm L_i\pm L_j(i\neq j)$.

 We use $e_{k,\ell}$ to denote matrix with
 $(k,\ell)$ element 1, otherwise 0. We denote
\begin{align*}
f^1_{L_i+L_j} &=(e_{i,j+n}-e_{j,i+n})_{i<j},  &\quad f^2_{L_i+L_j}
&=\textrm{i}(
e_{i,j+n}+e_{j,i+n})_{i<j},\\
f^1_{L_i-L_j} &=(e_{i,j}-e_{j+n,i+n})_{i\neq j}, &\quad
f^2_{L_i-L_j}& =\textrm{i}( e_{i,j}+e_{j+n,i+n})_{i\neq j},\\
f^1_{-L_i-L_j} &=(e_{j+n,i}-e_{i+n,j})_{i<j}, &\quad
f^2_{-L_i-L_j} &=\textrm{i}( e_{j+n,i}+e_{i+n,j})_{i<j},\\
(f_{L_i}^\ell)_1& =(e_{i,2n+\ell}-e_{2n+\ell,i+n})_{\ell\leq m-n},
&\quad (f_{L_i}^\ell)_2 &=\textrm{i}( e_{i,2n+\ell}+
e_{2n+\ell,i+n})_{i\leq m-n},\\
(f_{-L_i}^\ell)_1 &=(e_{i+n,2n+\ell}-e_{2n+\ell,i})_{\ell\leq m-n},
&\quad (f_{-L_i}^\ell)_2 &=\textrm{i}( e_{i+n,2n+\ell}+
e_{2n+\ell,i})_{i\leq m-n},\\
f_{2L_i}&=\textrm{i} e_{i,i+n}, &\quad f_{-2L_i}&=\textrm{i}
e_{i+n,i}.
\end{align*}

For any complex number $z$, we use $r(z)$ to denote the real part
and $i(z)$ to denote the imaginary part. Thus we have
\begin{align*}
U^{r}_\RR=&\{\exp
\left(r(z)f^1_{r}\right)\exp\left(i(z)f^2_{r}\right)\mid z\in \CC\}
\text{ for }r=\pm L_i\pm L_j(i<j), \\
U^{\alpha}_\RR=&\{\exp \bigl(\bigl(t-\sum_j
r(a_j)i(a_j)\bigl)f_{2\alpha}\bigr)\exp
\bigl(r(a_1)(f^{1}_{\alpha})_1\bigl)\exp\bigl(i(a_1)(f^{1}_{\alpha})_2\bigl)\\
&,\dots,\exp
\bigl(r(a_{m-n})(f^{m-n}_{\alpha})_1\bigl)\exp\bigl(i(a_{m-n})(f^{m-n}_{\alpha})_2\bigl)\\
&\mid a=(a_1,\dots,a_{m-n})\in \CC^{m-n},t\in\RR\}\qquad\text{ for
}\alpha=\pm L_i.
\end{align*} Correspondingly, for $t\in \RR$, $z\in\CC$ and
$a=(a_1,\dots,a_{m-n})\in \CC^{m-n}$ we write
\begin{align*}
x_r(z)=&\exp \left(r(z)f^1_{r}\right)\exp\left(i(z)f^2_{r}\right)\in
U^{r}_\RR\qquad\text{ for }r=\pm L_i\pm
L_j(i<j),\\
x_\alpha(t,a)=&\exp \bigl(\bigl(t-\sum_j
r(a_j)i(a_j)\bigr)f_{2\alpha}\bigr)\exp
\bigl(r(a_1)(f^{1}_{\alpha})_1\bigl)\exp\left(i(a_1)(f^{1}_{\alpha})_2\right)\\
&,\dots,\exp
\bigl(r(a_{m-n})(f^{m-n}_{\alpha})_1\bigl)\exp\left(i(a_{m-n})(f^{m-n}_{\alpha})_2\right)\in
U^{\alpha}_\RR\\
&\qquad\text{ for }\alpha=\pm L_i.
\end{align*}
Notice if $a=0$, $x_\alpha(t,0)=\exp(tf_{2\alpha})$.
\subsection{``Chain'' in $SU(m,n)$ and basic relations}
Our next step is to determine explicitly the ``chain" (cf. Lemma
\ref{le:2}) corresponding to the element $x_\alpha(t,a)(\neq e)\in
U^{\alpha}_\RR(\alpha=\pm L_i)$  where $a=(a_1,\dots,a_{m-n})\in
\CC^{m-n}$.

We determine the ``chain" for $x_\alpha(0,a)$ at first. For this,
define $f:\CC^{m-n}\backslash 0\rightarrow \CC^{m-n}\backslash 0$ by
$f(a)=\bigl(\frac{2a_1}{\sum
\abs{a_i}^{2}},\dots,\frac{2a_{m-n}}{\sum \abs{a_i}^{2}}\bigr)$ for
$a=(a_1,\dots,a_{m-n})\in \CC^{m-n}\backslash 0$. With this
notation, we have:
\begin{lemma}
For $x_\alpha(0,a)(\neq e)\in U^{\alpha}_\RR(\alpha=\pm L_i)$, the
``chain" corresponding to it is given by
\begin{align*}
&x_i=x_\alpha(0,a), i\in \ZZ;\qquad y_i=x_{-\alpha}(0,f(a)), i\in
\ZZ.
\end{align*}
Denoting the element $w_\alpha(x_\alpha(0,a))$ by $w_\alpha(0,a)$,
we have
\begin{align*}
w_\alpha(0,a)&=x_\alpha(0,a)x_{-\alpha}(0,f(a))x_\alpha(0,a).
\end{align*}
\end{lemma}

\begin{proof} It is easy to check $\{X,Y,[X,Y]\}$ span a
three-dimensional Lie algebra isomorphic to $SL_2(\RR)$, where
\begin{align*}
&X=\sum_j (r(a_j)(f^{j}_{\alpha})_1+i(a_j)(f^{j}_{\alpha})_2),\\
&Y=\sum_j
-\frac{2r(a_j)(f^{j}_{-\alpha})_1+2i(a_j)(f^{j}_{-\alpha})_2}{\sum
\abs{a_i}^{2}}.
\end{align*}
And we have
\begin{align*}
\exp(X)=x_\alpha(0,a),\qquad\exp(-Y)=x_\alpha(0,f(a)).
\end{align*}
 Thus by Remark
\ref{re:1} we get the conclusion.
\end{proof}
\begin{remark}\label{re:3}
Similar computations can be made for the other roots, such as $\pm
L_i\pm L_j(i<j)$, $\pm 2L_i$. We record the results here:
\begin{align*}
w_{r}(z)=x_{r}(z)x_{-r}(-z^{-1})x_{r}(z),\qquad z\in \CC^*, r=\pm
L_i\pm L_j(i<j),
\end{align*}
where
\begin{align*}
x_i=x_{r}(z)\forall i,\qquad y_i=x_{-r}(-z^{-1})\forall i.
\end{align*}
\begin{align*}
w_{\beta}(t)=w_{\alpha}(t,0)=x_{\alpha}(t,0)x_{-\alpha}(t^{-1},0)x_{\alpha}(t,0),\qquad
t\in \RR^*, \beta=2\alpha, \alpha=\pm L_i,
\end{align*}
where
\begin{align*}
x_i=x_{\alpha}(t,0)\forall i,\qquad y_i=x_{-\alpha}(t^{-1},0)\forall
i.
\end{align*}
Correspondingly, we define
\begin{align*}
&h_r(z)=w_{r}(z)w_{r}(1)^{-1},\qquad z\in\CC^*,r=\pm L_i\pm L_j(i<j),\\
&h_\beta(t)=w_{\beta}(t)w_{\beta}(1)^{-1},\qquad t\in\RR^*,\beta=\pm 2L_i,\\
&h_\alpha((0,a),(0,b))=w_{\alpha}(0,a)w_{\alpha}(0,b)^{-1},\qquad
a,b\in \CC^{m-n}\backslash 0, \alpha=\pm L_i.
\end{align*}
Using the same notations as in Remark \ref{re:2} we have:
\begin{align*}
w_{L_i-L_j}(z)=p(\pi)\text{diag}\bigl((-z^{-1})_i,z_j,(-\overline{z})_{i+n},(\overline{z}^{-1})_{j+n}\bigr),
\qquad \text{ for }z\in \CC^*
\end{align*}
where $\pi$ only permutes $(i,j)$ and $(i+n,j+n)$ while fixes other
numbers.
\begin{align*}
w_{L_i+L_j}(z)=p(\pi)\text{diag}\bigl((-z^{-1})_i,(\overline{z}^{-1})_j,(-\overline{z})_{i+n},z_{j+n}\bigr),\qquad
\text{ for }z\in \CC^*
\end{align*}
where $\pi$ only permutes $(i,j+n)$ and $(j,i+n)$ while fixes other
numbers.
\begin{align*}
w_{L_i}(0,a)=p(\pi)\text{diag}\bigl((-2\abs{a}^{-2})_i,(-\frac{1}{2}\abs{a}^2)_{i+n},B_{2n+1}\bigr),\qquad
\text{ for }a\in \CC^{m-n}\backslash 0,
\end{align*}
where $B\in U(m-n)$ and $\pi$ only permutes $(i,i+n)$ while fixes
other numbers.
\begin{align*}
w_{2L_i}(t)=p(\pi)\text{diag}\bigl((t^{-1}\textrm{i})_i,(t\textrm{i})_{i+n}\bigl),\qquad
\text{ for }t\in \CC^*,
\end{align*}
where $\pi$ only permutes $(i,i+n)$ while fixes other numbers.
\end{remark}

\begin{definition}We can now define elements $\tilde{x}_{r}(z)$,
$\tilde{x}_{\beta}(t)$, $\tilde{x}_{\alpha}(t,a)$,
$\tilde{w}_{r}(z)$, $\tilde{w}_{\beta}(t)$,
$\tilde{w}_{\alpha}(0,a)$, $\tilde{h}_r(z)$, $\tilde{h}_\beta(t)$,
$\tilde{h}_\alpha\bigl((0,a),(0,b)\bigr)$ etc. as was done in
Section \ref{sec:1}. We denote by $\tilde{W}_r(r=\pm L_i\pm
L_j,i<j)$ the subgroup of $\widetilde{G}$ generated by
$\tilde{w}_{r}(z)$, $\tilde{H}_r(r=\pm L_i\pm L_j,i<j)$ the subgroup
generated by $\tilde{h}_{r}(z)$.
\end{definition}

 Also, by Lemma \ref{le:5}, it is clear that certain relations
hold both in $\widetilde{G}$ and $G_\RR$. We record these results in
2 separate lemmas(Lemma \ref{le:21} and Lemma \ref{le:16}), since
they will serve as ready references later.
\begin{lemma}\label{le:21}
If $a\in\CC^{m-n}\backslash 0$, $z\in\CC^*$, $t\in\RR^*$ the
following hold in $\widetilde{G}$:
\begin{itemize}
\item[1] $\tilde{w}_{L_n}(0,a)\tilde{w}_{L_{n-1}-L_n}(z)\tilde{w}_{L_n}(0,a)^{-1}=\tilde{w}_{L_{n-1}+L_n}(-\frac{1}{2}\abs{a}^{2}z)$,

\item[2] $\tilde{w}_{L_n}(0,a)\tilde{w}_{L_{n-1}+L_n}(z)\tilde{w}_{L_n}(0,a)^{-1}=\tilde{w}_{L_{n-1}-L_n}(-2\abs{a}^{-2}z)$,

\item[3] $\tilde{w}_{L_{n-1}-L_n}(z)\tilde{w}_{L_n}(0,a)\tilde{w}_{L_{n-1}-L_n}(z)^{-1}=\tilde{w}_{L_{n-1}}(0,az)$,

\item[4] $\tilde{w}_{L_{n-1}-L_n}(z)\tilde{w}_{L_{n-1}}(0,a)\tilde{w}_{L_{n-1}-L_n}(z)^{-1}=\tilde{w}_{L_{n}}(0,-az^{-1})$,

\item[5]
$\tilde{w}_{L_{n-1}-L_n}(z)\tilde{w}_{2L_{n}}(t)\tilde{w}_{L_{n-1}-L_n}(z)^{-1}=\tilde{w}_{2L_{n-1}}(t\abs{z}^2)$,

\item[6]$\tilde{w}_{2L_{n}}(t)\tilde{w}_{L_{n-1}-L_n}(z)\tilde{w}_{2L_{n}}(t)^{-1}=\tilde{w}_{L_{n-1}+L_n}(-tz
\emph{i})$.
\end{itemize}
Hence,
\begin{itemize}
\item[5]$\tilde{h}_{L_{n-1}-L_n}(z)\tilde{w}_{L_n}(0,a)\tilde{h}_{L_{n-1}-L_n}(z)^{-1}=\tilde{w}_{L_{n}}(0,az^{-1})$,

\item[6]$\tilde{h}_{L_{n-1}-L_n}(z)\tilde{w}_{2L_n}(t)\tilde{h}_{L_{n-1}-L_n}(z)^{-1}=\tilde{w}_{2L_{n}}(t\abs{z}^{-2})$,

\item[7]$\tilde{w}_{L_n}(0,a)\tilde{h}_{L_{n-1}-L_n}(z)\tilde{w}_{L_n}(0,a)^{-1}\\
=\tilde{h}_{L_{n-1}+L_n}\left(-\frac{1}{2}\abs{a}^{2}z\right)\tilde{h}_{L_{n-1}+L_n}\left(-\frac{1}{2}\abs{a}^{2}\right)^{-1}$,

\item[8]$\tilde{w}_{2L_{n}}(t)\tilde{h}_{L_{n-1}-L_n}(z)\tilde{w}_{2L_{n}}(t)^{-1}=\tilde{h}_{L_{n-1}+L_n}(-tz
\emph{i})\tilde{h}_{L_{n-1}+L_n}(-t\emph{i})^{-1}$.
\end{itemize}
\end{lemma}

We denote by $S_{\RR}^{i}$ the sphere in $\RR^{i+1}$ and by
$S_{\CC}^{i}$ the sphere in $\CC^{i+1}$. Let $\tilde{W}_{s}$ be the
subgroup generated by $\tilde{w}_{L_n}(0,\sqrt{2}a)$, $a\in
S_{\CC}^{m-n-1}$. If $a=(a_1,\dots,a_n)$, then
$$\pi_1\bigl(\tilde{w}_{L_n}(0,\sqrt{2}a)\bigr)=p(\pi)\text{diag}\bigl((-1)_n,(-1)_{2n},B_{2n+1}\bigr),$$
where $\pi$ permutes n and 2n while fixes other numbers and $B\in
U(m-n)$ with entries $B_{i,j}=-2\overline{a_i}a_j$, for $i\neq j$
and $B_{i,i}=1-2\abs{a_i}^2$. Then $B$ is a reflection in the
hyperplane orthogonal to $\overline{a}$. Thus for any $w\in
\tilde{W}_{s}$,
$$\pi_1(w)=p(\pi)\text{diag}\bigl((-1)^\delta_n,(-1)^\delta_{2n},B_{2n+1}\bigr),$$
where $\delta=2$ if $p(\pi)=I_{m+n}$ and $B\in SU(m-n)$; $\delta=1$
if $p(\pi)$ permutes $n$ and $2n$ and and $B\in U(m-n)$ with
determinant $-1$. Without confusion, we identify
$\pi_1(\tilde{w}_{L_n}(0,\sqrt{2}a))$ and $B$.

Arguments similar to those in Lemma \ref{le:14} show the following
lemma:
\begin{lemma}\label{le:16}
If $w\in \tilde{W}_s$,
$\pi_1(w)=p(\pi)\diag((-1)^\delta_n,1,(-1)^\delta_{2n},B_{2n+1})$,
$\delta=1$ or $2$, $B\in U(m-n)$, $a\in S_{\CC}^{m-n-1}$,
$b\in\CC^{m-n}$, $t\in\RR^*$ then
\begin{align}
w\tilde{w}_{L_n}(0,\sqrt{2}a)w^{-1}&=\left\{\begin{aligned}
 \tilde{w}_{L_n}(0,\sqrt{2}\overline{B}\cdot a), & \quad &\text{ if } p(\pi)=I_{m+n},\\
 \tilde{w}_{L_{n}}(0,-\sqrt{2}\overline{B}\cdot a), & \quad &\text{ if } p(\pi)\neq I_{m+n}\label{for:14},\\
 \end{aligned}
 \right.\\
 w\tilde{x}_{L_n}(t,b)w^{-1}&=\left\{\begin{aligned}
 \tilde{x}_{L_n}(t,\overline{B}\cdot b), & \quad &\text{ if } p(\pi)=I_{m+n},\\
 \tilde{x}_{-L_n}(-t,-\overline{B}\cdot b), & \quad &\text{ if } p(\pi)\neq I_{m+n}\label{for:14}.\\
 \end{aligned}
 \right.
\end{align}
where $\cdot$ means linear operation on vectors.
\end{lemma}
We now determine the ``chain" for $x_\alpha(t,a)$.
\begin{lemma}\label{le:25}
For $x_\alpha(t,a)(\neq e)\in U^{\alpha}_\RR(\alpha=\pm L_n)$, the
``chain" corresponding to it is given by
\begin{align*}
&x_i=x_\alpha\bigl(t,\overline{a_0^i}a_0^{-i}a\bigr), i\in \ZZ;\\
&
y_i=x_{-\alpha}\bigl(\abs{a_0}^{-2}t,-\overline{a_0^i}a_0^{-i-1}a\bigr),
i\in \ZZ,
\end{align*}
where $a_0=-\frac{1}{2}\abs{a}^2+t\emph{i}$. Denoting the element
$w_\alpha(x_\alpha(t,a))$ by $w_\alpha(t,a)$, we have
\begin{align*}
w_\alpha(t,a)&=x_\alpha(t,a)x_{-\alpha}\bigl(\abs{a_0}^{-2}t,-a_0^{-1}a\bigr)x_\alpha\bigl(t,\overline{a_0}a_0^{-1}a\bigr).
\end{align*}
\end{lemma}
\begin{proof} For $a$, we can find $B\in SU(m-n)$ such that $B\cdot a=(\abs{a},\dots,0)$. Denote
$(\abs{a},\dots,0)$ by $a'$. By remarks after Lemma \ref{le:21}, we
can find $b_i\in S_{\CC}^{m-n-1}$ such that $$\pi_1\bigl(\prod_i
w_{\alpha}(\sqrt{2}b_i)\bigr)=\overline{B}.$$ Let $w=\prod_i
w_{\alpha}(\sqrt{2}b_i)$. Using Lemma \ref{le:16} we have
\begin{align*}
&wx_\alpha\bigl(t,\overline{a_0^i}a_0^{-i}a\bigr)x_{-\alpha}\bigl(\abs{a_0}^{-2}t,-\overline{a_0^i}a_0^{-i-1}a\bigr)x_\alpha\bigl(t,\overline{a_0^{i+1}}a_0^{-i-1}a\bigr)w^{-1}\\
&=x_\alpha\bigl(t,\overline{a_0^i}a_0^{-i}a'\bigr)x_{-\alpha}\bigl(\abs{a_0}^{-2}t,-\overline{a_0^i}a_0^{-i-1}a'\bigr)x_\alpha\bigl(t,\overline{a_0^{i+1}}a_0^{-i-1}a'\bigr).
\end{align*}
In [\cite{Deodhar}, p. 30], it was proved
$$w_\alpha(t,a')=x_\alpha\bigl(t,\overline{a_0^i}a_0^{-i}a'\bigr)x_{-\alpha}\bigl(\abs{a_0}^{-2}t,-\overline{a_0^i}a_0^{-i-1}a'\bigr)x_\alpha\bigl(t,\overline{a_0^{i+1}}a_0^{-i-1}a'\bigr)\in
N(S)_\RR$$ and in fact acts as the reflection with respect to
$\alpha$. Since $w\in Z(S)_{\RR}$, it  follows that
$x_{-\alpha}\bigl(\abs{a_0}^{-2}t,-\overline{a_0^i}a_0^{-i-1}a\bigr)=y_i$
and $x_\alpha\bigl(t,\overline{a_0^{i+1}}a_0^{-i-1}a\bigr)=x_i$ are
the ``right" elements in the chain of $x_\alpha(t,a)$ and
$$w_\alpha(t,a)=x_\alpha\bigl(t,\overline{a_0^i}a_0^{-i}a\bigr)x_{-\alpha}\bigl(\abs{a_0}^{-2}t,-\overline{a_0^i}a_0^{-i-1}a\bigr)x_\alpha\bigl(t,\overline{a_0^i}a_0^{-i}a\bigr).$$
And we have
\begin{align*}
w_\alpha(t,a)=w^{-1}w_\alpha(t,a')w.
\end{align*}
Repeating the argument, we find that
\begin{align*}
w_\alpha(t,a)=x_{-\alpha}\bigl(\abs{a_0}^{-2}t,-\overline{a_0^i}a_0^{-i-1}a\bigr)x_\alpha\bigl(t,\overline{a_0^{i+1}}a_0^{-i-1}a\bigr)
x_{-\alpha}\bigl(\abs{a_0}^{-2}t,-\overline{a_0^{i+1}}a_0^{-i-2}a\bigr).
\end{align*}
This proves the lemma completely.
\end{proof}

\begin{definition}\label{def:3}
Denote $w_{L_n}(t,z)=w_{L_n}(t,z,0,\dots,0)$ for
$(t,z)\in(\RR\times\CC)\backslash 0$. Let $W_{L_n}$ be the subgroup
generated by all $w_{L_n}(t,a)$ where
$(t,a)\in(\RR\times\CC^{m-n})\backslash 0$ and $H_{L_n}$ the
subgroup generated by all $w_{L_n}(t_1,a_1)w_{L_n}(t_2,a_2)^{-1}$
where $(t_1,a_1)$, $(t_2,a_2)\in(\RR\times\CC^{m-n})\backslash 0$.
Denote by $W_u$ the subgroup generated by $w_{L_n}(t,z)$, where
$(t,z)\in(\RR\times\CC)\backslash 0$ and denote by $H_u$ the
subgroup generated by all $w_{L_n}(t_1,z_1)w_{L_n}(t_2,z_2)^{-1}$
where $(t_1,z_1),(t_2,z_2)\in(\RR\times\CC)\backslash 0$. Denote by
$W_v$ the subgroup generated by $w_{L_n}(0,a)$ where
$a\in\CC^{m-n}\backslash 0$ and denote by $H_v$ the subgroup
generated by $w_{L_n}(0,a)w_{L_n}(0,b)^{-1}$ where
$a,b\in\CC^{m-n}\backslash 0$. Let $\tilde{W}_{L_n}$,
$\tilde{H}_{L_n}$, $\tilde{W}_u$, $\tilde{W}_v$, $\tilde{H}_v$ and
$\tilde{H}_u$ be the corresponding subgroups in $\widetilde{G}$.
\end{definition}

For $(t,a)\in(\RR\times\CC^{m-n})\backslash 0$, where
$a=(a_1,...,a_n)$. $a_0=-\frac{1}{2}\abs{a}^2+t \textrm{i}$ then
$$\pi_1\bigl(\tilde{w}_{L_n}(t,a)\bigr)=p(\pi)\text{diag}\bigl((\overline{a_0^{-1}})_n,(a_0)_{2n},B_{2n+1}\bigr),$$
where $\pi$ permutes n and 2n while fixes other numbers and $B\in
U(m-n)$(determinant of $B$ is $-\overline{a_0}a_0^{-1}$) with
entries $B_{i,j}=\overline{a_i}a_ja_0^{-1}$, for $i\neq j$ and
$B_{i,i}=1+\abs{a_i}^2a_0^{-1}$. Without confusion, we identify
$\pi_1(w_{L_n}(t,a))$ and $B$. The following lemma are proved easily
by computations using the Steinberg's relations [\cite{Steinberg2},
p. 40] or by using Lemma \ref{le:5}:
\begin{lemma}\label{le:27}
For $(t,a),(t_1,b)\in(\RR\times\CC^{m-n})\backslash 0$, $z\in\CC^*$,
$\pi_1\bigl(\tilde{w}_{L_n}(t,a)\bigr)=p(\pi)\diag\bigl((\overline{a_0^{-1}})_n,(a_0)_{2n},B_{2n+1}\bigr)$
where $a_0=(-\frac{1}{2}\abs{a}^2+t\emph{i})$. we have
\begin{itemize}
\item[1]$\tilde{w}_{L_n}(t,a)\tilde{w}_{L_n}(t_1,b)\tilde{w}_{L_n}(t,a)^{-1}=\tilde{w}_{-L_n}(t_1\abs{a_0}^{-2},\overline{a_0^{-1}B}\cdot b)$,

\item[2]$\tilde{w}_{L_n}(t,a)\tilde{w}_{L_{n-1}-L_n}(z)\tilde{w}_{L_n}(t,a)^{-1}=\tilde{w}_{L_{n-1}+L_n}(z\overline{a_0})$,

\item[3]$\tilde{w}_{L_n}(t,a)\tilde{w}_{L_{n-1}+L_n}(z)\tilde{w}_{L_n}(t,a)^{-1}=\tilde{w}_{L_{n-1}-L_n}(z_1a_0^{-1})$,

\item[4]$\tilde{w}_{L_{n-1}+L_n}(z)\tilde{w}_{L_n}(t,a)\tilde{w}_{L_{n-1}+L_n}(z)^{-1}=\tilde{w}_{-L_{n-1}}(t\abs{z}^{-2},\overline{z_1^{-1}}a)$,

\item[5]$\tilde{w}_{L_{n-1}-L_n}(z)\tilde{w}_{L_n}(t,a)\tilde{w}_{L_{n-1}-L_n}(z)^{-1}=\tilde{w}_{L_{n-1}}(t\abs{z}^2,za)$.
\end{itemize}

\end{lemma}
Notice when $m\geq n+1$, $\tilde{w}_{L_n}(t,0)=\tilde{w}_{2L_n}(t)$
where $t\in\RR^*$.

 For $z_1,z_2\in\CC^{*}$, we define:
\begin{align*}
&\{z_1,z_2\}=\tilde{h}_{L_1-L_2}(z_1)\tilde{h}_{L_1-L_2}(z_2)\tilde{h}_{L_1-L_2}(z_1z_2)^{-1}.
\end{align*}
In exactly the same manner as proofs of Lemma \ref{le:23} and
\ref{le:24}, we have the followings:
\begin{lemma}\label{le:36}
\begin{align*}
&(1) \ker(\pi_1)\cap \tilde{H}_{L_1-L_2}=\{\prod_i
\tilde{h}_{L_1-L_2}(z_i)\mid \text{ \emph{with} }\prod_i z_i=1\}.\\
&(2)  \ker(\pi_1)\cap \tilde{H}_{r}= \ker(\pi_1)\cap
\tilde{H}_{L_1-L_2}, \qquad \text{ \emph{for} }r=\pm L_i\pm
L_j(i\neq j).
\end{align*}
\end{lemma}
\begin{lemma}\label{le:28}
\begin{align*}
\{z_1, z_2\}&=\{z_2,z_1\}^{-1}\qquad \forall z_1,z_2\in \CC^*,\\
\{z_1, z_2\cdot z_3\}&=\{z_1,z_2\}\cdot\{z_1,z_3\}\qquad \forall z_1, z_2,z_3\in \CC^*,\\
\{z_1\cdot z_2,
z_3\}&=\{z_1,z_3\}\cdot\{z_2,z_3\}\qquad \forall z_1, z_2,z_3\in \CC^*,\\
\{z,1-z\}&=1\qquad \forall z\in \CC^*, z\neq 1,\\
\{z,-z\}&=1\qquad \forall z\in\CC^*.
\end{align*}
\end{lemma}
Thus we define a symbol on $\CC$.

\subsection{Structure of $\ker(\pi_1)$}
If $m\geq n+2$, for any $(a,b)\in S^1_\CC$, $j\leq m-n-1$, we define
\begin{align*}
\tilde{h}^{j}_{L_n}(0,\sqrt{2}a,\sqrt{2}b)&=\tilde{w}_{L_n}(0,\dots,0,\underset{j+1}{\sqrt{2}a},\underset{j+2}{\sqrt{2}b},0,\dots,0)\\
&\cdot
\tilde{w}_{L_n}(0,\dots,0,\underset{j+1}{-\sqrt{2}},0,\dots,0).
\end{align*}

Let $\tilde{H}_{s_0}$ denote the subgroup generated by
$\tilde{h}^j_{L_n}(0,\sqrt{2}a, \sqrt{2}b), ((a,b)\in S_{\CC}^{1})$
and $\tilde{H}_{0}$ denote the cyclic group generated by
$\tilde{h}_{2n}(-1)\tilde{h}_{2n}(-1)$ and $\tilde{H}_{c}$ denote
the cyclic group generated by $\tilde{h}_{2n}(-1)$.

An important step in proving Theorem \ref{th:4} is:
\begin{theorem}\label{th:9}
If $n\leq m\leq n+1$, $\ker(\pi_1)=\bigl(\ker(\pi_1)\cap
\tilde{H}_{L_1-L_2}\bigl)\cdot \tilde{H}_0$; if $m\geq n+2$,
$\ker(\pi_1)=\bigl(\ker(\pi_1)\cap \tilde{H}_{L_1-L_2}\bigr)\cdot
\tilde{H}_0\cdot\bigl(\ker(\pi_1)\cap \tilde{H}_{s_0}\bigr)$.
\end{theorem}
The proof of this theorem relies on the following result. Recall the
notation set in Definition \ref{def:3}. We have
\begin{lemma}\label{le:26}
\begin{itemize}
\item[(i)]$\tilde{H}_{2L_n}\subseteq \tilde{H}_{L_{n-1}-L_n}\cdot
\tilde{H}_{L_{n-1}+L_n}\cdot \tilde{H}_c$.

\item[(ii)]$\ker(\pi_1)\cap\tilde{H}_{2L_n}\subseteq(\ker(\pi_1)\cap\tilde{H}_{L_{1}-L_2})\cdot\tilde{H}_0$.

\item[(iii)]$\tilde{H}_v\subseteq \tilde{H}_{L_{n-1}-L_n}\cdot
\tilde{H}_{L_{n-1}+L_n}\cdot\tilde{H}_{s_0}$.

\item[(iv)]$\ker(\pi_1)\cap\tilde{H}_v\subseteq (\ker(\pi_1)\cap\tilde{H}_{L_{n-1}-L_n})\cdot
(\ker(\pi_1)\cap\tilde{H}_{s_0})$.

\item[(v)]$\ker(\pi_1)\cap\tilde{H}_{L_n}\subseteq \bigl(\ker(\pi_1)\cap \tilde{H}_{L_1-L_2}\bigr)\cdot
\tilde{H}_0\cdot\bigl(\ker(\pi_1)\cap \tilde{H}_{s_0}\bigr)$
\end{itemize}
\end{lemma}
\begin{proof}
(i) Using Lemma \ref{le:21}, for $\forall t\in\RR^*$, let
$z\in\CC^*$ such that $\abs{z}=\abs{t}$. We have
\begin{align*}
\tilde{h}_{2L_n}(t)&=\tilde{w}_{2L_n}(t)\tilde{w}_{2L_n}(-1)\\
&=\tilde{h}_{L_{n-1}-L_n}(z^{-\frac{1}{2}})\tilde{w}_{2L_n}(t\abs{z}^{-1})\tilde{h}_{L_{n-1}-L_n}(z^{-\frac{1}{2}})^{-1}\tilde{w}_{2L_n}(-1)\\
&=\tilde{h}_{L_{n-1}-L_n}(z^{-\frac{1}{2}})\tilde{w}_{2L_n}(t\abs{t}^{-1})\tilde{h}_{L_{n-1}-L_n}(z^{-\frac{1}{2}})^{-1}\tilde{w}_{2L_n}(-1).
\end{align*}
If $t>0$ we have
\begin{align*}
&\tilde{h}_{L_{n-1}-L_n}(z^{-\frac{1}{2}})\tilde{w}_{2L_n}(t\abs{t}^{-1})\tilde{h}_{L_{n-1}-L_n}(z^{-\frac{1}{2}})^{-1}\tilde{w}_{2L_n}(-1)\\
=&\tilde{h}_{L_{n-1}-L_n}(z^{-\frac{1}{2}})(\tilde{w}_{2L_n}(1)\tilde{h}_{L_{n-1}-L_n}(z^{-\frac{1}{2}})^{-1}\tilde{w}_{2L_n}(-1))\\
=&\tilde{h}_{L_{n-1}-L_n}(z^{-\frac{1}{2}})\tilde{h}_{L_{n-1}+L_n}(-
\textrm{i})\tilde{h}_{L_{n-1}+L_n}(-z^{-\frac{1}{2}}
\textrm{i})^{-1}.
\end{align*}
If $t<0$ we have
\begin{align*}
&\tilde{h}_{L_{n-1}-L_n}(z^{-\frac{1}{2}})\tilde{w}_{2L_n}(t\abs{t}^{-1})\tilde{h}_{L_{n-1}-L_n}(z^{-\frac{1}{2}})^{-1}\tilde{w}_{2L_n}(-1)\notag\\
=&\tilde{h}_{L_{n-1}-L_n}(z^{-\frac{1}{2}})\tilde{w}_{2L_n}(-1)\tilde{h}_{L_{n-1}-L_n}(z^{-\frac{1}{2}})^{-1}\tilde{w}_{2L_n}(-1)\notag\\
=&\tilde{h}_{L_{n-1}-L_n}(z^{-\frac{1}{2}})\tilde{h}_{L_{n-1}+L_n}(\textrm{i})\tilde{h}_{L_{n-1}+L_n}(z^{-\frac{1}{2}}
\textrm{i})^{-1}\tilde{h}_{2L_n}(-1).
\end{align*}
Especially, if $t=-1$, we have
$$e=\tilde{h}_{L_{n-1}-L_n}(z)\tilde{h}_{L_{n-1}+L_n}(\textrm{i})\tilde{h}_{L_{n-1}+L_n}(z \textrm{i})^{-1},$$
for $\forall z\in S^0_\CC$. By Lemma \ref{le:36}, we have
$$\tilde{h}_{L_{n-1}+L_n}(\textrm{i})\tilde{h}_{L_{n-1}+L_n}(z \textrm{i})^{-1}\in \tilde{h}_{L_{n-1}+L_n}(z^{-1})\cdot\bigl(\ker(\pi_1)
\cap \tilde{H}_{L_{1}-L_2}\bigr),$$ then it follows
\begin{align}\label{for:23}
\tilde{h}_{L_{n-1}-L_n}(z)\tilde{h}_{L_{n-1}+L_n}(z^{-1})\in\ker(\pi_1)\cap
\tilde{H}_{L_{1}-L_2}
\end{align}
for $\forall z\in S^0_\CC$.
 Hence we proved (i).

 (ii) By Lemma \ref{le:36} and (i), any $h\in \tilde{H}_{2L_n}$ can be written as
$$h=\tilde{h}_{L_{n-1}-L_n}(z_1)\tilde{h}_{L_{n-1}+L_n}(z_2)h_1h_2$$
where $z_1,z_2\in\CC^*$, $h_1\in\ker(\pi_1)\cap
\tilde{H}_{L_{1}-L_2}$ and $h_2\in \tilde{H}_c$.

If $\pi_1(h)=I_{m+n}$, we have
$z_1=\overline{z_2}=\overline{z_1^{-1}}$, and $\pi_1(h_2)=I_{m+n}$.
By (\ref{for:23}) we have
$$\tilde{h}_{L_{n-1}-L_n}(z_1)\tilde{h}_{L_{n-1}+L_n}(z_2)\in \ker(\pi_1)\cap \tilde{H}_{L_{1}-L_2}.$$
Notice $\pi_1(\tilde{h}_{2n}(-1))=$diag$((-1)_n,(-1)_{2n})$, it
follows $h_2=\bigl(\tilde{h}_{2n}(-1)\bigl)^{2k}$, $k\in\ZZ$. Hence
we proved (ii).

(iii)  Notice for $(a,b)\in S_{\CC}^1$, we have
$$\pi_1\bigl(\tilde{h}^{j}_{L_n}(0,\sqrt{2}a,\sqrt{2}b)\bigr)=\text{diag}(R_{2n+j}),$$
where $$R_j=\left(
\begin{array}{ccc}
   \abs{a}^2-\abs{b}^2& -2\overline{a}b  \\
   2\overline{b}a& \abs{a}^2-\abs{b}^2\\
\end{array}
\right).$$  Then
$\pi_1\bigl(\tilde{h}^{j}_{L_n}(0,\sqrt{2}a,\sqrt{2}b)\bigl)$,
$j\leq m-n-1$ generate a subgroup isomorphic to $SU(m-n)$. Using
Lemma \ref{le:27}, similar to proofs in Lemma \ref{le:6} and
Corollary \ref{le:9}, it follows
\begin{align*}
\tilde{H}_v\subseteq \tilde{H}_{L_{n-1}-L_n}\cdot
\tilde{H}_{L_{n-1}+L_n}\tilde{H}_{s_0}.
\end{align*}

(iv) By Lemma \ref{le:36} and (iii), any $h\in \tilde{H}_v$ can be
written as
$$h=\tilde{h}_{L_{n-1}-L_n}(z_1)\tilde{h}_{L_{n-1}+L_n}(z_2)h_1h_2$$
where $z_1,z_2\in\CC^*$, $h_1\in\ker(\pi_1)\cap
\tilde{H}_{L_{1}-L_2}$ and $h_2\in \tilde{H}_{s_0}$.

If $\pi_1(h)=I_{m+n}$, we have
$z_1=\overline{z_2}=\overline{z_1^{-1}}$, and $\pi_1(h_2)=I_{m+n}$.
Along the line in proof of (ii), we get (iv).

(v)  By Lemma \ref{le:25}, for $\forall
(t,a)\in(\RR\times\CC^{m-n})\backslash 0$ we can find $w\in
\tilde{H}_{v}$ such that
\begin{align*}
\tilde{w}_{L_n}(t,a)=w^{-1}\tilde{w}_{L_n}(t,\abs{a})w.
\end{align*}
Using Lemma \ref{le:27}, for any $w_1\in \tilde{W}_u$ we have
\begin{align*}
w_1\tilde{W}_v w_1^{-1}\subseteq \tilde{W}_v \qquad\text{ and
}\qquad w_1\tilde{H}_v w_1^{-1}\subseteq \tilde{H}_v.
\end{align*}
Thus it follows
\begin{align*}
\tilde{H}_{L_n}\subseteq \tilde{H}_{v}\cdot \tilde{H}_{u}.
\end{align*}
Then any $h\in\tilde{H}_{L_n}$ can be written as $h=h_vh_u$ where
$h_v\in \tilde{H}_{v}$ and $h_u\in \tilde{H}_{u}$.

If $\pi_1(h)=I_{m+n}$ we have $\pi_1(h_u)=\pi_1(h_v)=I_{m+n}$. Thus
it follows that $\ker(\pi_1)\cap\tilde{H}_{L_n}\subseteq
(\ker(\pi_1)\cap\tilde{H}_{v})\cdot (\ker(\pi_1)\cap\tilde{H}_{u})$.

In [\cite{Deodhar}, p37-p59] Theorem 2.13 asserts that if $m=n+1$,
$\ker(\pi_1)=\ker\pi_1)\cap\tilde{H}_{2L_n}$. It follows
$\ker(\pi_1)\cap \tilde{H}_{u}\subseteq \ker(\pi_1)\cap
\tilde{H}_{2L_n}$. By (ii) and (iv) we get (v).
\end{proof}
\textbf{Proof of Theorem \ref{th:9}}\label{for:19}

If $m=n$, by Remark \ref{re:4}, $\ker(\pi_1)\subseteq (\prod_{r\in
\Delta}\ker(\pi_1)\cap\tilde{H}_{r})$ where
$\Delta=\{L_i-L_{i+1},2L_n\}$. By Lemma \ref{le:36},
$\ker(\pi_1)\cap\tilde{H}_{L_i-L_{i+1}}=\ker(\pi_1)\cap\tilde{H}_{L_{1}-L_{2}}$.
By Lemma \ref{le:26} (2), we are thus though this case.

We are though the case $m=n+1$ by referring to
[\cite{Deodhar},Theorem 2.13] which asserts that if $m=n+1$,
$\ker(\pi_1)=\ker(\pi_1)\cap\tilde{H}_{2L_n}$.

When $m\geq n+2$, by Remark \ref{re:4}, $\ker(\pi_1)\subseteq
(\prod_{r\in \Delta}\ker(\pi_1)\cap\tilde{H}_{r})$ where
$\Delta=\{L_i-L_{i+1},L_n\}$. Along the line in proof of the case
$m=n$, and by Lemma \ref{le:26} (5), we are though this case.
\subsection{Construction of $S_{\RR}^1$-symbols}
We now proceed with the study of $\ker(\pi_1)\cap \tilde{H}_{s_0}$
when $m-n\geq 2$. Up to Lemma \ref{pop:2}, we study properties of
$\tilde{w}_{L_n}(0,\sqrt{2}a)(a\in S_{\CC}^{m-n-1})$ and
$\tilde{h}^i_{L_n}(0,\sqrt{2}b,\sqrt{2}c)\bigl((b,c)\in S^1\bigr)$
which build up the whole $\tilde{H}_{s_0}$, and construct new
symbols on $S_{\RR}^1$ to get prepared for further study of
$\ker(\pi_1)\cap \tilde{H}_{s_0}$.

\begin{definition}\label{de:1} Recall that $\tilde{W}_{s}$ the subgroup
generated by $\tilde{w}_{L_n}(0,\sqrt{2}a)$, where $a\in
S_{\CC}^{m-n-1}$.  We will do calculations inside $\tilde{W}_s$
until the end of Section \ref{sec:14}. For the sake of simplicity,
we denote $\tilde{w}_{L_n}(0,\sqrt{2}a)(a\in S^{m-n-1})$ by
$\tilde{w}_{L_n}(a)$,
$\tilde{h}^j_{L_n}(0,\sqrt{2}a,\sqrt{2}b)(a,b\in S^{1}_\CC)$ by
$\tilde{h}^j_{L_n}(a,b)$.
\end{definition}

\begin{lemma}\label{le:31} For $(a,b), (c,d)\in S_{\RR}^1$, $v\in S_{\CC}^0$, we have
\begin{align*}
(1) &\tilde{h}^i_{L_n}(v,0)=e,\\
(2) &[\tilde{h}^i_{L_n}(a,b),\tilde{h}^i_{L_n}(c,d)]\notag \\
&=\tilde{h}^i_{L_n}\bigl((a^2-b^2,2ab)\cdot(c,d)\bigl)\tilde{h}^i_{L_n}(a^2-b^2,2ab)^{-1}\tilde{h}^i_{L_n}(c,d)^{-1}\notag \\
&=\tilde{h}^i_{L_n}(a,b)\tilde{h}^i_{L_n}(c^2-d^2,2cd)\tilde{h}^i_{L_n}\bigl((a,b)\cdot(c^2-d^2,2cd)\bigl)^{-1},
\end{align*}
\begin{align*}
&(3) [\tilde{h}^i_{L_n}(a,b\emph{i}),\tilde{h}^i_{L_n}(c,d\emph{i})]\notag \\
=&\tilde{h}^i_{L_n}\bigl((a^2-b^2,-2ab\emph{i})\cdot(c,d \emph{i})\bigl)\tilde{h}^i_{L_n}(a^2-b^2,-2ab\emph{i})^{-1}\tilde{h}^i_{L_n}(c,d\emph{i})^{-1}\notag \\
=&\tilde{h}^i_{L_n}(a,b\emph{i})\tilde{h}^i_{L_n}(c^2-d^2,-2cd\emph{i})\tilde{h}^i_{L_n}\bigl((a,b\emph{i})\cdot(c^2-d^2,-2cd\emph{i})\bigl)^{-1},
\end{align*}
where the $\cdot$ follows multiplication rule among quarternions if
we identify any quarternion $(x+y\emph{i}+z\emph{j}+w\emph{k})$ with
$(x+y\emph{i},z+w\emph{i})$.
\end{lemma}
\begin{proof}
(1) Notice $\pi_1(\tilde{h}^i_{L_n}(v,0))=I_{m+n}$ for any $v\in
S_{\CC}^0$. For $\forall u\in S^0_\CC$, choose $h\in
\tilde{H}_{s_0}$ such that
$\pi_1(h)$=diag$(\overline{u}_{2n+i},u_{2n+i+1})$, we have
\begin{align*}
&\tilde{h}^i_{L_n}(v,0)=\tilde{h}\tilde{h}^i_{L_n}(v,0)\tilde{h}^{-1}\\
&=\tilde{w}_{L_n}(0,\dots,(uv)_{i},\dots,0)\tilde{w}_{L_n}(0,\dots,-u_{i},\dots,0)\\
&=\tilde{h}^i_{L_n}(uv,0)\tilde{h}^i_{L_n}(u,0)^{-1}.
\end{align*}
Thus for any $u,v\in S_{\CC}^0$, we have
$$\tilde{h}^i_{L_n}(uv,0)=\tilde{h}^i_{L_n}(u,0)\tilde{h}^i_{L_n}(v,0).$$
Thus we have
\begin{align*}
&\tilde{h}^i_{L_n}(-v^2,0)=\tilde{h}^i_{L_n}(v,0)\tilde{h}^i_{L_n}(-v,0)\\
&=\tilde{w}_{L_n}(0,\dots,v_i,\dots,0)\tilde{w}_{L_n}(0,\dots,(-1)_i,\dots,0)\\
&\cdot \tilde{w}_{L_n}(0,\dots,(-v)_i,\dots,0)\tilde{w}_{L_n}(0,\dots,(-1)_i,\dots,0)\\
&=\tilde{w}_{L_n}(0,\dots,(-1)_i,\dots,0)\tilde{w}_{L_n}(0,\dots,(-1)_i,\dots,0)\\
&=\tilde{h}^i_{L_n}(-1,0).
\end{align*}
Since $\tilde{h}^i_{L_n}(-\textrm{i}^2,0)=\tilde{h}^i_{L_n}(1,0)=e$,
we thus proved (1).

 (2) and (3) By Lemma
\ref{le:16} we have
\begin{align*}
&\tilde{h}^i_{L_n}(a,b)\tilde{w}_{L_n}(0,\dots,c_i,d_{i+1},0\dots,0)\tilde{h}^i_{L_n}(a,b)^{-1}\\
&=\tilde{w}_{L_n}(ca^2-cb^2-2dba,da^2-db^2+2cba)\\
&\tilde{h}^i_{L_n}(a,b \textrm{i})\tilde{w}_{L_n}(0,\dots,c_i,d_{i+1} \textrm{i},0,\dots,0)\tilde{h}^i_{L_n}(a,b\textrm{i})^{-1}\\
&=\tilde{w}_{L_n}(ca^2-cb^2+2dba,(da^2-db^2-2cba)\textrm{i}).
\end{align*}
Similar to the proof of Lemma \ref{le:3}, we get (2) and (3).
\end{proof}
\begin{lemma}\label{pop:2}
For $\forall v\in S_{\CC}^0$,  $\forall a\in S_{\CC}^{m-n-1}$ we
have $\tilde{w}_{L_n}(a)=\tilde{w}_{L_n}(va)$.
\end{lemma}
\begin{proof}
 For $v\in S_{\CC}^0$ $a=(a_1,\dots,a_n)\in S_{\CC}^{m-n-1}$,
define $f:S_{\CC}^{m-n-1}\times S_{\CC}^0\rightarrow
S_{\CC}^{m-n-1}$ to be
$f(a,g)=((2\abs{a_1}^2-1)g,2ga_2\overline{a_1},\dots,2ga_n\overline{a_1})$.
It is easy to check $f$ is surjective. Using Lemma \ref{le:16} and
Lemma \ref{le:31}, for any $u,v\in S_{\CC}^0$ we have
\begin{align*}
e&=\tilde{h}^1_{L_n}(u,0)(\tilde{h}^1_{L_n}(-v,0))^{-1}\\
&=\tilde{w}_{L_n}(u,\dots,0)\tilde{w}_{L_n}(v,\dots,0)\\
&=\tilde{w}_{L_n}(a)\tilde{w}_{L_n}(u,\dots,0)\tilde{w}_{L_n}(v,\dots,0)\tilde{w}_{L_n}(a)^{-1}\\
&=\tilde{w}_{L_n}\bigl((2\abs{a_1}^2-1)u,2ua_2\overline{a_1},\dots,2ua_n\overline{a_1}\bigr)\\
&\cdot
\tilde{w}_{L_n}\bigl((2\abs{a_1}^2-1)v,2va_2\overline{a_1},\dots,2va_n\overline{a_1}\bigr).
\end{align*}
Thus we proved the lemma.
\end{proof}

For $(a,b),(c,d)\in S_{\RR}^1$ we define:
\begin{align*}
&\{(a,b),(c,d)\}^1_i=\tilde{h}^i_{L_n}\bigl((a,b)\cdot(c,d)\bigr)\tilde{h}^i_{L_n}(a,b)^{-1}\tilde{h}^i_{L_n}(c,d)^{-1},\\
&\{(a,b),(c,d)\}^2_i=\tilde{h}^i_{L_n}\bigl((a,b
\textrm{i})\cdot(c,d
\textrm{i})\bigr)\tilde{h}^i_{L_n}(a,b\textrm{i})^{-1}\tilde{h}^i_{L_n}(c,d\textrm{i})^{-1}.
\end{align*}

Let $i, j$ be distinct, and let
\begin{align*}
w&=\tilde{w}_{L_n}(0,\dots,(-\frac{\sqrt{2}}{2})_{i+1},\dots,(\frac{\sqrt{2}}{2})_{j+1},0\dots,0)\\
&\cdot
\tilde{w}_{L_n}(0,\dots,(-\frac{\sqrt{2}}{2})_i,\dots,(\frac{\sqrt{2}}{2})_{j},0,\dots,0).
\end{align*}

By Lemma \ref{le:16}, for $\forall (u,v)\in S^1_\CC$ we have
$$w\tilde{h}^i_{L_n}(u,v)w^{-1}=\tilde{h}^j_{L_n}(u,v).$$

Since $\{(a,b), (c,d)\}^\delta_i\in Z(\tilde{G})$($\delta=1,2$), it
follows that
\begin{lemma}\label{le:20}
\begin{align*}
\{(a,b), (c,d)\}^\delta=\{(a,b), (c,d)\}^\delta_i\qquad \delta=1,2,
\end{align*}
are well defined.
\end{lemma}
Using Lemma \ref{le:31}, in exactly the same manner as the proof in
the appendix of [8], we prove that these $\{(a,b),(c,d)\}$'s satisfy
the conditions
\begin{lemma}\label{le:30}
For all $(a,b),(c,d),(a_1,b_1),(c_1,d_1)\in S_{\RR}^1$ and
$\delta=1,2$ we have:
\begin{align*}
\{(a,b), (c,d)\}^\delta&=(\{(c,d), (a,b)\}^\delta)^{-1},\\
\{(a,b), (c,d)\cdot(c_1,d_1)\}^\delta&=\{(a,b),(c,d)\}^\delta\cdot\{(a,b),(c_1,d_1)\}^\delta,\\
\{(a,b)\cdot(a_1,b_1),
(c,d)\}^\delta&=\{(a,b),(c,d)\}^\delta\cdot\{(a_1,b_1),(c,d)\}^\delta,\\
\{(c,d),(-c,-d)\}^\delta&=1.
\end{align*}
\end{lemma}
Thus we define 2 symbols on $S^1_{\RR}$. Denote by $H_{sym}$ The
subgroup generated by these symbols.

\subsection{Structure of $H^i$}
 Let $H^i_0$ be the subgroup
generated by $\tilde{h}^i_{L_n}(a,b)\bigl((a,b)\in S^1_\RR\bigr)$,
$H^i_1$ be the subgroup generated by
$\tilde{h}^i_{L_n}(a,b\textrm{i})\bigl((a,b )\in S^1_\RR\bigr)$ and
$H^i$ the subgroup generated by $\tilde{h}^i(u,v)\bigl((u,v)\in
S^1_\CC\bigr)$. Using Lemma \ref{le:20}, along the line of  proof of
Lemma \ref{le:12}, We get the following:
\begin{lemma}\label{le:32}
\begin{align*}
&(1) \ker(\pi_1)\cap H^j_\delta=\ker(\pi_1)\cap H^1_\delta\qquad
 j\leq m-n-1, \delta=0,1, \\
&(2) \ker(\pi_1)\cap H^1_\delta=H_{sym}\cap H^1_\delta\qquad
\delta=0,1.
\end{align*}
\end{lemma}

We now make a slight digression to state a fact and prove a lemma
whose roles will be clear from the subsequent development. However,
the fact and the lemma in themselves seems to be interesting.
\begin{fact}\label{fact:2}
If $\pi_1\bigl(\tilde{w}_{L_n}(a)\tilde{w}_{L_n}(b)\bigl)=C$, where
$a,b\in S_{\CC}^{m-n-1}$. Then we have
$$4\abs{\langle a,b\rangle}^2+m-n-4=\textrm{trace}(C),$$ where $\langle a,b\rangle=a\cdot \overline{b}$ is the inner product of $a$ and $b$.
\end{fact}
Denote
$\tilde{w}^j_{L_n}(a,b)=\tilde{w}_{L_n}(0,\dots,a_{j},b_{j+1},0,\dots,0)$
where $(a,b)\in S^1_\CC$.
\begin{lemma}\label{le:10}
If $\langle (a_1,a_2),(b_1,b_2)\rangle=\langle
(c_1,c_2),(d_1,d_2)\rangle$, where $(a_1,a_2)$, $(b_1,b_2)$,
$(c_1,c_2)$, $(d_1,d_2)\in S^1_{\CC}$, there exists $g\in S^0_{\CC}$
and $h\in SU(2)$ such that
$$h\cdot(a_1,a_2)=(gc_1,gc_2),\qquad h\cdot(b_1,b_2)=(gd_1,gd_2).$$
\end{lemma}
\begin{proof} We first consider the case $(c_1,c_2)=(1,0)$.
Direct calculation shows for any $g\in S^0_{\CC}$, if we let
$$ h=\begin{pmatrix}\overline{a_1}g & \overline{a_2}g\\
-a_2\overline{g} & a_1\overline{g}\\
\end{pmatrix},$$
then
\begin{align*}
h\cdot(a_1,a_2)&=(gc_1,0),\\
h\cdot(b_1,b_2)&=\bigl(g(\overline{a_1}b_1+\overline{a_2}b_2),g(-a_2\overline{g}^2b_1+a_1\overline{g}^2b_2)\bigl).
\end{align*}
Since $h$ preserves inner product, we have
\begin{align*}
\langle h\cdot(a_1,a_2),h\cdot(b_1,b_2)\rangle=\langle
(c_1,c_2),(d_1,d_2)\rangle=\overline{d_1}.
\end{align*}
Thus it follows $\overline{a_1}b_1+\overline{a_2}b_2=d_1$,
$\abs{-a_2b_2+a_1b_2}=\sqrt{1-\abs{d_1}^2}=\abs{d_2}$. Hence we can
choose right $g$ such that
$-a_2\overline{g}^2b_2+a_1\overline{g}^2b_2=d_2$. We thus proved the
case for $(c_1,c_2)=(1,0)$.

If $(c_1,c_2)\neq(1,0)$, there exists $h'\in SU(2)$ such that
$h'\cdot(c_1,c_2)=(1,0)$. Then we reduce it to previous case. We
hence proved the lemma.
\end{proof}

 An important step in proving the main Theorem \ref{th:4} is:
\begin{theorem}\label{th:10}
$\ker(\pi_1)\cap H^j=H_{sym}$ for $\forall j\leq m-n-1$.
\end{theorem}

The ensuring discussion up to Lemma \ref{le:13} proves the theorem.
Recall the definition for $\tilde{W}_{s}$ in Definition \ref{de:1}.
Consider the quotient group $\tilde{W}_{s}/H_{sym}$ until the end of
Lemma \ref{le:13}.  Note that $H_{sym}$ being central, this is well
defined. We continue to write $\tilde{h}^j_{L_n}(a,b)\bigl((a,b)\in
S^1_\CC\bigr)$ and $\tilde{w}^j_{L_n}(a)(a\in S^{m-n-1}_\CC)$ for
their images in $\tilde{W}_{s}/H_{sym}$. However, there is no
confusion in doing so.

\begin{lemma}\label{le:37} Let $g_1,g_2:[0,2\pi]^2\rightarrow S^1_\CC$ be defined as follows:
\begin{align*}
&g_1(a,b)=(\cos a\cos b-\sin a\sin b \emph{i},\cos a\sin b-\sin a\cos b \emph{i}),\\
&g_2(a,b)=(\cos a\cos b-\sin a\sin b \emph{i},\sin a\cos b+\cos
a\sin b \emph{i}).
\end{align*}
 For any $a,b,x\in
[0,2\pi]$, there exist $c,d,y,$ $c_1,d_1,y_1\in [0,2\pi]$ such that
\begin{align*}
&(1) \tilde{w}^j_{L_n}\bigl(g_2(a,b)\bigr)\tilde{w}^j_{L_n}(\cos
x,\sin
x)=\tilde{w}^j_{L_n}\bigl(g_1(c,d)\bigr)\tilde{w}^j_{L_n}(\cos
y,\sin y
\emph{i}),\\
&(2) \tilde{w}^j_{L_n}\bigl(g_1(a,b)\bigl)\tilde{w}^j_{L_n}(\cos
x,\sin x
\emph{i})=\tilde{w}^j_{L_n}\bigl(g_2(c_1,d_1)\bigr)\tilde{w}^j_{L_n}(\cos
y_1,\sin y_1).
\end{align*}
\end{lemma}
\begin{proof} (1) Fix $j$. At first, we want to show that there
exist $(z_1,z_2)\in S^1_{\CC}$ and $y\in[0,2\pi]$  such that
\begin{align}\label{for:31}
\pi_1\bigl(\tilde{w}^j_{L_n}(g_2(a,b))\tilde{w}^j_{L_n}(\cos x,\sin
x)\bigr)=\pi_1\bigl(\tilde{w}^j_{L_n}(z_1,z_2)\tilde{w}^j_{L_n}(\cos
y,\sin y \textrm{i})\bigr).
\end{align}
Suppose
$\pi_1\bigl(\tilde{w}^j_{L_n}(g_2(a,b))\tilde{w}^j_{L_n}(\cos x,\sin
x)\bigr)$ is given by the following matrix
$$ \begin{pmatrix}\alpha & \beta\\
-\overline{\beta} & \overline{\alpha}\\
\end{pmatrix}.$$
If $\alpha\in\RR$, $\beta \textrm{i}\in\RR$, there is nothing to
prove.

 Now suppose $(\alpha,\beta)\notin\RR\times \RR\textrm{i}$.
If $\alpha\notin\RR$, $\beta\notin\RR\textrm{i}$, let
$u=t(1-\alpha^2+\beta^2)$,
$v=\frac{u\overline{\beta}-\beta\overline{u}}{\alpha-\overline{\alpha}}$
where $t\in\RR$ satisfying $\abs{v}^2+\abs{u}^2=1$; if
$\alpha\notin\RR$, $\beta \textrm{i}\in\RR$ let
$u=t\alpha\textrm{i}$,
$v=\frac{u\overline{\beta}-\beta\overline{u}}{\alpha-\overline{\alpha}}$
where $t\in\RR$ satisfying $\abs{v}^2+\abs{u}^2=1$; if
$\alpha\in\RR$, $\beta\notin\RR\textrm{i}$, let $u=t\beta$,
$v=-\frac{u\overline{\alpha}+\alpha\overline{u}}{\overline{\beta}+\beta}$
where $t\in\RR$ satisfying $\abs{v}^2+\abs{u}^2=1$.

Let $R$ be the following matrix
$$ \begin{pmatrix}v & u\\
\overline{u} & -v\\
\end{pmatrix},$$
then we have $R\cdot(\alpha,-\overline{\beta})\in\RR\times\RR
\textrm{i}$. Thus there exists $y\in\RR$ such that
$R\cdot(\alpha,-\overline{\beta})=(-\cos 2y,\sin 2y \textrm{i})$.

Let
$(z_1,z_2)=(\sqrt{\frac{1}{2}(1-v)},-\frac{u}{2\sqrt{\frac{1}{2}(1-v)}})$
if $v\neq 1$; $(z_1,z_2)=(0,1)$ if $v=1$, then we have
$\pi_1\bigl(\tilde{w}^j_{L_n}(z_1,z_2)\bigl)=R$ and
\begin{align*}
\pi_1\bigl(\tilde{w}^j_{L_n}(g_2(a,b))\tilde{w}^j_{L_n}(\cos x,\sin
x)\bigr)=\pi_1\bigl(\tilde{w}^j_{L_n}(z_1,z_2)\tilde{w}^j_{L_n}(\cos
y,\sin y \textrm{i})\bigr).
\end{align*}
Next, we want  to show there exists $(c,d)\in\RR^2$ such that
$$\pi_1\bigl(\tilde{w}^j_{L_n}(g_1(c,d))\bigl)=\pi_1\bigl(\tilde{w}^j_{L_n}(z_1,z_2)\bigl).$$

By Lemma \ref{pop:2}, we just need to show there exist $(c,d)$ and
$z\in S^0_{\CC}$ such that $g_1(c,d)=(zz_1,zz_2)$. Notice
$z_1\in\RR$. Let $A$ denote
the following matrix:$$ \begin{pmatrix}\cos c & \sin c \textrm{i}\\
\sin c \textrm{i} & \cos c\\
\end{pmatrix}.$$
If $z_2\in\RR$, let $\sin c=0$, $z=1$, then $A\cdot(zz_1,zz_2)\in
S^1_{\RR}$. If $z_1=0$, let $\sin c=0$,
$z=\frac{\overline{z_2}}{\abs{z_2}}$, then $A\cdot(zz_1,zz_2)\in
S^1_{\RR}$.

Now suppose $z_1z_2\neq0$ and $z_2\notin\RR$.  If $z_2\in\RR
\textrm{i}$, then $(z_1,z_2)=(\cos r,\sin r \textrm{i})$ for some
$r\in\RR$. Let $z=1$, $c=-r$, then $A\cdot(zz_1,zz_2)\in S^1_{\RR}$.

Now suppose $z_1z_2\neq0$,  $z_2\notin\RR$ and $z_2\notin\RR
\textrm{i}$. Let
$z=\frac{\sqrt{z_1^2-\overline{z_2}^2}}{\abs{\sqrt{z_1^2-\overline{z_2}^2}}}$,
$\cot c=\frac{(zz_2+\overline{zz_2})
\textrm{i}}{z_1\overline{z}-zz_1}$(notice in this case
$z_1^2-\overline{z_2}^2\neq 0$ and $\overline{z}z_1-zz_1\neq 0$),
then we have $A\cdot(zz_1,zz_2)\in S^1_{\RR}$. Let $d\in\RR$
satisfying  $(\cos d,\sin d)=A\cdot(zz_1,zz_2)$, then
$g_1(c,d)=(zz_1,zz_2)$.

Hence we proved that for any given $a,b$, any $x$ we can find
$(c,d)\in [0,2\pi]$ and $y\in[0,2\pi]$ such that
\begin{align*}
\pi_1\bigl(\tilde{w}^j_{L_n}(g_2(a,b))\tilde{w}^j_{L_n}(\cos x,\sin
x)\bigr)=\pi_1\bigl(\tilde{w}^j_{L_n}(g_1(c,d))\tilde{w}^j_{L_n}(\cos
y,\sin y \textrm{i})\bigr).
\end{align*}
By Fact \ref{fact:2}
 we have
\begin{align*}
\abs{\bigl\langle g_2(c,d),(\cos x,\sin
x)\bigr\rangle}=\abs{\bigl\langle
 g_1(a,b),(\cos y,\sin y \textrm{i})\bigr\rangle}.
\end{align*}
There exists $z_1\in S_{\CC}^1$ such that
\begin{align*}
\bigl\langle g_2(c,d),(\cos x,\sin x)\bigr\rangle=\bigl\langle
 z_1g_1(a,b),(\cos y,\sin y \textrm{i})\bigr\rangle.
\end{align*}

By Lemma \ref{le:10}, there exist $z_0\in S^0_{\CC}$ and $h'\in H^j$
such that
\begin{align*}
\overline{\pi_1(h')}(g_2(c,d))=z_0z_1g_1(a,b)\text{ and
}\overline{\pi_1(h')}(\cos x,\sin x)=(z_0\cos y,z_0\sin y
\textrm{i}).
\end{align*}
 Then it follows
\begin{align}\label{for:16}
&h'\bigl(\tilde{w}^j_{L_n}(g_2(a,b))\tilde{w}^j_{L_n}(\cos x,\sin
x)\bigr)h'^{-1}\notag\\
&=\tilde{w}^j_{L_n}(z_0z_1g_1(a,b))\tilde{w}^j_{L_n}(z_0\cos
y,z_0\sin y
\textrm{i})\notag\\
&=\tilde{w}^j_{L_n}\bigl(g_1(a,b)\bigr)\tilde{w}^j_{L_n}(\cos y,\sin
y \textrm{i}).&(\text{ by }\text{ Lemma }\ref{pop:2})
\end{align}
If $\pi_1\bigl(\tilde{w}^j_{L_n}(g_2(a,b))\tilde{w}^j_{L_n}(\cos
x,\sin x)\bigr)=I_{m+n}$, there is nothing to prove.

 Now suppose
$\pi_1\bigl(\tilde{w}^j_{L_n}(g_2(a,b))\tilde{w}^j_{L_n}(\cos x,\sin
x)\bigr)\neq I_{m+n}$.

Let $x_1,x_2\in\RR$ such that $$\abs{\cos x_1\cos x_2+\sin x_1\sin
x_2}=\abs{\bigl\langle g_2(a,b),(\cos x,\sin x)\bigr\rangle}.$$ Then
there exists $z\in S^0_\CC$ such that
\begin{align*}
\bigl\langle(\cos x_1,\sin x_1),(\cos x_2,\sin
x_2)\bigr\rangle=\bigl\langle zg_2(a,b),(\cos x,\sin x)\bigr\rangle,
\end{align*}
and by Lemma \ref{le:10}, there exist $h_1\in H^j$ and $z_2\in
S^0_{\CC}$ satisfying
\begin{align*}
&\overline{\pi_1(h_1)}(zg_2(a,b))=(z_2\cos x_1,z_2\sin x_1),\\
&\overline{\pi_1(h_1)}(\cos x,\sin x)=(z_2\cos x_2,z_2\sin x_2 ).
\end{align*}
Thus following (\ref{for:16}) and Lemma \ref{pop:2} we have
\begin{align*}\label{for:9}
&h_1\bigl(\tilde{w}^j_{L_n}(g_1(a,b))\tilde{w}^j_{L_n}(\cos y,\sin y
\textrm{i})\bigr)h_1^{-1}\\
 &=h_1h'\bigl(\tilde{w}^j_{L_n}(g_2(a,b))\tilde{w}^j_{L_n}(\cos x,\sin
x)\bigr)h'^{-1}h_1^{-1}\notag\\
&=(h_1h'h_1^{-1})\bigl(h_1\tilde{w}^j_{L_n}(zg_2(a,b))\tilde{w}^j_{L_n}(\cos x,\sin x)h_1^{-1}\bigr)\cdot (h_1h'h_1^{-1})^{-1}\\
&=(h_1h'h_1^{-1})\bigl(\tilde{w}^j_{L_n}(z_2\cos x_1,z_2\sin x_1
)\tilde{w}^j_{L_n}(z_2\cos x_2,z_2\sin
x_2)\bigr)(h_1h'h_1^{-1})^{-1}\notag\\
&=(h_1h'h_1^{-1})\bigl(\tilde{w}^j_{L_n}(\cos x_1,\sin x_1
)\tilde{w}^j_{L_n}(\cos x_2,\sin x_2)\bigr)(h_1h'h_1^{-1})^{-1}.
\end{align*}
Notice
\begin{align*}
&\pi_1\bigl(h_1\tilde{w}^j_{L_n}(g_1(a,b))\tilde{w}^j_{L_n}(\cos
y,\sin y \textrm{i})h_1^{-1}\bigr)\\
&=\pi_1\bigl(h_1\tilde{w}^j_{L_n}(zg_2(a,b))\tilde{w}^j_{L_n}(\cos
x,\sin x)h_1^{-1}\bigr)\\
&=\pi_1\bigl(\tilde{w}^j_{L_n}(\cos x_1,\sin x_1
)\tilde{w}^j_{L_n}(\cos x_2,\sin x_2)\bigr).
\end{align*}
 Since $\pi_1(h_1h'h_1^{-1})$ commute with
$\pi_1\bigl(\tilde{w}^j_{L_n}(\cos x_1,\sin x_1
)\tilde{w}^j_{L_n}(\cos x_2,\sin x_2)\bigr)$ which is conjugate with
$\pi_1\bigl(\tilde{w}^j_{L_n}(g_1(a,b))\tilde{w}^j_{L_n}(\cos y,\sin
y \textrm{i})\bigr)\neq I_{m+n}$ and is in $\pi_1(H^j_0)$, hence
there exists $h\in H^j_0$ and $h_0\in\ker(\pi_1)$ such that
$h_1h'h_1^{-1}=hh_0$.

Hence we have
\begin{align*}
&h_1\tilde{w}^j_{L_n}(g_1(a,b))\tilde{w}^j_{L_n}(\cos y,\sin y
\textrm{i})h_1^{-1}\\
&=(h_1h'h_1^{-1})\bigl(\tilde{w}^j_{L_n}(\cos
x_1,\sin x_1 )\tilde{w}^j_{L_n}(\cos x_2,\sin x_2)\bigr)(h_1h'h_1^{-1})^{-1}\\
&=hh_0\bigl(\tilde{w}^j_{L_n}(\cos x_1,\sin x_1
)\tilde{w}^j_{L_n}(\cos x_2,\sin x_2)\bigl)(hh_0)^{-1}\\
&=\tilde{w}^j_{L_n}(\cos x_1,\sin x_1 )\tilde{w}^j_{L_n}(\cos x_2,\sin x_2)(\text{ by Lemma }\ref{le:32})\\
&=h_1\tilde{w}^j_{L_n}(g_2(a,b))\tilde{w}^j_{L_n}(\cos x,\sin
x)h_1^{-1}.
\end{align*}
Hence we have proved (1).

(2) Similar arguments hold for (2).
\end{proof}

 As a second step toward the proof of Theorem \ref{th:10}, we
prove

\begin{lemma}\label{le:13}
For any $\theta_1,\theta_2,\theta_3$, we can find $\beta_1, \beta_2,
\beta_3$, $\alpha_1, \alpha_2, \alpha_3$, such that
\begin{align*}
(1) &\tilde{h}^{j}_{L_n}(\cos\theta_1,\sin\theta_1\emph{i})\tilde{h}^{j}_{L_n}(\cos\theta_2,\sin\theta_2)\tilde{h}^{j}_{L_n}(\cos\theta_3,\sin\theta_3\emph{i})\\
&=
\tilde{h}^j_{L_n}(\cos\beta_1,\sin\beta_1)\tilde{h}^{j}_{L_n}(\cos\beta_2,\sin\beta_2\emph{i})\tilde{h}^j_{L_n}(\cos\beta_3,\sin\beta_3),\\
(2) &\tilde{h}^j_{L_n}(\cos\theta_1,\sin\theta_1)\tilde{h}^{j}_{L_n}(\cos\theta_2,\sin\theta_2\emph{i})\tilde{h}^j_{L_n}(\cos\theta_3,\sin\theta_3)\\
&=
\tilde{h}^{j}_{L_n}(\cos\alpha_1,\sin\alpha_1\emph{i})\tilde{h}^j_{L_n}(\cos\alpha_2,\sin\alpha_2)\tilde{h}^{j}_{L_n}(\cos\alpha_3,\sin\alpha_3\emph{i}).
\end{align*}
\end{lemma}
\begin{proof}
(1) Let $\pi_1\bigl(\tilde{h}^{j}_{L_n}(\cos\theta_1,\sin\theta_1
\textrm{i})\bigl)=A$, we have
\begin{align*}
&\tilde{h}^{j}_{L_n}(\cos\theta_1,\sin\theta_1\textrm{i})\tilde{h}^{j}_{L_n}(\cos\theta_2,\sin\theta_2)\tilde{h}^{j}_{L_n}(\cos\theta_3,\sin\theta_3\textrm{i})\\
&=\tilde{h}^{j}_{L_n}(\cos\theta_1,\sin\theta_1\textrm{i})\tilde{w}^{j}_{L_n}(\cos\theta_2,\sin\theta_2)\tilde{w}^{j}_{L_n}(-1,0)\tilde{h}^{j}_{L_n}(\cos\theta_3,\sin\theta_3\textrm{i})\\
&=\tilde{w}^{j}_{L_n}\bigl(\overline{A}\cdot(\cos\theta_2,\sin\theta_2)\bigl)\tilde{h}^{j}_{L_n}(\cos\theta_1,\sin\theta_1\textrm{i})\\
&\cdot
\tilde{w}^{j}_{L_n}(-1,0)\tilde{h}^{j}_{L_n}(\cos\theta_3,\sin\theta_3\textrm{i})\tilde{w}^{j}_{L_n}(-1,0)\tilde{w}^{j}_{L_n}(1,0)\\
&=\tilde{w}^{j}_{L_n}\bigl(g_1(-2\theta_1,\theta_2)\bigl)\tilde{h}^{j}_{L_n}(\cos\theta_1,\sin\theta_1\textrm{i})\\
&\cdot
\tilde{w}^{j}_{L_n}(-1,0)\tilde{h}^{j}_{L_n}(\cos\theta_3,\sin\theta_3\textrm{i})\tilde{w}^{j}_{L_n}(-1,0)\tilde{w}^{j}_{L_n}(1,0).
\end{align*}
Since
$$\tilde{h}^{j}_{L_n}(\cos\theta_1,\sin\theta_1\textrm{i})\tilde{w}^{j}_{L_n}(-1,0)\tilde{h}^{j}_{L_n}(\cos\theta_3,\sin\theta_3\textrm{i})\tilde{w}^{j}_{L_n}(-1,0)\in
H^j_1,$$ there exists $x_1\in\RR$ such that
\begin{align*}
&\tilde{h}^{j}_{L_n}(\cos\theta_1,\sin\theta_1\textrm{i})\tilde{w}^{j}_{L_n}(-1,0)\tilde{h}^{j}_{L_n}(\cos\theta_3,\sin\theta_3\textrm{i})\tilde{w}^{j}_{L_n}(1,0)\\
&=\tilde{w}^{j}_{L_n}(\cos x_1,\sin
x_1\textrm{i})\tilde{w}^{j}_{L_n}(-1,0).
\end{align*}
Thus we have
\begin{align*}
&\tilde{h}^{j}_{L_n}(\cos\theta_1,\sin\theta_1\textrm{i})\tilde{h}^{j}_{L_n}(\cos\theta_2,\sin\theta_2)\tilde{h}^{j}_{L_n}(\cos\theta_3,\sin\theta_3\textrm{i})\\
&=\tilde{w}^{j}_{L_n}\bigl(g_1(-2\theta_1,\theta_2)\bigl)\tilde{w}^{j}_{L_n}(\cos
x_1,\sin x_1\textrm{i}).
\end{align*}
By Lemma \ref{le:37}, there exist $a,b,x_2\in [0,2\pi]$ such that
$$\tilde{w}^{j}_{L_n}\bigl(g_1(-2\theta_1,\theta_2)\bigl)\tilde{w}^{j}_{L_n}(\cos
x_1,\sin
x_1\textrm{i})=\tilde{w}^{j}_{L_n}\bigl(g_2(2a,b)\bigl)\tilde{w}^{j}_{L_n}(\cos
x_2,\sin x_2).$$ It follows
\begin{align*}
&\tilde{h}^{j}_{L_n}(\cos\theta_1,\sin\theta_1\textrm{i})\tilde{h}^{j}_{L_n}(\cos\theta_2,\sin\theta_2)\tilde{h}^{j}_{L_n}(\cos\theta_3,\sin\theta_3\textrm{i})\\
&=\tilde{w}^{j}_{L_n}\bigl(g_2(2a,b)\bigl)\tilde{w}^{j}_{L_n}(\cos x_2,\sin x_2)\\
&=\tilde{h}^{j}_{L_n}(\cos a,\sin b)\tilde{w}^{j}_{L_n}(\cos b,\sin
b\textrm{i})\tilde{h}^{j}_{L_n}(\cos a,\sin b)^{-1}\\
&\cdot \tilde{w}^{j}_{L_n}(\cos x_2,\sin x_2)\\
&=\tilde{h}^{j}_{L_n}(\cos a,\sin b)\tilde{w}^{j}_{L_n}(\cos b,\sin
b\textrm{i})\tilde{w}^{j}_{L_n}(-1,0)\\
&\cdot \tilde{w}^{j}_{L_n}(1,0)\tilde{h}^{j}_{L_n}(\cos a,\sin b)^{-1}\tilde{w}^{j}_{L_n}(\cos x_2,\sin x_2)\\
&=\tilde{h}^{j}_{L_n}(\cos a,\sin b)\tilde{h}^{j}_{L_n}(\cos b,\sin
b\textrm{i})\\
&\cdot \tilde{w}^{j}_{L_n}(1,0)\tilde{h}^{j}_{L_n}(\cos a,\sin
b)^{-1}\tilde{w}^{j}_{L_n}(\cos x_2,\sin x_2).
\end{align*}
Notice $$\tilde{w}^{j}_{L_n}(1,0)\tilde{h}^{j}_{L_n}(\cos a,\sin
b)^{-1}\tilde{w}^{j}_{L_n}(\cos x_2,\sin x_2)\in H^j_0$$ by Lemma
\ref{le:32} there exists $x\in[0,2\pi]$ such that
\begin{align*}
&\tilde{h}^{j}_{L_n}(\cos x,\sin x)\\
&=\tilde{w}^{j}_{L_n}(1,0)\tilde{h}^{j}_{L_n}(\cos a,\sin
b)^{-1}\tilde{w}^{j}_{L_n}(\cos x_2,\sin x_2).
\end{align*}
Thus we have proved
\begin{align*}
&\tilde{h}^{j}_{L_n}(\cos\theta_1,\sin\theta_1 \textrm{i})\tilde{h}^{j}_{L_n}(\cos\theta_2,\sin\theta_2)\tilde{h}^{j}_{L_n}(\cos\theta_3,\sin\theta_3\textrm{i})\\
&=\tilde{h}^{j}_{L_n}(\cos a,\sin a)\tilde{h}^{j}_{L_n}(\cos b,\sin
b\textrm{i})\tilde{h}^{j}_{L_n}(\cos x,\sin x).
\end{align*}
Let $\beta_1=a$, $\beta_2=b$ and $\beta_3=x$, we proved (1).

 (2) It follows almost the same manner as the
proof of (1).

We thus completely proved the lemma.
\end{proof}

\subsection{Proof of Lemma \ref{th:10}}\label{sec:14}
Observe that for $\forall j$, $\pi_1(H^j_0)$ and $\pi_1(H^j_1)$
generate a subgroup isomorphic to $SU(2)$. Using the same trick as
we used in proof of Corollary \ref{le:9} and Lemma \ref{le:25} etc.,
it follows $H^j$ is generated by $H^j_0$ and $H^j_1$. By Lemma
\ref{le:32} and Lemma \ref{le:13}, every element $h_0\in H_j$ can be
express as
$$h=\tilde{h}^j_{L_n}(\cos x_1,\sin x_1)\tilde{h}^j_{L_n}(\cos x_2,\sin
x_2 \textrm{i})\tilde{h}^j_{L_n}(\cos x_3,\sin x_3)h_0$$ for some
$x_1,x_2,x_3\in \RR$ and $h_0\in H_{sym}$.

If $\pi_1(h)=I_{m+n}$, we have
\begin{align*}
&\pi_1\bigl(\tilde{h}^j_{L_n}(\cos x_2,\sin x_2 \textrm{i})\bigl)=I_{m+n},\\
&\pi_1\bigl(\tilde{h}^j_{L_n}(\cos x_3,\sin
x_3)\tilde{h}^j_{L_n}(\cos x_1,\sin x_1)\bigl)=I_{m+n}.
\end{align*}
Or
\begin{align*}
&\pi_1\bigl(\tilde{h}^j_{L_n}(\cos x_2,\sin x_2 \textrm{i})\bigl)=\text{diag}(-1_{2n+j},-1_{2n+j+1}),\\
&\pi_1\bigl(\tilde{h}^j_{L_n}(\cos x_3,\sin
x_3)\tilde{h}^j_{L_n}(\cos x_1,\sin
x_1)\bigl)=\text{diag}(-1_{2n+j},-1_{2n+j+1}).
\end{align*}
For the former case, it is clear $h\in H_{sym}$. For the latter one,
we have $\cos x_2=\cos (x_1+x_3)=0$. Using Lemma \ref{pop:2} we have
$$\tilde{h}^j_{L_n}(\cos x_2,\sin x_2 \textrm{i})=\tilde{h}^j_{L_n}(0,\pm \textrm{i})=\tilde{h}^j_{L_n}(0,1).$$ Hence we also get $h\in
H_{sym}$ for this case. We thus have prove the theorem completely.
\begin{corollary}\label{cor:2}
If $m-n=2$, $\ker(\pi_1)=\bigl(\ker(\pi_1)\cap
\tilde{H}_{L_1-L_2}\bigr)\cdot \tilde{H}_0\cdot H_{sym}$.
\end{corollary}
\begin{proof}
The conclusion is clear by Theorem \ref{th:9} and Theorem
\ref{th:10}.
\end{proof}

\subsection{Structure of $\ker(\pi_1)\cap \tilde{H}_{s_0}$}
Let
$L_i=\{(t,a)|(t,a)\in\RR\times\CC^{i},\frac{1}{4}\abs{a}^4+t^2=1\}$.
Let $H=\bigl(\ker(\pi_1)\cap \tilde{H}_{L_1-L_2}\bigr)\cdot
\tilde{H}_0\cdot H_{sym}$. Denote
$$\tilde{w}^j(t,a,b,c)=\tilde{w}_{L_n}(t,\dots,a_{j+1},b_{j+2},c_{j+3},\dots,0)$$ where
$(t,a,b,c)\in L_3$. Let
$\tilde{h}^j(0,a,b)=\tilde{w}^j(0,a,b,0)\tilde{w}^j(t,-\sqrt{2},0,0)$
where $(0,a,b)\in L_2$. Notice if $t=0$, then $\abs{a}=\sqrt{2}$.
Thus the notation set here coincide with what was defined in
Definition \ref{de:1}.

The crucial step in proving Theorem \ref{th:4} is:
\begin{theorem}\label{th:11}
$\ker(\pi_1)\cap \tilde{H}_{s_0}\subseteq H$.
\end{theorem}
The ensuring discussion up to Lemma \ref{le:19} proves the theorem.
We prove a technical lemma at first.
\begin{lemma}\label{le:39}
If
$(t_1,a_1,b_1,0),(t_2,a_2,b_2,0),(t_3,a_3,b_3,0),(t_4,a_4,b_4,0)\in
L_3$, satisfying
$$\pi_1\bigl(\tilde{w}^j(t_1,a_1,b_1,0)\tilde{w}^j(t_2,a_2,b_2,0)\bigl)=\pi_1\bigl(\tilde{w}^j(t_3,a_3,b_3,0)\tilde{w}^j(t_4,a_4,b_4,0)\bigl),$$
\end{lemma}
there exists $h_0\in H$ such that
\begin{align*}
\tilde{w}^j(t_1,a_1,b_1,0)\tilde{w}^j(t_2,a_2,b_2,0)=\tilde{w}^j(t_3,a_3,b_3,0)\tilde{w}^j(t_4,a_4,b_4,0)h_0.
\end{align*}
\begin{proof}
Let
$$w=\tilde{w}_{L_n}(0,0,-1,\dots,1_{j+2},0\dots,0)\cdot
\tilde{w}_{L_n}(0,-1,\dots,1_{j+1},0,\dots,0).$$ By Lemma
\ref{le:27} we have
\begin{align*}
&\tilde{w}^j(t_1,a_1,b_1,0)\tilde{w}^j(t_2,a_2,b_2,0)\bigl(\tilde{w}^j(t_3,a_3,b_3,0)\tilde{w}^j(t_4,a_4,b_4,0)\bigl)^{-1}\\
&=w\tilde{w}^j(t_1,a_1,b_1,0)\tilde{w}^j(t_2,a_2,b_2,0)w^{-1}\\
&\cdot w\bigl(\tilde{w}^j(t_3,a_3,b_3,0)\tilde{w}^j(t_4,a_4,b_4,0)\bigl)^{-1}w^{-1}\\
&=\tilde{w}^1(t_1,a_1,b_1,0)\tilde{w}^1(t_2,a_2,b_2,0)\\
&\cdot
\bigl(\tilde{w}^1(t_3,a_3,b_3,0)\tilde{w}^1(t_4,a_4,b_4,0)\bigl)^{-1}.
\end{align*}
Using Corollary \ref{cor:2} we get the conclusion.
\end{proof}
Now we consider the quotient group $\tilde{W}_{L_n}(t,a)/H$ where
$(t,a)\in L$ until Lemma \ref{le:19}. We continue to write
$\tilde{h}^j_{L_n}(0,a,b)\bigl((0,a,b)\in L\bigr)$ and
$\tilde{w}^j_{L_n}(t,a)\bigl((t,a)\in L\bigl)$ for their images in
$\tilde{W}_{L_n}(t,a)/H$ without confusion.
\begin{lemma}\label{le:33}
For $\forall(0,a,b,c)\in L_3$, $\forall(0,u,v)\in L_2$, we can find
$(t_1,a_1,b_1,c_1)$, $(t_2,a_2,b_2,c_2)\in L_3$, $(t_1,u_1,v_1)$,
$(t_2,u_2,v_2)\in L_2$ such that
\begin{align*}
&(1)\tilde{w}^j(0,a,b,c)\tilde{w}^j(0,0,u,v)=\tilde{w}^j(t_1,a_1,b_1,c_1)\tilde{w}^j(-t_1,u_1,v_1,0),\\
&(2)\tilde{w}^j(0,a,b,c)\tilde{w}^j(0,u,v,0)=\tilde{w}^j(t_2,a_2,b_2,c_2)\tilde{w}^j(-t_2,0,u_2,v_2).
\end{align*}
\end{lemma}
\begin{proof}
(1) If $(a,b,c)$ and $(0,u,v)$ are collinear, then
$(a,b,c)=\pm(0,u,v)$. We can assume $(a,b,c)=(0,-u,-v)$ by Lemma
\ref{pop:2}. Let $t_1=0$, $(a_1,b_1,c_1)=(-u_1,-v_1,0)$ we get the
conclusion.

If they are not collinear, there exists $(u_1',v_1')\in S^1_\CC$
such that $(u_1',v_1',0)$ is on the plane generated by $(a,b,c)$ and
$(0,u,v)$. Let $h$ be an element of the subgroup generated by $H^j$
and $H^{j+1}$ such that $\pi_1(h)$ maps $(a,b,c)$ and $(0,u,v)$ to
$xy$-plane in $\CC^3$ with last coordinate $0$ and
$\pi_1(h)\cdot(u_1',v_1',0)=(0,1,0)$. Let
$(u_1,v_1)=(\sqrt[4]{4-4t^2}u_1',\sqrt[4]{4-4t^2}v_1')$ where $0\leq
t\leq 1$.

Let $f:(L_2\cap\RR^3)\times S^0_\CC\rightarrow S^1_{\CC}$ defined as
follows:
$$f(t,t_1,t_2,g)=\bigl((a+t_1^2)\overline{a},-t_1t_2ga)$$
where $a=-\sqrt{1-t^2}+t \textrm{i}$. It is clear that $f$ is
surjective.

Since
$\pi_1\bigl(\tilde{w}^j(t,t_1,t_2g,0)\bigl)\cdot\pi_1\bigl(\tilde{w}^j(-t,\overline{\pi_1(h)}\cdot(u_1,v_1,0)\bigl)$
is given by the following matrix
$$ \begin{pmatrix}(a+t_1^2)\overline{a} & -t_1t_2ga\\
t_1t_2\overline{ga} & (\overline{a}+t_1^2)a\\
\end{pmatrix},$$
if we denote $\pi_1(h)=B$, by assumption about $h$, there exists
$(t,t_1,t_2,g)$ such that
\begin{align*}
&\pi_1\bigl(\tilde{w}^j(0,\overline{B}\cdot(a,b,c))\tilde{w}^j(0,\overline{B}\cdot(0,u,v))\bigl)\\
&=\pi_1\bigl(\tilde{w}^j(t,t_1,t_2g,0)\bigl)\pi_1\bigl(\tilde{w}^j(-t,\overline{\pi_1(h)}\cdot(u_1,v_1,0)\bigl).
\end{align*}
Hence by Lemma \ref{le:39}, it follows
$$\tilde{w}^j(t,t_1,t_2g,0)\tilde{w}^j\bigl(-t,\overline{B}\cdot(u_1,v_1,0)\bigl)=\tilde{w}^j\bigl(0,\overline{B}\cdot(a,b,c)\bigl)\tilde{w}^j\bigl(0,\overline{B}\cdot(0,u,v)\bigl).$$
 Let $(a_1,b_1,c_1)=\overline{\pi_1(h)^{-1}}(t_1,t_2g,0)$,  then we have
\begin{align*}
&h\tilde{w}^j(t,a_1,b_1,c_1)\tilde{w}^j(-t,u_1,v_1,0)h^{-1}\\
&=\tilde{w}^j(t,t_1,t_2g,0)\tilde{w}^j\bigl(-t,\overline{B}\cdot(u_1,v_1,0)\bigl)\\
&=\tilde{w}^j\bigl(0,\overline{B}\cdot(a,b,c)\bigl)\tilde{w}^j\bigl(0,\overline{B}\cdot(0,u,v)\bigl)\\
&=h\tilde{w}^j(0,a,b,c)\tilde{w}^j(0,0,u,v)h^{-1}.
\end{align*}
Thus we get
$$\tilde{w}^j(t,a_1,b_1,c_1)\tilde{w}^j(-t,u_1,v_1,0)=\tilde{w}^j(0,a,b,c)\tilde{w}^j(0,0,u,v).$$
Hence we proved (1).

(2) Similar arguments hold for (2).
\end{proof}
Let $H_j^0$ denote the subgroup generated by
$\tilde{w}^j(0,a,b,0)\tilde{w}^j(0,1,1,0)$ where
$\abs{a}=\abs{b}=1$. Observe that $\pi_1(H_j^0)$ is isomorphic to
set of all diagonal matrices in $SU(2)$.

\begin{lemma}\label{le:18}
\begin{align*}
H^{j}H^{j+1}H^{j}=H^{j+1}H^{j}H^{j+1}.
\end{align*}
\end{lemma}
\begin{proof}
Fix $j$. Observe that $\pi_1\bigl(\tilde{h}^j_{L_n}(0,a,b)\bigl)$
where $(0,a,b)\in L_2$ and $\pi_1(H_j^0)$ generate a subgroup
isomorphic to $SU(2)$, by Theorem \ref{th:10}, every element in
$H_{j}$ can be expressed as $\tilde{h}^j_{L_n}(0,a,b)h$ where
$(0,a,b)\in L_2$ and $h\in H_j^0$.

We now prove $$H^{j}H^{j+1}H^{j}\subseteq H^{j+1}H^{j}H^{j+1}.$$

Let $h_1^j,h_3^j\in H^{j}$, $h_2^{j+1}\in H^{j+1}$. By above
analysis, there exist $(0,a_2,b_2)\in L_2$ and $h_2\in H_{j+1}^0$
such that $h_2^{j+1}=\tilde{h}^{j+1}_{L_n}(0,a_2,b_2)h_2$. Let
$\pi_1(h_1^j)=A$.

Keep using Lemma \ref{le:16}, we have
\begin{align*}
h_1^jh_2^{j+1}h_3^j&=h_1^j\tilde{h}^{j+1}_{L_n}(0,a_2,b_2)h_2h_3^j\\
&=h_1^j\tilde{w}^j(0,0,a_2,b_2)\tilde{w}^j(0,0,-1,0)h_2h_3^j\\
&=\tilde{w}^j\bigl(0,\overline{A}\cdot(0,a_2,b_2)\bigl)h_1^j\tilde{w}^j(0,0,-1,0)h_2h_3^j\\
&=\tilde{w}^j\bigl(0,\overline{A}\cdot(0,a_2,b_2)\bigl)h_1^j\tilde{w}^j(0,0,-1,0)(h_3^j)'h_2
\end{align*}
for some $(h_3^j)'\in H^j$ since $h_2H^jh_2^{-1}\subseteq H^j.$
Notice $$h_1^j\tilde{w}^j(0,0,-1,0)(h_3^j)'\tilde{w}^j(0,0,-1,0)\in
H^j,$$ by Theorem \ref{th:10} there exist $(0,a_1,b_1)\in L_2$ and
$h'\in H_j^0$ such that
\begin{align*}
h_1^j\tilde{w}^j(0,0,-1,0)(h_3^j)'\tilde{w}^j(0,0,-1,0)=\tilde{w}^j(0,a_1,b_1,0)h'\tilde{w}^j(0,0,-1,0)
\end{align*}
Continue, we have
\begin{align*}
h_1^jh_2^{j+1}h_3^j&=\tilde{w}^j\bigl(\overline{A}\cdot(0,a_2,b_2)\bigl)h_1^j\tilde{w}^j(0,0,-1,0)(h_3^j)'\tilde{w}^j(0,0,-1,0)\\
&\cdot \tilde{w}^j(0,0,1,0)h_2\\
&=\tilde{w}^j\bigl(\overline{A}\cdot(0,a_2,b_2)\bigl)\tilde{w}^j(0,a_1,b_1,0)h'h_2.
\end{align*}
By Lemma \ref{le:33}, there exist $(t,a,b,c)\in L_3$ and $(t,u,v)\in
L_2$ such that
\begin{align*}
&\tilde{w}^j\bigl(\overline{A}\cdot(0,a_2,b_2)\bigl)\tilde{w}^j(0,a_1,b_1,0)=\tilde{w}^j(t,a,b,c)\tilde{w}^j(-t,0,u,v).
\end{align*}
Hence we have
\begin{align*}
&h_1^jh_2^{j+1}h_3^j=\tilde{w}^j(t,a,b,c)\tilde{w}^j(-t,0,u,v)h'h_2.
\end{align*}
If $b=c=0$, let $t_1=-\frac{1}{2}\abs{a}^2+t \textrm{i}$, by Lemma
\ref{le:16} it follows
\begin{align*}
&h_1^jh_2^{j+1}h_3^j=\tilde{w}^j(t,a,b,c)\tilde{w}^j(-t,0,u,v)h'h_2\\
&=\bigl(\tilde{w}^j(t,a,b,c)\tilde{w}^j(-t,0,u,v)\tilde{w}^j(t,a,b,c)^{-1}\bigl)\tilde{w}^j(t,a,b,c)h'h_2\\
&=\tilde{w}^j(-t,0,t_1u,t_1v)\tilde{w}^j(t,a,0,0)h'h_2\\
 &=\tilde{w}^j(-t,0,t_1u,t_1v)\tilde{w}^j(t,0,a,0)\tilde{w}^j(-t,0,-a,0)\tilde{w}^j(t,a,0,0)h'h_2.
\end{align*}

Let $h^{j+1}_4=\tilde{w}^j(-t,0,t_1u,t_1v)\tilde{w}^j(t,0,a,0)$,
$h^j_5=\tilde{w}^j(-t,0,-a,0)\tilde{w}^j(t,a,0,0)h'$,
$h^{j+1}_6=h_2$, then we have
$$h_1^jh_2^{j+1}h_3^j=h^{j+1}_4h^j_5h^{j+1}_6.$$ Notice
$\pi_1(h^{j+1}_4),\pi_1(h^j_5)\in SU(2)$,  by Lemma \ref{le:39},
$h^{j+1}_4\in H^{j+1}$, $h^j_5\in H^j$. Hence we are though the case
$b=c=0$.

Suppose $b\neq 0$. Let $t_1=\sqrt{\abs{c}^2+\abs{b}^2}$,
$\alpha=t_1^{-1}b$, $\beta=-t_1^{-1}\overline{c}$,
$(x,y)=(a,b\alpha^{-1})$ and $h_1\in H^{j+1}$ such that $\pi_1(h_1)$
given by the following matrix
$$ \begin{pmatrix}\alpha & \beta\\
-\overline{\beta} & \overline{\alpha}\\
\end{pmatrix}.$$
Thus we have $\overline{\pi_1(h_1)}\cdot(x,y,0)=(a,b,c)$. We have
\begin{align*}
h_1^jh_2^{j+1}h_3^j&=\tilde{w}^j(t,a,b,c)\tilde{w}^j(-t,0,u,v)h'h_2\\
&=h_1\tilde{w}^j(t,x,y,0)h_1^{-1}\tilde{w}^j(-t,0,u,v)h'h_2\\
&=h_1\tilde{w}^j(t,x,y,0)\tilde{w}^j(-t,0,\sqrt[4]{4-4t^2},0)\\
&\cdot
\tilde{w}^j(t,0,-\sqrt[4]{4-4t^2},0)h_1^{-1}\tilde{w}^j(-t,0,u,v)h'h_2.
\end{align*}
Notice $h'H^{j+1}(h')^{-1}\subseteq H^{j+1}$, hence if we let
\begin{align*}
&h^{j}=\tilde{w}^j(t,x,y,0)\tilde{w}^j(-t,0,\sqrt[4]{4-4t^2},0)h',\\
&h^{j+1}=(h')^{-1}\tilde{w}^j(t,0,-\sqrt[4]{4-4t^2},0)h_1^{-1}\tilde{w}^j(-t,0,u,v)h'h_2,
\end{align*}
we have
$$h_1^jh_2^{j+1}h_3^j=h_1h^{j}h^{j+1}.$$ Notice $\pi_1(h^{j}),\pi_1(h^{j+1})\in
SU(2)$, by Lemma \ref{le:39}, $h^j\in H^j$ and $h^{j+1}\in H^{j+1}$.
Hence we are though the case $b\neq 0$. For case $c\neq 0$, we
follow exactly the same way as the proof of (1) except changing
$(x,y)=(a,b\alpha^{-1})$ to $(x,y)=(a,-c\overline{\beta^{-1}})$.

Thus we have proved
$$H^{j}H^{j+1}H^{j}\subseteq H^{j+1}H^{j}H^{j+1}.$$

The proof of inverse containment is similar.

\end{proof}

The last tool we need is the following
\begin{lemma}\label{le:19}
\begin{align*}
\tilde{H}_{s_0}=\bigl(\prod_{i=1}^{m-n-1}H^i\bigl)\bigl(\prod_{i=1}^{m-n-2}H^i\bigl),\dots,\bigl(\prod_{i=1}H^i\bigl)
\end{align*}
\end{lemma}

\begin{proof}
The proof is exactly the same as the proof of Lemma \ref{le:9} after
changing $h^j_{L_n}$ with $H^j$.
\end{proof}

\subsection{Proof of Theorem \ref{th:11}}
By Lemma \ref{le:19},  any element of $\Tilde{h}_{s_0}$ can be
written as
\begin{align*}
h=\bigl(\prod_{i=1}^{m-n-1}\tilde{h}^i_{m-n-1}\bigl)
\bigl(\prod_{i=1}^{m-n-2}
\tilde{h}^i_{m-n-2}\bigl),\dots,\tilde{h}^1_1h_0,
\end{align*}
where $\tilde{h}^i_k\in H^i$, $k\leq m-n-1$ and $h_0\in H$. Notice
the the element in the lower right corner of matrix $\pi_1(h)$ is
the same as that of matrix $\pi_1(\tilde{h}^{m-n-1}_{m-n-1})$. If
$\pi_1(h)=I_{m+n}$, we have
$\pi_1(\tilde{h}^{m-n-1}_{m-n-1})=I_{m+n}$, which means
$\tilde{h}^{m-n-1}_{m-n-1}\in H_{sym}$ by Theorem \ref{th:10}. By
induction it follows $h\in H$. Hence we proved the theorem.

\subsection{Proof of Theorem \ref{th:4}}

\begin{proof}
It is a direct conclusion of Theorem \ref{th:9} and Theorem
\ref{th:11}.
\end{proof}

\end{document}